\documentclass{amsart}
\usepackage{amsmath,amsthm,amssymb,amscd}
\usepackage{tikz}
\usepackage[all,cmtip,color]{xy}
\usepackage[german,english]{babel}
\usepackage{mathrsfs}
\usepackage{wrapfig}
\usepackage[all,cmtip]{xy}
\input diagxy  
\usepackage{hyperref}
\usepackage[margin=1.2in]{geometry}

\DeclareMathOperator{\Nb}{\mathbf{N}}

\DeclareMathOperator{\Ch}{\mathsf{Ch}^b}
\DeclareMathOperator{\Chac}{\mathsf{Ch}^b_{ac}}

\DeclareMathOperator{\height}{ht}
\DeclareMathOperator{\dMod}{\mathsf{DMod}}
\DeclareMathOperator{\der}{\text{der}}

\DeclareMathOperator{\Cb}{\mathbf{C}}

\DeclareMathOperator{\Xc}{\mathcal{X}}

\DeclareMathOperator{\Lex}{\mathsf{Lex}}
\DeclareMathOperator{\st}{\mathsf{Cat}_{\infty,st}}

\DeclareMathOperator{\calk}{calk}

\DeclareMathOperator{\tate}{tate}
\DeclareMathOperator{\ind}{ind}
\DeclareMathOperator{\pro}{pro}

\DeclareMathOperator{\Calk}{\mathsf{Calk}}

\DeclareMathOperator{\Loc}{\mathsf{L}}
\DeclareMathOperator{\Comp}{\mathsf{C}}

\DeclareMathOperator{\Max}{Max}
\DeclareMathOperator{\Ho}{Ho}

\newcommand{\Eb}{\mathbb{E}}
\newcommand{\Gb}{\mathbb{G}}

\newcommand{\Vect}{\mathsf{Vect}}
\newcommand{\Index}{\mathsf{Index}}

\DeclareMathOperator{\Grp}{\mathsf{Grp}}
\DeclareMathOperator{\Spaces}{\mathsf{Space}}

\DeclareMathOperator{\Perf}{Perf}

\DeclareMathOperator{\ic}{ic}

\newcommand{\Proa}{\mathsf{Pro}^a}
\newcommand{\Inda}{\mathsf{Ind}^{\textit{a}}}

\newcommand{\Ind}{\mathsf{Ind}}
\DeclareMathOperator*{\colim}{\varinjlim}

\DeclareMathOperator{\Mod}{\textsf{Mod}}
\DeclareMathOperator{\Hom}{\textsf{Hom}}

\DeclareMathOperator{\VB}{VB}

\newcommand{\C}{\mathsf{C}}

\newcommand{\Gr}{Gr}
\newcommand{\Spectra}{\mathsf{Sp}}
\newcommand{\Sf}{\mathbf{S}}
\newcommand{\Tf}{\mathbf{T}}

\DeclareMathOperator{\id}{id}
\DeclareMathOperator{\grp}{\times}

\DeclareMathOperator{\Frac}{Frac}
\DeclareMathOperator{\Sp}{\mathsf{Sp}}

\newcommand{\Pro}{\mathsf{Pro}}

\DeclareMathOperator{\Sb}{\mathbb{S}}

\DeclareMathOperator{\Vv}{\mathcal{V}}
\DeclareMathOperator{\Uu}{\mathcal{U}}

\DeclareMathOperator{\invlim}{\varprojlim}

\DeclareMathOperator{\Fun}{Fun}

\DeclareMathOperator{\op}{op}

\DeclareMathOperator{\Ab}{\mathbb{A}}
\DeclareMathOperator{\Aff}{Aff}

\DeclareMathOperator{\QC}{\mathsf{DQCoh}}

\DeclareMathOperator{\Coh}{Coh}

\DeclareMathOperator{\F}{\mathcal{F}}

\DeclareMathOperator{\Kb}{\mathbb{K}}
\DeclareMathOperator{\image}{Im}

\DeclareMathOperator{\Aut}{Aut}
\DeclareMathOperator{\Map}{Map}
\DeclareMathOperator{\G}{\mathbb{G}}

\DeclareMathOperator{\GL}{GL}

\DeclareMathOperator{\D}{\mathsf{D}}

\DeclareMathOperator{\Kk}{\mathcal{K}}

\DeclareMathOperator{\Gg}{\mathcal{G}}

\DeclareMathOperator{\Ss}{\mathcal{S}}

\DeclareMathOperator{\Spec}{Spec}
\DeclareMathOperator{\nTate}{\mathit{n}\text{-}\mathsf{Tate}}
\DeclareMathOperator{\Tate}{\mathsf{Tate}}
\DeclareMathOperator{\elTate}{\mathsf{Tate}^{\textit{el}}}

\DeclareMathOperator{\End}{\mathsf{End}}

\DeclareMathOperator{\Oo}{\mathcal{O}}
\DeclareMathOperator{\red}{red}

\theoremstyle{plain}
\newtheorem{definition}{Definition}[section]
\newtheorem{theorem}[definition]{Theorem}
\newtheorem{theoremc}[definition]{Theorem-Construction}
\newtheorem{proposition}[definition]{Proposition}

\newtheorem{corollary}[definition]{Corollary}

\newtheorem{lemma}[definition]{Lemma}

\newtheorem*{mainprinc}{Main Principle}

\theoremstyle{remark}
\newtheorem{rmk}[definition]{Remark}
\newtheorem{example}[definition]{Example}
\newtheorem{claim}[definition]{Claim}
\newtheorem{warning}[definition]{Warning}

\newtheorem{afterword}[definition]{Afterword}

\begin{document}
\title[Contou-Carr\`ere symbol and reciprocity]{A Generalized Contou-Carr\`ere Symbol \\and its Reciprocity Laws in Higher Dimensions}

\author{Oliver Braunling}
\address{Department of Mathematics, University of Freiburg, Germany}
\email{oliver.braeunling@math.uni-freiburg.de}

\author{Michael Groechenig}
\address{Department of Mathematics, University of Toronto, Canada}
\email{michael.groechenig@utoronto.ca}

\author{Jesse Wolfson}
\address{Department of Mathematics, University of California - Irvine, USA}
\email{wolfson@uci.edu}

\begin{abstract}
We generalize Contou-Carr\`ere symbols to higher dimensions. To an $(n+1)$-tuple $f_0,\dots,f_n \in A((t_1))\cdots((t_n))^{\times}$, where $A$ denotes an algebra over a field $k$, we associate an element $(f_0,\dots,f_n) \in A^{\times}$, extending the higher tame symbol for $k = A$, and earlier constructions for $n = 1$ by Contou-Carr\`ere, and $n = 2$ by Osipov--Zhu. It is based on the concept of \emph{higher commutators} for central extensions by spectra. Using these tools, we describe the higher Contou-Carr\`ere symbol as a composition of boundary maps in algebraic $K$-theory, and prove a version of Parshin--Kato reciprocity for higher Contou-Carr\`ere symbols.
\end{abstract}

\thanks{O.B.\ was supported by DFG SFB/TR 45 ``Periods, moduli spaces and arithmetic of algebraic varieties'', the Alexander von Humboldt Foundation, and DFG GK1821 \textquotedblleft Cohomological Methods in Geometry\textquotedblright\ . M.G.\ was partially supported by EPRSC Grant No.\ EP/G06170X/1. J.W.\ was partially supported by an NSF Graduate Research Fellowship under Grant No.\ DGE-0824162, by an NSF Research Training Group in the Mathematical Sciences under Grant No.\ DMS-0636646, and by an NSF Post-doctoral Research Fellowship under Grant No.\ DMS-1400349. J.W. was a guest of K. Saito at IPMU while this paper was being completed. Our research was supported in part by NSF Grant No.\ DMS-1303100 and EPSRC Mathematics Platform grant EP/I019111/1.}

\maketitle


\section{Introduction}

This article concerns a higher-dimensional generalization of the
Contou-Carr\`{e}re symbol \cite{MR1272340}. The original symbol plays a key
role in the local theory of generalized Jacobians for a relative curve, as
developed by Contou-Carr\`{e}re \cite{MR552212}, \cite{MR1106897}. This theory
was inspired by a conjectural picture due to Grothendieck \cite{gsletter}. If
the relative curve is just a plain curve over a field, the symbol specializes
to the tame symbol. We review this in detail along with an explicit definition
below in \S \ref{sect_BackToSymbol}. But in general the Contou-Carr\`{e}re
symbol is far richer. For example, one recovers the residue symbol in its
tangent space. This aspect cannot be seen in the tame symbol.

If $G:\mathsf{Rings}\rightarrow\mathsf{Groups}$ is a group functor, one
defines its \emph{(formal) loop group} $LG$ as the group functor%
\begin{equation}
LG(A):=G\left(  \left.  A((T))\right.  \right)  \text{,}\qquad\text{where}%
\qquad A((T)):=A[[T]][T^{-1}]\text{.}\label{lmtt1}%
\end{equation}
The classical Contou-Carr\`{e}re symbol is a non-degenerate pairing of loop
groups%
\[
L\mathbb{G}_{m}\times L\mathbb{G}_{m}\longrightarrow\mathbb{G}_{m}\text{,}%
\]
which can also be seen as the statement that $L\mathbb{G}_{m}$ is self-dual
under Cartier duality. Our generalized symbol will be $(n+1)$-multilinear on
$n$-fold loops%
\[
\underset{n+1\text{ factors}}{\underbrace{L^{n}\mathbb{G}_{m}\times
\cdots\times L^{n}\mathbb{G}_{m}}}\longrightarrow\mathbb{G}_{m}%
\]
for any $n\geq1$. This might at first sight not look like an appropriate
generalization of a duality, but we shall explain below both why the
generalization should have this form, as well as our approach for definining
it. See Theorem \ref{intronew_thm1} if you want to jump ahead to a precise formulation of the
properties of our symbol (including compatibility with the classical Contou-Carr\`ere symbol and with a previous construction of a two-dimensional Contou-Carr\`{e}re symbol in \cite{Osipov:2013fk}), or jump to Theorem \ref{thm:second} for the reciprocity law
which we prove for it, generalizing the reciprocity law of the
Contou-Carr\`{e}re symbol on curves. Even when speaking of the
classical Contou-Carr\`{e}re symbol, the literature approaches the topic from
various angles and we use this introduction as an opportunity to explain the
relations between these viewpoints. This is also vital to explain the idea
behind our construction in arbitrary dimension.

\subsection{The origins}

Let us first review the classical story before Contou-Carr\`{e}re's theory.
Suppose (for simplicity) that $X/k$ is a smooth curve over an algebraically
closed field $k$, not necessarily proper. The curve comes equipped with a
generalized Jacobian $J$ along with an Abel--Jacobi map%
\[
X\longrightarrow J
\]
sending a closed point $x$ to the degree one\footnote{Other people prefer to
fix an auxiliary point $p$ and take differences $[x]-[p]$ so that one obtains
a degree zero line bundle, living in what is perhaps more classically called
the Jacobian or $\operatorname*{Pic}^{0}$. The dependency on the choice of $p$
makes this less functorial. We use the term \textquotedblleft
Jacobian\textquotedblright\ in a broader sense here.} line bundle
$\mathcal{O}([x])$. There are many ways to formulate geometric class field
theory, but a reasonable summary can be given in terms of the following two principles:

Every morphism $X\rightarrow G$ to a commutative algebraic $k$-group $G$
factors uniquely over $J$. This can be phrased as an isomorphism%
\begin{equation}
\operatorname*{Hom}\nolimits_{k-\text{groups}}(J,G)\cong H^{0}(X,G)\text{,}%
\label{lwwac1}%
\end{equation}
where we consider \textit{fppf} cohomology on the right side. This is
essentially characterizing $J$ as a type of\ Albanese variety\footnote{in a
generalized sense}. Moreover, extensions of $J$ by $G$ correspond to
$G$-torsors,%
\begin{equation}
\operatorname*{Ext}\nolimits_{k-\text{groups}}^{1}(J,G)\cong H^{1}%
(X,G)\text{.}\label{lwwac2}%
\end{equation}
This property provides a link to class field theory: As a special case of it,
one obtains that every abelian finite \'{e}tale covering of $X$ arises as the
pullback of an isogeny of the Jacobian. For example, if $X$ is $\mathbb{P}%
^{1}$ minus at least two points, the Jacobian has a non-trivial torus part and
a pullback of the isogeny $\mathbb{G}_{m}\overset{\cdot n}{\rightarrow
}\mathbb{G}_{m}$ yields a degree $n$ cyclic Kummer extension. The kernel
sequence of this isogeny defines the corresponding extension in
$\operatorname*{Ext}^{1}$.

There is a more precise formulation, where one replaces $J$ by a Jacobian with
respect to a fixed modulus\footnote{It is standard to call this a
\emph{modulus} in this setting, but in this context it is the same thing as an
effective Weil divisor. The Jacobian $J_{\mathfrak{m}}$ classifies line bundles with extra trivializations at the support of the modulus. Sections of such can be understood in terms of certain lattices; a concept we shall soon return to in \S \ref{subsect_intro_Grassm}.} $\mathfrak{m}$ with support in $\overline{X}\setminus
X$, and then one obtains a description of exactly such abelian finite
coverings which are \'{e}tale over $X$ and whose ramification at the boundary
$\overline{X}\setminus X$ is bounded by the multiplicities of $\mathfrak{m}$.

Background can be found in \cite{MR918564}, but our exposition here follows
\cite[Tome 3, Expos\'{e} XVIII]{SGA4} and \cite[Appendix, Deligne's letter,
(e)]{MR1857369}.\footnote{Recently, it has become more popular to re-interpret
geometric class field theory as rank one local systems arising as pullbacks
from the Jacobian. We refrain from using this slight shift of perspective in
this text.}

\subsection{\label{iv_TheRelativeSituation}The relative situation}

Contou-Carr\`{e}re generalized this story to the situation of relative curves,
i.e. the compactified curve $\overline{X}/k$ is replaced by a flat morphism of
finite presentation%
\[
f:\overline{X}\longrightarrow S
\]
such that the fibers are geometrically integral of dimension one and locally
projective over the base and $X:=\overline{X}\setminus D$ is taken to be the
open complement of a relative divisor $D$. The papers \cite{MR552212},
\cite{MR1106897} set up a corresponding theory of a relative generalized
Jacobian attached to $f$, along with a local theory \cite{MR612541}. The
analogue of Equation \eqref{lwwac1} is set up in \cite[Thm. 1.6.6]{MR3148632}.

The present paper also concerns the relative situation, but we should first
explain a few more concepts in a simpler setting.\footnote{For the sake of completeness, we mention that Deligne \cite{MR1114212} has also found the Contou-Carr\`{e}re symbol, albeit in an analytic setting. This extends the overall picture in a different direction and would lead us too far here.}

\subsection{Local symbols and the Contou-Carr\`{e}re symbol}

Returning to the original formulation of class field theory for curves, i.e.
back in the situation $S:=\operatorname*{Spec}k$ with $k=\overline{k}$, one
can also understand abelian finite \'{e}tale coverings with bounded
ramification using a more classical approach based on the id\`{e}le class
group and methods adapted from number theory\footnote{that is: approaches to
the global class field theory of curves which do not rely on the Jacobian
(there are several ways to do this).}.

In terms of the id\`{e}le class group, the choice of a modulus $\mathfrak{m}$
bounding the allowed ramification identifies the deck transformation group
with a so-called ray class group. Such is a quotient of the id\`{e}le class
group. The fact that the global reciprocity map is trivial on the terms which
we quotient out, amounts to a reciprocity law. Since the global reciprocity
map is a product of the local reciprocity maps, the triviality of the global
action means that a suitable linear combination of local terms adds up to
zero. Neglecting a few details, these contributions amount to the so-called
\emph{local symbols}. The formalism of local symbols extends beyond the mere
application in class field theory to all commutative algebraic $k$-groups
$G$.\footnote{This theory has since found a new formulation in terms of
reciprocity sheaves \cite{MR3606996}, \cite{MR3568941} or more broadly motives
with modulus.}

For example, the tame symbol is a local contribution which arises in the
context of Kummer cyclic coverings. These abelian extensions arise as the
pullback along an isogeny $\mathbb{G}_{m}\overset{\cdot n}{\rightarrow
}\mathbb{G}_{m}$, just as mentioned above.

This suggests the existence of a local analogue of the entire story, where the
roles of $X,\overline{X}$ are replaced by%
\[
X:=\operatorname*{Spec}\operatorname*{Frac}\mathcal{O}_{X,x}\qquad
\text{and}\qquad\overline{X}:=\operatorname*{Spec}\mathcal{O}_{X,x}\text{,}%
\]
so that one can think of $X$ as a punctured disc and $\overline{X}$ the
\textquotedblleft compactification\textquotedblright\ obtained by filling the
puncture. This setting would still retain most of the global geometry since
the field of fractions of course determines the curve (so it is not `as local'
as one might wish for). This suggests to work with the formal completions
instead.%
\begin{equation}
X:=\operatorname*{Spec}\operatorname*{Frac}\widehat{\mathcal{O}}_{X,x}%
\qquad\text{and}\qquad\overline{X}:=\operatorname*{Spec}\widehat{\mathcal{O}%
}_{X,x}\text{.}\label{lsimpls3}%
\end{equation}
Of course, one can choose a local coordinate and obtain (non-canonical)
isomorphisms%
\begin{equation}
\operatorname*{Frac}\widehat{\mathcal{O}}_{X,x}\simeq\kappa(x)((t))\qquad
\text{and}\qquad\widehat{\mathcal{O}}_{X,x}\simeq\kappa(x)[[t]]\text{.}%
\label{lsimpls4}%
\end{equation}
The analogy to the loop group construction in Equation \eqref{lmtt1} is apparent.

Before we continue, let us recall that these (formal) local contributions
admit a class field theory in their own right, known as local class field theory.

\subsection{\label{sect_DualGalois}Duality formulation of local class field
theory}

Let us first look at the original local theory originating from arithmetic.
Suppose $F$ is a finite extension of $\mathbb{Q}_{p}$. Local class field
theory can be expressed as a duality in Galois cohomology. The pairing%
\begin{equation}
H^{i}(F,\mathbb{Z}/n(1))\otimes H^{2-i}(F,\mathbb{Z}/n)\longrightarrow
H^{2}(F,\mathbb{Z}/n(1))\cong\mathbb{Z}/n\label{lsmips1}%
\end{equation}
is non-degenerate for any $n\geq1$ and any $i$. Here $\mathbb{Z}/n(1)$ refers
to the Tate twist; one could also write $\mathbf{\mu}_{n}$. Since%
\[
H^{1}(F,\mathbb{Z}/n)=\operatorname*{Hom}(\operatorname*{Gal}%
(F^{\operatorname*{sep}}/F),\mathbb{Z}/n)
\]
is the $\mathbb{Z}/n$-dual of the abelianized Galois group, this encodes the
classification of degree $n$ abelian \'{e}tale coverings of
$\operatorname*{Spec}F$ in terms of $H^{1}(F,\mathbb{Z}/n(1))\cong F^{\times
}/nF^{\times}$.

The same is true if $F$ is a finite extension of $\mathbb{F}_{p}((t))$, except
that a more involved formulation is necessary if $p\mid n$, which we do not
wish to discuss in the introduction (to keep this exposition at reasonable
length\footnote{The story is entirely analogous to what happens in geometric
class field theory, where $\mathbb{G}_{a}$ (or truncated Witt vectors) are
needed as the relevant commutative group scheme, and the pullbacks are
Artin--Schreier--Witt extensions.}).

Let us now discuss a generalization of this which is vital for understanding
the deeper motivations for the present paper.

The above duality formulation of local class field theory can be generalized
to $r$-local fields, e.g., when $F$ is a finite extension of%
\begin{equation}
\mathbb{Q}_{p}((t_{1}))((t_{2}))\ldots((t_{r-1}))\qquad\text{or}%
\qquad\mathbb{F}_{p}((t_{1}))((t_{2}))\ldots((t_{r}))\text{.} \label{lsmips2}%
\end{equation}
There are more $r$-local fields than just these, but again let us sweep this
under the rug for the purpose of this introduction.

A duality formulation of class field theory as in Equation \eqref{lsmips1}
remains intact also in this broader setting, but the cohomological dimension
increases from $2$ to $r+1$. We get, again ignoring the case where the
characteristic divides $n$, a non-degenerate pairing%
\begin{equation}
H^{i}(F,\mathbb{Z}/n(i))\otimes H^{r+1-i}(F,\mathbb{Z}/n)\longrightarrow
H^{r+1}(F,\mathbb{Z}/n(i))\cong\mathbb{Z}/n\text{.} \label{lsmips3}%
\end{equation}
Letting $i:=r$, this now pairs the $\mathbb{Z}/n$-dual of the abelianized Galois
group with the cohomology group%
\begin{equation}
H^{r}(F,\mathbb{Z}/n(r))\cong K_{r}^{M}(F)/n\text{,} \label{lsmipse1}%
\end{equation}
where the map is the norm residue isomorphism. We observe two key facts: (1)
as the cohomological dimension increases, the duality moves to higher
homological degrees, and (2) the role of $\mathbb{G}_{m}$ in $1$-dimensional
class field theory is now taken over by a Milnor $K$-group (or the motivic
sheaf $\mathbb{Z}(r)$, but let us stay entirely in the language of $K$-theory;
see Remark \ref{rmk_RelationToMotivicCohomology}).

\subsection{Back to the Contou-Carr\`{e}re symbol\label{sect_BackToSymbol}}

The duality considerations in \S \ref{sect_DualGalois} were only on the level
of Galois cohomology, or the \'{e}tale topos if you will. They are not
geometric. Despite the formal similarity to Poinc\'{a}re duality, the
underlying scheme is just $\operatorname*{Spec}F$ and the duality a
group-theoretic fact of $\operatorname*{Gal}(F^{\operatorname*{sep}}/F)$. One
would expect more, especially when attempting to move this story to the
relative setting of \S \ref{iv_TheRelativeSituation}. This is the motivation
for the original Contou-Carr\`{e}re symbol \cite{MR1272340}.

We return to the situation of a relative curve. The \emph{Contou--Carr\`{e}re
symbol} is a non-degenerate pairing on the loop group of $\mathbb{G}_{m}$.%
\[
(-,-):L\mathbb{G}_{m}\times L\mathbb{G}_{m}\longrightarrow\mathbb{G}%
_{m}\text{.}%
\]
It can be given by an explicit formula. Using a presentation
\begin{equation}
f=\prod_{-\infty}^{i=-1}(1-a_{i}t^{i})a_{0}t^{\nu(f)}\prod_{i=1}^{\infty
}(1-a_{i}t^{i})\text{,}\qquad\text{and}\qquad g=\prod_{-\infty}^{i=-1}%
(1-b_{i}t^{i})b_{0}t^{\nu(g)}\prod_{i=1}^{\infty}(1-b_{i}t^{i})\text{,}%
\label{lemis3}%
\end{equation}
for suitable $\nu(f),\nu(g)\in\mathbb{Z}$ (just the order of the power series) and
$a_{i},b_{i}\in A$ (nilpotent for $i$ negative), the value is given by
\begin{equation}
(f,g)=(-1)^{\nu(f)\nu(g)}\frac{a_{0}^{\nu(g)}\prod_{i=1}^{\infty}\prod
_{j=1}^{\infty}(1-a_{i}^{j/(i,j)}b_{-j}^{i/(i,j)})^{(i,j)}}{b_{0}^{\nu
(f)}\prod_{i=1}^{\infty}\prod_{j=1}^{\infty}(1-a_{-i}^{j/(i,j)}b_{j}%
^{i/(i,j)})^{(i,j)}}\text{.}\label{lemis3e}%
\end{equation}
We can directly connect this to the local class field theory story of
\S \ref{sect_DualGalois}. If we evaluate the Contou--Carr\`{e}re symbol on a
field $k$, it simplifies to%
\begin{align}
L\mathbb{G}_{m}(k)\times L\mathbb{G}_{m}(k) &  \longrightarrow\mathbb{G}%
_{m}(k)\label{lccdeg}\\
k((t))^{\times}\times k((t))^{\times} &  \longrightarrow k^{\times}\nonumber
\end{align}
sending%
\begin{equation}
(f,g)\longmapsto(-1)^{v(f)v(g)}\left(  \frac{f^{v(g)}}{g^{v(f)}}\right)
(0)\text{.}\label{lsi1}%
\end{equation}
Here we exploit that since the fraction in the big brackets has degree zero,
its evaluation at zero is possible and non-zero. This expression is known as
the \emph{tame symbol}. Its relation to local class field theory is as
follows: Taking $F:=k((t))$ for any field $k$ such that $\operatorname*{char}%
(F)\nmid n$, the Galois cohomology pairing
\begin{equation}
H^{1}(F,\mathbb{Z}/n(1))\otimes H^{1}(F,\mathbb{Z}/n(1))\longrightarrow
H^{2}(F,\mathbb{Z}/n(2))\label{lsi1a}%
\end{equation}
can, through the norm residue isomorphism (as in Equation \eqref{lsmipse1}) be
realized as a quotient of the natural pairing in Milnor $K$-theory%
\begin{equation}%
\xymatrix{ K_{1}^M(F) \otimes K_{1}^M(F) \ar[r] \ar[d] & K_{2}^{M}%
(F) \ar[d] \\ H^{1}(F,\mathbb{Z}/n(1)) \otimes H^{1}(F,\mathbb{Z}%
/n(1)) \ar[r] & H^{2}(F,\mathbb{Z}/n(2)), }%
\label{lcioo1}%
\end{equation}
and along with the boundary map%
\begin{equation}
K_{1}^{M}(F)\otimes K_{1}^{M}(F)\longrightarrow K_{2}^{M}(F)\overset{\partial
}{\longrightarrow}k^{\times}\text{,}\label{lwww1f}%
\end{equation}
the composition of maps in the top row is given by the same formula as in
Equation \eqref{lsi1}. This shows that the duality maps which occur in local
class field theory are at least close to the ones realized by the tame symbol;
and thus are reasonable to generalize in some way to the Contou--Carr\`{e}re
symbol. We also get a strong hint of the relevance of $K$-theory here. The
full story is a little more complicated: The pairing in Equation \eqref{lsi1a}
has a different Tate twist than in Equation \eqref{lsmips1}, so it is a little
bit off. Once the field contains a primitive $n$-th root of unity, one can
pick an isomorphism of \'{e}tale sheaves $\mathbb{Z}/n\simeq\mathbb{Z}/n(1)$
($=\mathbf{\mu}_{n}$) to fix this, but really the tame symbol corresponds to
the (prime to the characteristic part of the) Hilbert symbol and not to the
reciprocity pairing. Let us sweep these issues under the rug for the purpose
of this introduction.

The boundary map $\partial$ in Equation \eqref{lwww1f} arises from the
localization sequence in (Quillen) $K$-theory, corresponding to the
open-closed complement decomposition%
\begin{equation}
\operatorname*{Spec}k\underset{\text{closed}}{\hookrightarrow}%
\operatorname*{Spec}\mathcal{O}_{F}\underset{\text{open}}{\hookleftarrow
}\operatorname*{Spec}F\text{.} \label{lwuw1}%
\end{equation}
Here $(\mathcal{O}_{F},\mathfrak{m})$ is the ring of integers in the local
field $F$ and we use that $\mathcal{O}_{F}/\mathfrak{m}=k$. The boundary map
appears in the attached long exact sequence in the spot%
\begin{equation}
\cdots\longrightarrow K_{2}(\mathcal{O}_{F})\longrightarrow K_{2}%
(F)\overset{\partial}{\longrightarrow}K_{1}(k)\longrightarrow\cdots\text{.}
\label{lwuw2}%
\end{equation}
In these low degrees there is no difference between Quillen $K$-theory and
Milnor $K$-theory (see \S \ref{sect_BackgrOnKTheory} for more on this).

It turns out that this description generalizes without any problem to the
$n$-dimensional case of the pairing in Equation \eqref{lsmips3}. This gives rise
to the \emph{higher tame symbol}. Its role in higher-dimensional class field
theory of schemes (as provided by Parshin \cite{MR514485}, \cite{MR752939},
\cite{ParshinAdeleTheory} and Kato \cite{MR550688}, \cite{MR726423},
Kato--Saito \cite{MR862639}) is analogous to the classical tame symbol. Its
reciprocity laws have the same formal shape as reciprocity laws\footnote{they
are customarily also called \emph{residue theorems} in this setting} for
rational $n$-forms in Grothendieck--Serre Duality for coherent sheaves. We
explain the higher tame symbol and the generalization of the boundary map
construction using $\partial$ in \S \ref{sect:CCquick} below.

For the higher tame symbol, one obtains the same object irrespective of
whether one uses Milnor $K$-theory or Quillen $K$-theory. This leads us to a
first idea how one might construct a higher Contou-Carr\`{e}re symbol.
Firstly, it should be concerned with higher formal loop groups, as in%
\begin{equation}
L^{r}G(A):=G\left(  \left.  A((T_{1}))((T_{2}))\ldots((T_{r}))\right.
\right)  \text{,}\label{literloop}%
\end{equation}
which is just the $r$-fold iterate of the loop construction in Equation
\eqref{lmtt1}, and is visibly a good formal model for various
(equicharacteristic) higher local fields, see Equation \eqref{lsmips2}.\medskip

\textbf{Idea 1:} Replace the open-closed complement in Equation \eqref{lwuw1} by%
\[
\operatorname*{Spec}A\underset{\text{closed}}{\hookrightarrow}%
\operatorname*{Spec}A[[T]]\underset{\text{open}}{\hookleftarrow}%
\operatorname*{Spec}A((T))
\]
and attempt to work with the corresponding boundary map $\partial$, imitating
the construction of the higher tame symbol. If $A$ is a field, this should
specialize to the previous situation and thus, by construction, this
generalized symbol would necessarily degenerate to the tame symbol in the
classical situation, analogous to what happened around Equation \eqref{lccdeg}%
-\ref{lsi1}.

Firstly, one should ask whether this recovers the original Contou-Carr\`{e}re
symbol even in the one-dimensional case. This had been suggested by
Kapranov--Vasserot \cite[4.3.7. (Remark)]{MR2332353} and is answered
affirmatively in this paper (see Theorem \ref{thm:ccompare2}), and was shown around the same time
also by Osipov--Zhu \cite{Osipov:2013fk}.

We pursue \textit{Idea 1 }in \S \ref{sect:CCquick}. It leads to one possible
construction of our Contou-Carr\`{e}re symbol in all dimensions (Definition~\ref{def:prelimCCsymbol}); probably the quickest.\footnote{Actually, we do something more general: Using Parshin--Beilinson ad\`{e}les one can run such a construction for arbitrary descending chains of subschemes. The case discussed in this introduction arises as a special case.}

\subsection{Central extensions}

On the other hand, this approach also has a drawback: Going from Equation
\eqref{lsimpls3} to Equation \eqref{lsimpls4} we chose a local coordinate. In
other words, we were using Cohen's Structure Theorem, telling us that an
equicharacteristic discrete valuation field is always isomorphic to a Laurent
series field,%
\begin{equation}
\operatorname*{Frac}\widehat{\mathcal{O}}_{X,x}\simeq\kappa(x)((t))\text{,}%
\label{lemis4}%
\end{equation}
where $\kappa(x)$ is the residue field. This isomorphism is highly
non-canonical. However, of course none of our constructions should depend on
the choice of such a coordinate\footnote{just as class field theory doesn't
depend on choosing a coordinate.}. Translated to the Contou--Carr\`{e}re
symbol, i.e. to abstract loop group functors%
\[
LG(A)=G\left(  \left.  A((t))\right.  \right)  \text{,}%
\]
this suggests that our constructions should really be invariant under all ring
automorphisms of $A((t))$, of which there are many\footnote{and themselves
representable as a group ind-scheme}. This property indeed holds for the
original Contou--Carr\`{e}re symbol, but note that it is not at all obvious
from the complicated formula in Equation \eqref{lemis3e}. This suggests to look
for a definition of the Contou-Carr\`{e}re symbol (as well as its higher
analogues) where this invariance is automatic by construction.

Tate in his famous paper \cite{MR0227171} had asked a related question:
Suppose $X/k$ is a curve. He wanted to define the residue of a Laurent series
at a closed point $x\in X$. While%
\[
f\,\mathrm{d}t\mapsto a_{-1}\qquad\text{for}\qquad f=\sum a_{i}t^{i}\in
k((t))\underset{(\ast)}{\simeq}\operatorname*{Frac}\widehat{\mathcal{O}}_{X,x}%
\]
is a clear candidate for a definition, it suffers from the same problem of
depending on the isomorphism $(\ast)$. Instead, he wanted a construction which
was a priori independent of the choice of a coordinate. This issue can be
connected to the Contou-Carr\`{e}re symbol, since it also encodes the residue:
The formula%
\begin{equation}
(1-\varepsilon f,1-\varepsilon g)\equiv1-\varepsilon^{2}\operatorname*{res}%
\nolimits_{t=0}(g\,\mathrm{d}f)\label{lmtt2}%
\end{equation}
holds for the choice $A:=k[\varepsilon]/(\varepsilon^{3})$, $k$ any field and
regarding $f,g\in k((t))^{\times}$ as a subgroup of $L\mathbb{G}_{m}(k)$,
\cite{MR2036223}, \cite{MR1988970}. Thus, a coordinate-invariant construction
of the classical Contou-Carr\`{e}re symbol includes such a
coordinate-independent approach to the residue. Conversely, our second method
for constructing a higher Contou-Carr\`{e}re symbol goes the reverse
direction: We adapt Tate's solution for residues in \cite{MR0227171}, which we
shall recall in \S \ref{sect_TatespaceIdea} below, to the Contou-Carr\`{e}re
symbol. Tate also showed the residue theorem using his method for curves.
Arbarello--de Concini--Kac have used the same idea to set up the tame symbol
and prove the corresponding reciprocity law on curves \cite{0699.22028} (and
more broadly \cite{MR1936475}). Based on this idea, Anderson--Pablos Romo
\cite{MR2036223} and Beilinson--Bloch--Esnault \cite{MR1988970} had the
insight that the same strategy should both make it possible to construct the
classical Contou-Carr\`{e}re symbol coordinate-independently and prove its
reciprocity law on suitable relative curves.

Next, let us explain how Tate's solution works since this is also the
foundation for our second construction of the higher symbol.

\subsection{Tate spaces\label{sect_TatespaceIdea}}

Let us briefly recall Tate's idea in modern terms: The ingredients for our
local symbols can always be written as an ind-pro limit of finite-dimensional
$k$-vector spaces, e.g.,%
\[
k((t))=\underset{n}{\underrightarrow{\operatorname*{colim}}}\underset
{m}{\underleftarrow{\lim}}\,T^{-n}k[T]/(T^{m})\qquad\text{or}\qquad
\operatorname*{Frac}\widehat{\mathcal{O}}_{X,x}=\underset{n}{\underrightarrow
{\operatorname*{colim}}}\underset{m}{\underleftarrow{\lim}}\,\pi^{-n}%
\widehat{\mathcal{O}}_{X,x}/\mathfrak{m}_{X,x}^{m}%
\]
(for any uniformizer $\pi$). Tate had the ingenious insight that for defining
the residue, one only needs to know these objects as ind-pro
limits\footnote{These have since become known as \emph{Tate vector spaces}.
Alternatively (but equivalently), one can work in the setting of locally
linearly compact topological $k$-vector spaces. However, the latter setting is
problematic to adapt to the relative situation.}. He manages to express the
residue as a certain commutator of endomorphisms of these ind-pro objects.
Since the ind-pro structure on $\operatorname*{Frac}\widehat{\mathcal{O}%
}_{X,x}$ can be given without choosing a coordinate isomorphism (as exhibited
above on the right), this solves the problem. The commutator in turn can be
understood as coming from a central extension of a suitable Lie algebra, i.e.
a Lie $2$-cocycle in $H^{2}(\mathfrak{g},k)$ for a suitable Lie algebra
$\mathfrak{g}$.

The papers \cite{0699.22028}, \cite{MR2036223} now recover the tame symbol by
studying the corresponding central extension of groups, i.e. a group
$2$-cocycle $H^{2}(G,k^{\times})$ for a suitable group. Neglecting various
details, one can visualize this as the Lie correspondence between Lie algebras
and Lie groups. This is also seen in Equation \eqref{lmtt2}, where the residue
is explicitly recovered in a tangent space (to a functor).

In fact, the Lie algebra $\mathfrak{g}$ can be taken to be the endomorphism
Lie algebra, and $G$ to be its group of automorphisms (in each case respecting
the ind-pro structure).

A key point of the present paper will be to explain how this approach is
compatible with the ideas about $K$-theory boundary maps earlier in the
introduction. As Anderson--Pablos Romo \cite{MR2036223} set up both the
classical Contou-Carr\`{e}re symbol as well as the reciprocity law using this
method, this is another promising approach to construct a higher
Contou-Carr\`{e}re symbol. A natural idea is to iterate the ind-pro limits,
corresponding the the iterated loop functor in Equation \eqref{literloop}.%
\begin{align*}
&  A((T_{1}))((T_{2}))\cdots((T_{r}))\\
&  \qquad\qquad=\underset{n_{1}}{\underrightarrow{\operatorname*{colim}}%
}\underset{m_{1}}{\underleftarrow{\lim}}\cdots\underset{n_{r}}%
{\underrightarrow{\operatorname*{colim}}}\underset{m_{r}}{\underleftarrow
{\lim}}\,T_{1}^{-n_{1}}\cdots T_{r}^{-n_{r}}k[T_{1},\ldots,T_{r}%
]/(T_{1}^{m_{1}},\ldots,T_{r}^{m_{r}})\text{.}%
\end{align*}
In order to treat such objects \textquotedblleft by
induction\textquotedblright\ in the number of loops $r$, it is natural to set
up a category of ind-pro objects with respect to an arbitrary input category
so that iterating this categorical construction corresponds to iterating the
loop group construction. These are the so-called \emph{Tate categories},
\cite{MR2872533}, \cite{MR3510209}. One then finds that the correct analogue
of the group $2$-cocycles above are higher group $(r+1)$-cocycles for the
automorphism groups of suitable objects in such iterated Tate categories
(called $r$-Tate categories).\footnote{This also works on the Lie algebra
level. A Lie algebra $(r+1)$-cocycle gives the higher residue symbols of Grothendieck--Serre coherent duality theory; this is
due to Beilinson \cite{MR565095}; see also \cite{MR3855636}.}

The two-dimensional tame symbol and its reciprocity law were set up by Osipov
and Osipov--Zhu \cite{MR2195664}, \cite{MR2833793}. Osipov--Zhu also
constructed a two-dimensional Contou-Carr\`{e}re symbol using this method
\cite{Osipov:2013fk} and showed its reciprocity law on surfaces. They also
showed how the residue symbol for $2$-forms on surfaces is encoded in their
symbol, generalizing Equation \eqref{lmtt2}. In the case of the tame symbol
these recover the Parshin reciprocity law from his approach to global class
field theory.\medskip

\textbf{Idea 2:} Construct a higher Contou-Carr\`{e}re symbol using a
generalized central extension, based on a higher group cocycle of an
automorphism group of an object in an iterated Tate category.\medskip

This will also work and we pursue this in \S \ref{symbols}. In some sense it is
more general since it really only relies on the iterated ind-pro structures.

\subsection{Our approach through homotopy theory\label{sect_HtpyTheoryMethod}}

A central part of this paper is devoted to establishing a clear connection
between these two ideas. To this end, we need to work with Quillen $K$-theory
as a space (or spectrum) and not just the individual $K$-groups. Let us sketch
the main idea.

The boundary maps $\partial$ between $K$-groups which appear in \emph{Idea 1}
really come from maps between spectra\footnote{We shall provide background on spectra in \S \ref{subsub:spectra}.}, e.g., using the boundary map of the
localization sequence on the level of spectra,%
\begin{equation}
\Omega K(F)\overset{\partial}{\longrightarrow}K(k)\label{lwojj1pre}%
\end{equation}
taking $\pi_{1}$ functorially yields the map in Equation \eqref{lwuw2}. Now
truncate the homotopy type of $K(k)$ to its $[0,1]$-type. Since $K_{1}%
(A)=A^{\times}$ (at least for local rings, not in general), and if we for
simplicity ignore $\pi_{0}$ (which is a serious oversimplification), $K(k)$
essentially looks like $B\mathbb{G}_{m}$. So, very roughly speaking, one
almost has a truncation map%
\begin{equation}
	\xymatrix{
\Omega K(F) \ar[r]^{\partial} & K(k) \ar@{..>}[r] &
B\mathbb{G}_{m}
}.\label{lwojj1}
\end{equation}
The dotted arrow does not quite exist because we ignored $\pi_{0}$. Nonetheless, writing
$F$-vector spaces as Tate $k$-vector spaces (i.e. as the aforementioned
ind-pro limits), we obtain a map $K(F)\rightarrow K(\mathsf{Tate}(k))$. Modulo the issues with the dotted map above, there is a factorization%
\begin{equation}%
\xymatrix{
\Omega K(F) \ar[d] \ar[r]^{\partial} & K(k) \ar@{..>}[r] & B\mathbb{G}_{m} \\
\Omega K(\mathsf{Tate}(k)). \ar@{-->}[ur] \ar@{-->}[urr]
}%
\label{lwwq1}%
\end{equation}
Next, in \cite[Theorem 1.4 (2)]{Braunling:2014vn} we showed that%
\begin{equation}
K(\mathsf{Tate}(k))\cong\left(  B\operatorname*{Aut}\nolimits_{\mathsf{Tate}%
(k)}(\text{\textquotedblleft}k((t))\text{\textquotedblright})\right)
^{+}\text{,}\label{lzza5}%
\end{equation}
where the plus superscript refers to the plus construction. This is an
analogue of Quillen's construction of $K$-theory via the plus construction,
i.e., $K(A)\simeq K_{0}(A)\times B\operatorname*{GL}(A)^{+}$, except that no
corrections to deal with $K_{0}$ are needed and instead of $\operatorname*{GL}%
$ we deal with the automorphism group of an object in the Tate
category.\footnote{As we explain in \cite[\S 4]{Braunling:2014vn}, this result
can also be thought of as an algebraic analogue of the Atiyah--J\"{a}nich
theorem in topological $K$-theory.} The outer diagonal arrow in Equation
\eqref{lwwq1} thus also pins down a map%
\[
\Omega\left(  B\operatorname*{Aut}\nolimits_{\mathsf{Tate}(k)}%
(\text{\textquotedblleft}k((t))\text{\textquotedblright})\right)
^{+}\longrightarrow B\mathbb{G}_{m}\text{.}%
\]
By adjunction we can move $\Omega$ to the right, lifting $B\mathbb{G}_{m}$ to
$B^{2}\mathbb{G}_{m}$. However, one of the defining properties of the plus
construction is that it does not affect the homology of a space. Thus, the
above map defines a degree $2$ cohomology class of the classifying space
$B\operatorname*{Aut}\nolimits_{\mathsf{Tate}(k)}($\textquotedblleft%
$k((t))$\textquotedblright$)$, \textit{without} having applied the plus
construction. This is equivalent\footnote{The group cohomology group $H^{j}_{grp}(G,M)$ can equivalently be described as the group of homotopy classes of maps from $BG$ to $B^{j}M$.} to providing a group $2$-cocycle%
\begin{equation}
H_{grp}^{2}(\operatorname*{Aut}\nolimits_{\mathsf{Tate}(k)}%
(\text{\textquotedblleft}k((t))\text{\textquotedblright}),\mathbb{G}%
_{m})\text{.}\label{lweeet1}%
\end{equation}
We have explained this in an oversimplified fashion here, especially our
imprecise handling of $\pi_{0}$ (which is just wrong). Also, we have not been
very precise what categories we work in. Nonetheless, the idea should have
become clear. We shall show that a careful variant of the above idea provides
the connecting link between defining the Contou-Carr\`{e}re symbol either via
\textit{Idea 1} or \textit{Idea 2}.

The above considerations necessitate to work with $K$-theory on the level of
spectra. Moreover, when handling $\pi_{0}$ correctly, the right side in
Equation \eqref{lwojj1} is not just a classifying space, but sits in several
homotopical degrees. Thus, one needs to work with a slight generalization of
the concept of a group extension when wanting to do this right (we shall work
in the context of \emph{spectral extensions}).

The above picture generalizes to explain also the connection between our two
approaches to higher Contou-Carr\`{e}re symbols. Iterated use of the boundary
maps $\partial$ corresponds to an interated use of Tate categories and a
straightforward generalization of Diagram \ref{lwwq1}. Moreover, all these
constructions (including Equation \eqref{lzza5}, \cite[Theorem 1.4
(2)]{Braunling:2014vn}) work for arbitrary rings $A$ and thus for $A((T))$ and
not just $k((t))$, making it possible to use it also in a relative setting.

This leads to our main construction, in the spirit of \textit{Idea 2}.

\begin{theoremc}\label{intronew_thm1} Let $k$ be a field and $A$ a $k$-algebra.
\begin{enumerate}
\item For every $n$-Tate object $\mathcal{V}\in\nTate(A)$, we construct a nontrivial
spectral extension of $\operatorname{Aut}(\mathcal{V})$ by the $(n-2)$-shifted
non-connective $K$-theory spectrum $\Sigma^{n-2}\mathbb{K}_A$ (we leave the detailed construction to the main body of the paper).
\item Restricting the latter to the units $A((t_1))\cdots((t_n))^{\times}$, they acquire a spectral extension by the non-connective $K$-theory spectrum $\Sigma^{n-2} \Kb_A$. For $f_0,\dots,f_n \in A((t_1))\cdots((t_n))^{\times}$ we define the \emph{Contou-Carr\`ere symbol} to be the corresponding higher commutator $(f_0,\dots,f_n)$.
\item For $n \leq 2$ the constructions of (1) and (2) recover the definitions of Contou-Carr\`ere and Osipov--Zhu \cite{Osipov:2013fk}.
\end{enumerate}
\end{theoremc}

The first two statements are an immediate consequence of our formalism of spectral extensions and higher commutators (\S \ref{extensions}) applied to the $n$-fold iterate of the index map (see Definition \ref{defi:index}). For the third statement, see Propositions \ref{prop:CC_commutator} and \ref{prop:OZ_CC}.

Denote by $\partial_{i}$ the boundary map in algebraic $K$-theory
, where $K(-,I)$ is $K$-theory with support in the subset given by the ideal $I$:
\[
\partial_{i}\colon K_{i+1}\left(  \,A((t_{1}))\cdots((t_{i}))[[t_{i+1}%
]]\,,\,(t_{i+1})\,\right)  \rightarrow K_{i}\left(  \,A((t_{1}))\cdots
((t_{i-1}))[[t_{i}]]\,,\,(t_{i})\,\right)  \text{.}%
\]
Let $\pi_{\ast}\colon K_{1}(A[[t_{1}]]\,,\,(t_{1}%
))\rightarrow K_{1}(A)$ be the map induced by $A\rightarrow A[[t_{1}]]$, and let
$\det\colon K_{1}(A)\rightarrow A^{\grp}$ be the determinant.

\begin{theorem}\label{intronew_thm2} Let $k$ be a field and $A$ a $k$-algebra. For $f_0,\ldots,f_n\in A((t_1))\cdots((t_n))^{\times}$, we have
            \begin{equation*}
                (f_0,\dots,f_n)^{(-1)^n} = \det  \pi_* \partial_1 \cdots  \partial_n\{f_0,\ldots,f_n\},
            \end{equation*}
where the left-hand side is our Contou-Carr\`ere symbol of Theorem \ref{intronew_thm1}. If $A$ is a field, it agrees with the higher tame symbol of Parshin and Kato.
\end{theorem}

See Theorem \ref{thm:CCboundary}. This theorem shows that \textit{Idea 1} is entirely compatible with our construction following \textit{Idea 2}.

The right-hand side in the above formula is probably the quickest way to define our higher Contou-Carr\`ere symbol. However, our construction following \textit{Idea 2} is more general since Theorem \ref{intronew_thm1} (1) defines an extension of the entire automorphism group, while the above only sees the restriction to the units of multiplication\footnote{The multiplication with any unit of the ring induces an automorphism of the Tate object. But of course there are many more automorphisms. For example, note that the multiplication automorphisms by units only span a commutative subgroup of the entire automorphism group.}. The key point of the above theorem is that it connects our generalization of the ideas around central extensions as in Arbarello--de Concini--Kac \cite{MR1013132} or Anderson--Pablos Romo \cite{MR2036223} with the purely algebraic perspective of boundary maps on the right side.

For concreteness, we now state a special case of these results: The following had been conjectured by Kapranov--Vasserot in \cite{MR2332353}.
\begin{theorem}\label{thm:ccompare2}
    Let $k$ be a field, and let $A$ be a $k$-algebra. The classical Contou-Carr\`{e}re symbol factors through the boundary map in $K$-theory
    \begin{equation*}
        \xymatrix{
            A((t))^{\times} \times A((t))^{\times} \ar[rr]^(.6){(-,-)^{-1}} \ar[d]_{\{-,-\}} && A^\times \\
            K_2(A((t))) \ar[rr]^\partial && K_1(A) \ar[u]_{\det(-)}
        }
    \end{equation*}
    or, in equations, $(f,g)^{-1}=\det(\partial(\{f,g\}))$. Here $(-,-)$ and $\{-,-\}$ refer to the classical commutator and classical Steinberg symbol respectively.
\end{theorem}
A second proof of this case has recently appeared in Osipov--Zhu \cite{Osipov:2013fk}.

\subsection{Grassmannian and determinant bundles}\label{subsect_intro_Grassm}

Previous papers on these subjects have constructed the relevant central
extensions of \S \ref{sect_HtpyTheoryMethod}, especially the group $2$-cocycle
in Equation \eqref{lweeet1}, using different devices. The most popular approach
to this proceeds by constructing the so-called (regularized) determinant line
bundle on the Sato Grassmannian directly. Let us explain this.

Let $k$ be a field as before. Let $\operatorname*{Pic}$ denote the Picard
groupoid of $k$-lines (without grading, for the moment). View $E:=k((t))$ as a
Tate $k$-vector space and let $\mathcal{G}rass$ denote its set of lattices\footnote{That lattices are of relevance for our considerations reflects a corresponding phenomenon, where lattices appear in Contou-Carr\`ere's local theory of Jacobians.}.
Recall that for any finite-dimensional vector space one can define its
determinant as its top exterior power%
\[
\det V={\bigwedge\nolimits^{\operatorname*{top}}V}%
\]
and this generalizes nicely to families. Lattices, being infinite-dimensional
over $k$, do not a priori have such a determinant.\ It would not be clear what
the \textquotedblleft top\textquotedblright\ exterior power should be once
$\dim V=\infty$.

In their approach to the tame symbol cocycle, Arbarello--de Concini--Kac
\cite{MR1013132} considered maps%
\[
\widetilde{\det}:\mathcal{G}rass\longrightarrow\operatorname*{Pic}\text{,}%
\]
associating a line to any lattice. Whenever $L^{\prime}\subseteq L$ for
lattices, they demand%
\begin{equation}
\widetilde{\det}(L)=\widetilde{\det}(L^{\prime})\otimes{\bigwedge
\nolimits^{\operatorname*{top}}}(L/L^{\prime})\tag{$\star$}\label{lv_det}%
\end{equation}
to hold, which makes sense since $L/L^{\prime}$ is finite-dimensional over
$k$. There are several choices of such maps $\widetilde{\det}$, in fact the
set of choices is a $k^{\times}$-torsor. Automorphisms $\operatorname{Aut}(E)$
of $E$ as a Tate vector space do not preserve this choice and rescale the
lines. As a result, $\operatorname{Aut}(E)$ does not act on `the total bundle
space' $\coprod_{L\in\mathcal{G}rass}\widetilde{\det}(L)$, only a central
extension does. This central extension yields a class in%
\[
H_{grp}^{2}(\operatorname{Aut}(E),k^{\times})\text{,}%
\]
giving the so-called unsigned tame symbol, which is like Equation \eqref{lsi1},
but without the sign term. This construction can be adapted to $E:=A((t))$,
i.e. to the relative situation of \S \ref{iv_TheRelativeSituation}. This class
is (except for the correct sign), the same one as the one in Equation
\eqref{lweeet1}. To get the full theory, $\operatorname*{Pic}$ can be upgraded
to be the Picard groupoid of graded lines $\operatorname*{Pic}^{\mathbb{Z}}$.
The corresponding cocycle then yields the full classical Contou-Carr\`{e}re
symbol, as was shown by an explicit computation in \cite{MR2036223}, \cite{MR1988970}.

Essentially, the above is an explicit construction of our homotopical approach
in \S \ref{sect_HtpyTheoryMethod}. It sets up the same cocycle using a group
action on the Grassmannian instead of a purely homotopical consideration.

We can also explain our higher Contou-Carr\`{e}re symbol in terms analogous to
the above, an \emph{Idea 3} if you will:\newline\textbf{(1)} The Tate vector
space $E$ is generalized to an $n$-Tate object. By the correspondence between
$1$-Tate objects of finite-dimensional $k$-vector spaces and locally linearly
compact $k$-vector spaces, this is equivalent to older literature when it
refers to similar constructions in terms of linearly compact vector
spaces.\newline\textbf{(2)} The group $\operatorname{Aut}(E)$ is taken to be
automorphisms in the category of $n$-Tate objects.\newline\textbf{(3)} The map
$\widetilde{\det}$ is trickier to generalize. We replace the lattice
Grassmannian $\mathcal{G}rass$ by a generalized flag space%
\[
L_{0}\hookrightarrow L_{1}\hookrightarrow\cdots\hookrightarrow L_{n}%
\hookrightarrow E
\]
of nested lattices $L_{i}$ in the $n$-Tate object $E$. We implement an
unpublished idea of Kapranov: We generalize $\widetilde{\det}$ to a map taking
values in $K$-theory, without any truncation, and since Waldhausen's explicit
$S_{\bullet}$-model for the $K$-theory of $k$ is a simplicial set with
simplices $0\hookrightarrow X_{1}\hookrightarrow\cdots\hookrightarrow
X_{n}\text{,}$ where the $X_{i}$ are finite-dimensional $k$-vector spaces, we
may define a map%
\begin{equation}
\left[  L_{0}\hookrightarrow L_{1}\hookrightarrow\cdots\hookrightarrow
L_{n}\hookrightarrow E\right]  \qquad\mapsto\qquad\left[  0\hookrightarrow
L_{1}/L_{0}\hookrightarrow\cdots\hookrightarrow L_{n}/L_{0}\right]
,\label{e_c6}%
\end{equation}
sending flags of lattices to simplices in the $K$-theory space. The special
case of just two lattices, $\left[  L^{\prime}\hookrightarrow L\hookrightarrow
E\right]  \mapsto\left[  0\hookrightarrow L/L^{\prime}\right]  $, should ring
a bell in view of Equation \eqref{lv_det}. We have worked out the simplicial
details of this in our previous paper \cite{Braunling:2014vn}, and use these
ideas here.

The role of the Picard groupoids $\operatorname*{Pic}$ or $\operatorname*{Pic}%
^{\mathbb{Z}}$ is seen as follows: Deligne had the insight that there is an
equivalence of homotopy categories%
\begin{equation}
\text{stable homotopy }[0,1]\text{-types}\qquad\Leftrightarrow\qquad
\text{Picard groupoids.}\label{lawo1}%
\end{equation}
This means that spectra whose homotopy groups vanish outside degrees $0$ and
$1$ can equivalently be modelled by Picard groupoids.\footnote{We shall elaborate a little on this and related facts in \S \ref{subsub:spectra}.} Thus, our homotopical
considerations in \S \ref{sect_HtpyTheoryMethod} can also be studied using Picard groupoids, at least once we truncate to homotopical degrees $0$ and $1$. When one studies the classical Contou-Carr\`{e}re symbol, it is
(cum grano salis) almost sufficient to work in such low degrees. Then we use
Deligne's insight that $\operatorname*{Pic}^{\mathbb{Z}}$ receives a map from
the truncated $K$-theory spectrum $\tau_{\leq1}{K}$: The Picard groupoid
$\operatorname*{Pic}^{\mathbb{Z}}$ is a simplified model for the homotopy type
of the $1$-truncation of the Quillen $K$-theory spectrum. This is, by the way,
just a different way of expressing how we found $B\mathbb{G}_{m}$ around
Equations \eqref{lwojj1pre}-\eqref{lwojj1}; $\operatorname*{Pic}$ is the Picard
groupoid corresponding to $B\mathbb{G}_{m}$ under the equivalence in Equation
\eqref{lawo1}.

This is the deeper reason why the above construction can use
$\operatorname*{Pic}^{\mathbb{Z}}$ and yields equivalent output to what we had
otherwise set up in \S \ref{sect_HtpyTheoryMethod} using homotopy types. The
need to work with graded lines is the same complication which we had around
$\pi_{0}$ in \S \ref{sect_HtpyTheoryMethod}.

This discussion also reveals that for higher Contou-Carr\`{e}re symbols, where
higher homotopical degrees are needed, one would have to work in more
complicated models than stable $[0,1]$-types.\newline\textbf{(4)} Cocycles
$H^{2}(G,A)$ are classically modelled through commutators. We phrase this as a
shuffle product, which generalizes easily to higher degrees. Based on this, we
define a concept of \emph{higher commutators} in \S \ref{extensions}.
Finally, we interpret all of these constructions consistently through homotopy
theory. The group of central extensions $H^{2}(G,A)$ equals the group of
homotopy classes of maps (of unpointed spaces) from the classifying space $BG$
to $B^{2}A$. To gain additional flexibility, we define a notion of
\textit{spectral extension}. It amounts to maps $\Sigma_{+}^{\infty}BG$ to
$\Sigma^{2}\mathbb{E}$ for a spectrum $\mathbb{E}$ (where $\Sigma_{+}^{\infty
}X$ denotes the infinite suspension of a (unpointed) space $X$ with a disjoint
basepoint added). This turns out to be the appropriate language to generalize
the Contou-Carr\`{e}re symbol. See \S \ref{extensions} for details.

\begin{mainprinc}
Ideas 1, 2 and 3 all yield the same concept of a higher Contou-Carr\`{e}re symbol.
\end{mainprinc}

The compatibility of \textit{Idea 1} and \textit{Idea 2} is Theorem \ref{intronew_thm2} and the compatibility to \textit{Idea 3} is part of Theorem \ref{intronew_thm1} (3).


\subsection{Higher reciprocity laws}
Our next result is a type of adelic reciprocity law: let $X$ be a reduced, separated $k$-scheme of finite type and dimension $n$. Fix an integer $0 \leq i \leq n$. Let $\zeta$ denote a flag of integral closed subschemes
$$\zeta=(Z_n \supset Z_{n-1} \supset \cdots\supset Z_{i+1}\supset Z_{i-1}\supset \cdots \supset Z_0),$$
indexed by $j \neq i$, with $\dim Z_j = j$. If $i = 0$, we assume that $Z_1$ is proper over $k$. Exactly one dimension is missing, namely $Z_i$; such flags are called \emph{almost saturated}. We denote by $A_{X,\zeta}$ a certain ring formed as an iterated completion of $A(X)$, the $A$-valued rational functions on $X$, at the places $Z_j\times_k \Spec(A)$, cf. \S \ref{subsub:reminder_adeles}. As for the classical ad\`eles, the ring $A_{X,\zeta}$ is built from rings $A_{X,\zeta_Z}$, one for each $i$-dimensional closed subset
\begin{equation*}
    Z_{i-1} \subset Z \subset Z_{i+1}.
\end{equation*}
Each of these rings carries a higher Contou-Carr\`ere symbol $(f_0,\dots,f_n)_{\xi_Z}$, and the geometry of $A_{X,\zeta}$ gives rise to a relation satisfied by these symbols:

\begin{theorem}\label{thm:second}
    For $f_0,\dots,f_n \in A_{X,\zeta}^{\times}$ the product of the Contou-Carr\`ere symbols over all $Z_{i-1} \subset Z \subset Z_{i+1}$ is well defined, and we have
    \begin{equation*}
        \prod_Z(f_0,\dots,f_n)_{\zeta_Z} = 1.
    \end{equation*}
\end{theorem}

See Theorem \ref{thm:classicalCC}. This theorem extends results for $X$ of dimension one by Anderson--Pablos Romo \cite{MR2036223} and P\'al \cite{MR2722779} (for $A$ 0-dimensional), Beilinson--Bloch--Esnault \cite{MR1988970} (for $A$ arbitrary), and results for $X$ of dimension 2 by Osipov--Zhu \cite{Osipov:2013fk}.

The finite dimensionality of the cohomology of a proper curve provides a key geometric input in proving the reciprocity law for 1-dimensional symbols. In the setting of higher dimensional reciprocity laws, we can morally interpret the ring $A_{X,\zeta}$ of Theorem \ref{thm:second} as the ring of $A$-valued rational functions of an exotic ``curve'' $X_{\zeta}$ associated to the almost saturated flag $\{Z_j\}_{j\neq i}$. In principle, this ``curve'' should be obtained by iteratively completing $X$ at the $Z_j$ and then removing the special point $Z_j$. However, at present, the theories of Berkovich or rigid analytic spaces are insufficient to handle such constructions.  Rather than develop such a theory, we take a non-commutative geometry approach and replace $X$ by its stable $\infty$-category of perfect complexes. The operations of localization and completion of schemes have analogues for stable $\infty$-categories, cf. Thomason--Trobaugh \cite{MR1106918} (localization) and Efimov \cite{Efimov:2010fk} (completion). We apply these in \S \ref{reciprocity} to construct a stable $\infty$-category which plays the role of ``$\Perf(X_{\zeta})$'' and we use the (non-commutative) ``geometry'' of this stable $\infty$-category to deduce the reciprocity law.

These categorical constructions could be pictured as a ``non-commutative shadow" of the formal scheme obtained by formal completion. Their role should be understood to be analogous to the one of the ``commutative shadows" utilized by Contou-Carr\`ere (and called \textit{ombres} in \cite{MR1272340}, \cite{MR3148632}).


For our proof of reciprocity, we adopt a general strategy which was first introduced by Gillet \cite{GilletThesis}\footnote{We thank the first anonymous referee for bringing this to our attention.}. The reciprocity law of Theorem \ref{thm:second} expresses information about the local geometry of a variety around an almost saturated flag.  As remarked above, our approach to higher symbols allows us to reduce the reciprocity law to the statement that $d^2=0$ in a Gersten-style complex.  As with the classical Gersten complex, the differentials arise as (sums of) boundary maps in $K$-theory localization sequences.  Our work on derived completion allows us to obtain these localization sequences in our setting and deduce reciprocity.

We now explain the strategy of this proof in the case $A=k$ and $n=2$. Let $Y$ be a smooth surface over $k$ and $x\in Y$ a closed point.  For a triple of non-zero elements $f,g,h$ of the fraction field of ${\Oo}_{Y,y}$ we must show that the product
$$\prod_{C} (f,g,h)_{C,x}$$
ranging over curves containing $x$, is well-defined and equals $1$. There exists a closed subset $Z \subset Y$, such that $Z$ is a union of curves containing $x$, and $f,g,h$ are regular elements on $U = Y \setminus Z$.  Our results above identify this product with a composition of boundary maps as in the lower path of the diagram

\[
\xymatrix{
K_3(U) \ar[r]^-{\partial} & K_2(Y\setminus \{x\},Z\setminus \{x\}) \ar[r] \ar[rd]_{\partial} & K_2(Y \setminus \{x\}) \ar[d]^{\partial} \\
& & K_1(\{x\}).
}
\]
However, this is also equivalent to the upper path of the diagram, the last two maps of which are successive maps in a long exact sequence.

For dimension $n>2$, we employ an analogous argument.  However, we must now replace the punctured surface $\Spec{\Oo}_{Y,x}-\{x\}$ with a more exotic object obtained by completing and removing at all the closed subsets in an almost saturated flag.  Our treatment of derived completions supplies us with the necessary localization sequences in this setting, while our treatment of symbols allows us to identify the appropriate product with a composition of boundary maps from these sequences.  It is then a relatively straightforward matter to show that this composition is zero when restricted to tuples of invertible elements of $A_{X,\zeta}$.


\subsection*{Acknowledgements}
We would like to thank T. Hausel for supporting a visit of the first and the third author to EPF Lausanne, where part of this work was carried out. We would like to thank A. Beilinson and V. Drinfeld for supporting a visit of the first and second author to the University of Chicago, where this paper was completed. The anonymous referees deserve our gratitude for their detailed reviews of our paper which led to major improvements in exposition.


\section{K-theory\label{sect_BackgrOnKTheory}}

\subsection{Background on the flavours of $K$-theory}

We shall use $K$-theory in various flavours, so let us quickly recall the key
players and motivate how and why they enter our considerations.

\subsubsection{Origins}
Historically inspired by the study of vector bundles in algebraic geometry, one can form
for any (small) exact category $\mathsf{C}$ the $K_{0}$-group%
\[
K_{0}(\mathsf{C})=\frac{\left\{  \text{iso-classes }[X]\text{ of objects }%
X\in\mathsf{C}\right\}  }{\text{relations }[X]=[X^{\prime}]+[X^{\prime\prime
}]\text{ for any exact sequence }X^{\prime}\hookrightarrow X\twoheadrightarrow
X^{\prime\prime}}\text{.}%
\]
Choosing $\mathsf{C}$ to be the exact category of vector bundles
$\operatorname*{VB}(X)$ on a variety $X$, this provided the necessary context for Grothendieck's extension of the classical Riemann--Roch theorem.  

The freedom to develop the whole theory for very general categories instead of
just vector bundles has proven very useful and will also be vital for our considerations.

\subsubsection{Localization (geometry)}
Returning to vector bundles, studying the relationship of the $K_{0}$-group
for a scheme $X$ in comparison to the one of a reduced closed subscheme
$Z\subseteq X$ and its open complement $U=X-Z$ leads to ``higher'' $K$-groups fitting together into  the so-called \emph{localization sequence}. In this geometric setting (and only if
everything is smooth), it takes the form of an exact sequence%
\begin{equation}
\cdots\longrightarrow K_{n}(Z)\longrightarrow K_{n}(X)\longrightarrow
K_{n}(U)\longrightarrow K_{n-1}(Z)\longrightarrow\cdots\text{.} \label{leex1}%
\end{equation}
In fact, this long exact sequence can be understood in terms of different
categories. For example, still assuming everything to be smooth, one gets the
relevant $K_{0}$-groups by taking the category of coherent sheaves on $Z,X$
and $U$ respectively, and obtains%
\[
\operatorname*{Coh}(U)=\operatorname*{Coh}(X)/\operatorname*{Coh}%
\nolimits_{Z}(X)\text{,}%
\]
expressing the category of coherent sheaves on the open complement $U$ as the
quotient abelian category of the coherent sheaves on $X$, modulo those having
support in $Z$, called $\operatorname*{Coh}\nolimits_{Z}(X)$. In other words:
The decomposition of $X$ into $Z$ and its complement $U$ can be reflected as a
subcategory and the respective quotient on the level of categories. This
suggests a general picture for categories, valid beyond this geometric application:

\subsubsection{Localization (general principles)}
Generalized to arbitrary (say abelian or exact categories) $\mathsf{C}$ and
suitable subcategories $\mathsf{C}^{\prime}\subseteq\mathsf{C}$, the above
picture generalizes to long exact sequences%
\begin{equation}
\cdots\longrightarrow K_{n}(\mathsf{C}^{\prime})\longrightarrow K_{n}%
(\mathsf{C})\longrightarrow K_{n}(\mathsf{C}/\mathsf{C}^{\prime}%
)\longrightarrow K_{n-1}(\mathsf{C}^{\prime})\longrightarrow\cdots\text{.}
\label{leex2}%
\end{equation}
In the hands of Quillen, general algebraic $K$-theory was defined as the
homotopy groups of certain spaces attached to (for example) exact categories,
as in%
\begin{equation}
K_{n}(\mathsf{C}):=\pi_{n}K(\mathsf{C})\text{,} \label{leex3}%
\end{equation}
where $K(\mathsf{C})$ is a pointed space. There are several ways to set up
$K(\mathsf{C})$;\ e.g., as a simplicial set using simplicial homotopy theory
or as a topological spaces using classical homotopy theory. Moreover, there
are different ways to set up these spaces, all leading to the same homotopy
type (e.g., the $Q$- or $S$-construction). These differences are not so
important for the present paper. Background for simplicial homotopy theory can
be found for example in \cite{MR1206474}, \cite{MR0245005} or \cite{MR2840650}.

\subsubsection{Finer points}\label{subsect:finerpointskthy}
To get a really nice picture, the above suggests various improvements:

(1a) As the $K$-groups are defined as the homotopy groups of a space as in
Equation \eqref{leex3}, it is natural to hope that the long exact sequences in
Equation \eqref{leex1} resp. Equation \eqref{leex2} stem from fiber sequences of
pointed spaces. This can indeed be implemented and leads to defining Quillen
$K$-theory as an invariant of certain categories, taking values in pointed
spaces. This path is already taken by Quillen \cite{MR0338129} or Waldhausen
\cite{MR802796}. These two approaches only differ in generality, but yield the
same theory, which in this paper will be called \emph{connective }%
$K$\emph{-theory}.

(1b) Actually, the pointed spaces $K(\mathsf{C})$ of connective $K$-theory are
of a very special type: They come equipped with the structure of an infinite
loop space \cite{MR505692}. While infinite loop spaces can be regarded on the
one hand as pointed spaces with extra structure, they can equivalently be
regarded as connective spectra, i.e. spectra $S$ such that $\pi_{i}S=0$ for
all $i<0$. Thus, modulo switching between equivalent categories, the $K(-)$ in
Equation \eqref{leex3} can alternatively be taken to refer to a
(connective)\ spectrum. Background on spectra can be found for example in \cite[\S 10.9]{MR1269324} (for a survey), or in
\cite[\S \ 1]{MR1695653}, \cite[\S \ 1]{MR1860878} or \cite[\S1.4]{Lurie:ha}for more general treatments.

(2) The sequence in Equation \eqref{leex1} only exists under very restrictive
assumptions, and using Quillen's $K$-theory it is not right-exact at $K_{0}$.
However, this nuisance can be smoothened out and leads to slightly modified
versions of $K$-theory. Nowadays, and also in the present paper, these are all
jointly generalized to the so-called \emph{non-connective }$K$\emph{-theory}
(we recall the details below). A general construction on the level of
arbitrary exact categories is given in \cite{MR2206639}. The cited paper also
proves the compatibility with the previous approaches to resolve this issue
(e.g., the so-called Bass `negative $K$-groups' \cite{MR0249491} or
Thomason--Trobaugh $K$-theory \cite{MR1106918}). Unfortunately, there is no
way to fix the lack of exactness at $K_{0}$ without needing negative
$K$-groups further to the right in the respective sequences as in Equation
\eqref{leex2}. Thus, non-connective $K$-theory cannot really be modelled in
spaces. However, the property to be a spectrum remains intact also for
non-connective $K$-theory. Hence, the natural habitat for non-connective
$K$-theory are spectra. This time, however, they are not necessarily
connective. In particular, it is not necessarily possible to still model this
using infinite loop spaces instead of spectra.

\begin{rmk}
This also explains the names of connective and non-connective
$K$-theory. This use of terminology is also in line with the conventions of
\cite{MR3070515}, which shows that one can also describe the two variants of
$K$-theory in terms of certain universal properties, giving a further
justification to work with both theories in parallel, yet carefully
distinguish between them.
\end{rmk}

(3) Quillen's foundations for connective $K$-theory from \cite{MR0338129}
allow all exact categories as input; and similarly \cite{MR2206639} gives
similar foundations for non-connective $K$-theory. However, wanting a very
general localization sequence as in Equation \eqref{leex2} there is an issue
with the formation of the quotient $\mathsf{C}/\mathsf{C}^{\prime}$. For many
natural choices of exact categories and subcategories this quotient does not
reasonably exist as an exact category. Going beyond this, there are various
interesting categories, for example arising from glueing constructions of
categories, which are of a profoundly more subtle nature than what can be
captured through the formalism of exact categories. To this end, it roughly
speaking makes sense to generalize $K$-theory to accept all stable $\infty
$-categories \cite{Lurie:ha} as input. Abelian categories (or exact
categories) have a natural attached stable $\infty$-category, so that this is
a genuine generalization. This generalization is available for both connective
and non-connective $K$-theory, and as described above, \cite{MR3070515}
describes either in terms of a universal property whose formulation also
necessitates the use of stable $\infty$-categories.

\subsubsection{Milnor $K$-theory}
Finally, we shall also use \emph{Milnor }$K$\emph{-theory}. Classically, this is only
defined for fields, even though the definition can be extended to
local rings \cite{MR2461425}, \cite{MR2551760}. At least for fields $F$, one
just has%
\[
K_{n}^{M}(F):=\frac{T_{\mathbb{Z}}^{\ast}(F^{\times})}{\left\langle
x\otimes(1-x)\text{ for all }x\in F\setminus\{0,1\}\right\rangle }\text{,}%
\]
where $T_{\mathbb{Z}}^{\ast}(M):=\bigoplus_{n\geq0}M^{\otimes_{\mathbb{Z}}%
^{n}}$ denotes the free tensor algebra of an abelian group. Historically, this
was regarded as a candidate definition for higher $K$-groups, but since then
the picture has clarified a lot: In this paper we mostly refer to Milnor
$K$-theory because of the simplicity of its definition, or the natural graded
ring homomorphism%
\[
K_{\ast}^{M}(F)\longrightarrow K_{\ast}(F)\text{,}%
\]
which easily exhibits high degree elements in the connective $K$-theory of fields.

\begin{rmk}
\label{rmk_RelationToMotivicCohomology}The deeper truth however is that the
motivic Atiyah--Hirzebruch spectral sequence starts from motivic cohomology
$H^{n}(F,\mathbb{Z}(m))$ on the $E_{2}$-page and converges to connective
$K$-theory. It satisfies%
\[
H^{n}(F,\mathbb{Z}(n))\cong K_{n}^{M}(F)\text{,}%
\]
so the deeper reason for the similarities between Milnor and connective
$K$-theory (of a field) is just their `proximity' as provided by the motivic
weight filtration on the $K$-theory spectrum, exhibited here through the
spectral sequence. The comparison of this with \'{e}tale K-theory (resp.
\'{e}tale motivic cohomology) also lies at the core behind the compatibility
to Galois cohomology in Diagram \ref{lcioo1}. However, none of this is needed
in the present paper. See \cite{MR2242284} and \cite{MR2181824} for background.
\end{rmk}

\subsection{Axiomatic review of Algebraic $K$-theory}\label{sub:k}

After this review, let us summarize the key statements we shall need in the format most suitable for us. We view algebraic $K$-theory as a machine, which assigns, to an exact category or stable $\infty$-category $\C$, its spectral shadow $\Kb_{\C}$. This machine sends exact functors $\C \to \D$ to maps of spectra $\Kb_{\C} \to \Kb_{\D},$ and preserves exact sequences. We refer the reader to appendix \ref{inftycats} for a brief overview of the theory of (stable and unstable) $\infty$-categories, and to \cite{Lurie:2009vn,Lurie:ha} for detailed references.  

We encourage the reader unfamiliar with stable $\infty$-categories to think of them as a higher homotopical enrichment of triangulated categories. For example, by  \cite[Theorem 1.1.2.14]{Lurie:ha}, the homotopy category $\Ho(\C)$ of any stable $\infty$-category inherits a canonical triangulated structure.  The advantage of working with stable $\infty$-categories is that many standard constructions for triangulated categories become better behaved and more conceptually straightforward in this context. 

Recall that given an $\infty$-category $\C$, we can form an $\infty$-category of ``ind-objects'' $\Ind(\C)$ (with subcategories $\Ind_\kappa(\C)$ for each regular $\kappa$) of ``formal filtered colimits'' of diagrams in $\C$ (or such over diagrams of size at most $\kappa$) (see \cite[\S5.3.5]{Lurie:2009vn}). If $\C$ is a stable $\infty$-category, then so is $\Ind_\kappa(\C)$  \cite[Prop. 1.1.3.6]{Lurie:ha}. Similarly, every $\infty$-category $\C$ admits an idempotent completion $\C\to\C^{\ic}$ (see \cite[\S5.1.4]{Lurie:2009vn}), and if $\C$ is stable, so is $\C^{\ic}$ \cite[Cor. 1.1.3.7]{Lurie:ha}.  Last, just as there is a good notion of exact functors of triangulated categories $\C\to\D$ and of exact sequences 
\[
	\C\to\D\to \D/\C
\]
of such functors, there is a good notion of such for stable $\infty$-categories. In fact, by \cite[Prop. 5.1.5]{MR3070515}, a sequence of stable $\infty$-categories is exact if and only if the induced sequence of homotopy categories is an exact sequence of triangulated categories. In particular, given a fully faithful exact functor of (presentable) stable $\infty$-categories $\C\to\D$, we can form the quotient stable $\infty$-category $\D/\C$, which we should think of as a higher homotopical analogue of the classical Verdier quotient. 

\subsubsection{Connective Algebraic $K$-theory}
The proposition below captures the most important phenomena for the so-called connective $K$-theory of stable $\infty$-categories (cf. \cite{MR3070515}). This is the flavour of $K$-theory which is compatible with Quillen's original definition of algebraic $K$-theory.

In the following we denote by $\Sp_{\geq}$ the stable $\infty$-category of connective spectra. We refer the reader to Subsection \ref{subsub:spectra} for a brief reminder of stable homotopy theory, and for more details to \cite[\S 10.9]{MR1269324} (for a survey),
\cite[\S \ 1]{MR1695653}, \cite[\S \ 1]{MR1860878} or \cite[\S1.4]{Lurie:ha}.

\begin{proposition}\label{prop:char_conn_K}
The functor of connective $K$-theory for stable $\infty$-categories
$$\Kk_-\colon \st \to \Sp_{\geq}$$
satisfies the following properties.
\begin{itemize}

	\item[(1)] If $\C$ is a stable $\infty$-category admitting countable products (or coproducts), then $\Kk_{\C} \cong 0$.
	
	\item[(2)] The inclusion $\C \to \C^{\ic}$ (where $\ic$ denotes idempotent completion) gives rise to a map of connective spectra $\Kk_{\C} \to \Kk_{\C^{\ic}}$, inducing an isomorphism on $\pi_i$ for $i\geq 1$, and a monomorphism on $\pi_0$.
	
	\item[(3)] Let $\C \hookrightarrow \D \twoheadrightarrow \D/\C$ be an exact sequence of stable $\infty$-categories, where we denote the functor $\C \to \D$ by $i$ and $\D \to \D/\C$ by $q$. Then, there is a fibre sequence
	\[
		\xymatrix{
			\Kk_{\C} \ar[r]^i \ar[d] &  \Kk_{\D} \ar[d]^q \\
			0 \ar[r] & \Kk_{\D/\C}
		}
	\]
	in the $\infty$-category $\Sp_{\geq}$ of connective spectra.	
	
\end{itemize}
\end{proposition}

Property (3) is often referred to as \emph{proto-localization} (e.g. by \cite{MR1106918}). The long exact fibration sequence for $\pi_{*}$ yields a long exact sequence in non-negative degrees. The map $\pi_0(\Kk_{\D}) \to \pi_0(\Kk_{\D/\C})$ will not be surjective in general\footnote{This is the same issue which already appears in purely geometric applications and is alluded to in \S \ref{subsect:finerpointskthy}.}. This suggests the existence of negative $K$-groups, obtained by the homotopy groups of a non-connective $K$-theory spectrum. This leads to \emph{non-connective} $K$-theory, whose properties we recall in the following section.

\subsubsection{Non-Connective Algebraic $K$-Theory}

In the work of Blumberg--Gepner--Tabuada, the following properties were shown to be characteristic for non-connective $K$-theory (see \cite[Thm. 9.10]{MR3070515}). In the following we denote by $\Sp$ the stable $\infty$-category of all spectra.

\begin{proposition}\label{prop:char_K}
Non-connective algebraic $K$-theory is a functor
$$\Kb_-\colon \st \to \Sp$$
satisfying the following properties.
\begin{itemize}

	\item[(1)] If $\C$ is a stable $\infty$-category admitting countable products (or countable coproducts), then $\Kb_{\C} \cong 0$.
	
	\item[(2)] The inclusion $\C \to \C^{\ic}$ (where $\ic$ denotes idempotent completion) gives rise to an equivalence of spectra $\Kb_{\C} \xrightarrow{\cong} \Kb_{\C^{\ic}}$.
	
	\item[(3)] Let $\C \hookrightarrow \D \twoheadrightarrow \D/\C$ be an exact sequence of stable $\infty$-categories, where we denote the functor $\C \to \D$ by $i$ and $\D \to \D/\C$ by $q$. Then, there is a a bi-cartesian square
	\[
		\xymatrix{
			\Kb_{\C} \ar[r]^i \ar[d] &  \Kb_{\D} \ar[d]^q \\
			0 \ar[r] & \Kb_{\D/\C}
		}
	\]
	in the stable $\infty$-category $\Sp$ of spectra.	
	
\end{itemize}
\end{proposition}

We say that non-connective $K$-theory $\Kb_-$ \emph{completes} connective $K$-theory $\Kk_-$, referring to the canonical equivalence
$$\Kk_{\C^{\ic}} \cong \tau_{\geq 0} \Kb_{\C}.$$ Following Schlichting \cite{MR2206639}, we see how every connective theory, satisfying the axioms of Proposition \ref{prop:char_conn_K}, induces a non-connective $K$-theory, subject to the properties of Proposition \ref{prop:char_K} (see also \cite{MR3070515}). This requires the \emph{suspension} of a stable $\infty$-category.

\begin{definition}\label{defi:suspension}
We define the \emph{suspension} of a stable $\infty$-category $\C$ as the stable $\infty$-category
$$\Ss_{\kappa}(\C) = \Ind_{\kappa}\C/\C,$$
where $\kappa$ denotes an arbitrary infinite cardinal, and $\Ind_{\kappa}\C$ denotes the stable $\infty$-category of Ind-objects represented by diagrams of size at most $\kappa$. Let $\Calk_{\kappa}(\C)$ denote $\Ss_{\kappa}(\C)^{\ic}$.
\end{definition}

By definition, we have an exact sequence of stable $\infty$-categories
$$\C \hookrightarrow \Ind(\C) \twoheadrightarrow \Ss(\C).$$
Using the fact that $\Ind(\C)$ admits countable coproducts, properties (1) and $(3)$ of Proposition \ref{prop:char_conn_K} imply that
$$\Kk_{\C} \to 0 \to \Kk_{\Ss(\C)}$$
is a fibre sequence of connective spectra. Since $\pi_0(0) = \pi_0(\Ss(\C))$, we know that it is actually a fibre-cofibre sequence of spectra. This allows us to identify $\Kk_{\C}$ with $\Omega\Kk_{\Ss(\C)}$. We define the non-connective completion $\Kb_-$ to be the functor
\begin{equation}\label{eqn:ncK}
\varinjlim \Omega^n\Kk_{\Calk^n(\C)}
\end{equation}
where $\Calk^0(\C):=\C$ and $\Calk^n(\C):=\Calk(\Calk^{n-1}(\C))$ for $n>0$.

\begin{definition}\label{defi:perf}
Let $\C$ be an (idempotent complete) exact category. We have a well-defined dg-category $\Ch(\C)$ of  bounded chain complexes in $\C$. We denote by $\Chac(\C)$ the full subcategory of acyclic complexes. The stable $\infty$-category $\Perf(\C)$ is defined to be the dg-nerve (see \cite[\S 1.3.1]{Lurie:ha}) of the dg-quotient $\Ch(\C)/\Chac(\C)$. Since the latter is a pre-triangulated dg-category (see \cite[\S 2]{MR1667558}), $\Perf(\C)$ is stable.
\end{definition}

The lemma below follows from the discussion in \cite[\S 9.1]{MR3070515} and \cite[\S 6.2]{MR2206639}
\begin{lemma}\label{lemma:exact_comparison}
Let $\C$ be an exact category. The non-connective $K$-theory of $\C$, in the sense of Schlichting \cite{MR2206639}, agrees with the non-connective $K$-theory of the stable $\infty$-category $\Perf(\C)$ in the sense of Blumberg--Gepner--Tabuada \cite{MR3070515}.
\end{lemma}

It will be necessary to compare algebraic $K$-theory with the original category, in order to be able to use it. Heuristically, this is captured by the slogan that $\Kb_{\C}$ is a spectrum, where objects in $\C$ give rise to points, automorphisms of objects give rise to loops, and, for $n\ge 1$, commuting $n$-tuples of automorphisms in $\C$ give rise to elements of $K_n(\C) = \pi_n(\Kb_{\C})$. This intuition is captured by the following observation.

\begin{rmk}\label{rmk:Ccc4}
We denote by $\C^{\times}$ the $(\infty-)$groupoid of objects in $\C$ (i.e. we discard all non-isomorphisms). Recall that every $(\infty-)$groupoid can be viewed as an unpointed space via the geometric realization of its nerve\footnote{We review the nerve, i.e. the ways of regarding a category as a simplicial set or space in \S \ref{sect:reviewnerveabit}.}. There exists a canonical morphism of pointed spaces
$(\C^{\grp})_+ \to \Omega^{\infty}\Kb_{\C},$
and by the adjunction $\Sigma^{\infty} \dashv \Omega^{\infty}$, a morphism of spectra
$\Sigma^{\infty}_+\C^{\grp} \to \Kb_{\C}$, see \cite[\S 1.3, p. 12]{MR802796}.
\end{rmk}

\begin{example}
Under special circumstances the last map of the previous remark can be promoted to an equivalence. Such a phenomenon underlies Equation \eqref{lzza5} along with the fact that the Tate category has vanishing $K_0$-group.
\end{example}

\begin{definition} We shall frequently use the following shorthands:
\begin{enumerate}
\item If $R$ is a ring, we write $\Kb_{R}$ to denote the nonconnective K-theory spectrum of the category of perfect complexes over $R$.
\item Analogously, if $X$ denotes a scheme, we write $\Kb_{X}$ for the nonconnective K-theory spectrum of perfect complexes on $X$.
\item If $X$ has a closed subscheme $Z$, we write $\Kb_{X,Z}$ for the nonconnective K-theory spectrum of the category of perfect complexes on $X$ with support in $Z$. The latter means that they are required to be acyclic over the complement $X-Z$.
\end{enumerate}
\end{definition}

\begin{example}\label{ex:localization_geometrically}
If $X$ is Noetherian (for example) and has the closed subscheme $Z$, then there is an exact sequence relating their stable $\infty$-categories of perfect complexes. Using Propositon \ref{prop:char_K} (3) we obtain the fibre sequence
$$ \Kb_{X,Z} \to \Kb_{X} \to \Kb_{X-Z} $$
of spectra. The induced long exact sequence of the homotopy groups of the $K$-theory spectra is perhaps the most prominent example of the localization sequence.
\end{example}


\section{The CC symbol via boundary maps}\label{sect:CCquick}

In this section we will give a first definition of our higher Contou-Carr\`ere symbol. We follow a generalization of the idea of boundary maps (which we had called \textit{Idea 1} in the introduction).
Instead of working with iterated loop groups, we use localizations-completions along flags of subschemes. Abstractly, these look like iterated loop groups, but our methods avoids choosing a coordinate. The case discussed in the introduction follows as a special case (see Example \ref{ex:stdx}).

\subsection{Flags of closed subschemes}\label{subsub:reminder_adeles}
Let $X$ be a reduced excellent separated scheme of dimension $n$. A \emph{flag} is a sequence
\[
\xi \colon Z_n \supset \cdots \supset Z_0
\]
of integral closed subschemes of pure dimension $\dim Z_i = i$, indexed by a subset of $\{0,1,\ldots,n\}$. If it is indexed by all of $\{0,1,\ldots,n\}$, we call it \emph{saturated}. If exactly one dimension is missing, such flags are called \emph{almost saturated}.

\begin{definition}\label{def:ad1}If $\xi:=(Z_{n}\supset\cdots\supset Z_{0})$ denotes a flag, we abbreviate the Parshin--Beilinson ad\`{e}le
ring by%
\[
F_{X,\xi} := A(\xi,\mathcal{O}_{X})=A(\{(Z_{n} \supset \cdots \supset Z_{0})\}),\mathcal{O}%
_{X})\text{.}%
\]
The notation $A(-,-)$ is as in Beilinson's original paper \cite[\S 2]{MR565095}.
\end{definition}
The definition of these ad\`{e}les is alternatively also given in \cite[Proposition 2.1.1]{MR1138291} or \cite[\S 2.1]{MR3536437}, viewed from different angles.

\begin{example}
This section also covers the case relevant for
\begin{equation}
L^{n}G(A):=G\left(  \left.  A((T_{1}))((T_{2}))\ldots((T_{n}))\right.
\right)
\end{equation}
as discussed in the introduction. Choose $X$ to be affine $n$-space and take a standard flag of coordinate hyperplane subspaces. See also Example \ref{ex:affinespacestd}.
\end{example}\label{ex:stdx}

For a saturated flag, there is a canonical isomorphism of rings
\begin{equation}\label{llc5}
F_{X,\xi}\cong\prod_{i=1}^{r}F_{i}%
\end{equation}
for some finite $r$, and each $F_{i}$ is an $n$-local field. For a proof, see \cite[\S 3]{MR1213064} or
\cite[Theorem 4.2]{MR3536437}.

Whenever $F$ denotes an $n$-local field, this means that it comes with a canonically determined diagram%
\begin{equation}%
\bfig\node x(0,1200)[F]
\node y(0,900)[\mathcal{O}_{1}]
\node z(300,900)[k_1]
\node w(300,600)[\mathcal{O}_{2}]
\node u(600,600)[k_2]
\node v(600,300)[\vdots,]
\arrow/{^{(}->}/[y`x;]
\arrow/{->>}/[y`z;]
\arrow/{^{(}->}/[w`z;]
\arrow/{->>}/[w`u;]
\arrow[v`u;]
\efig
\label{lcg1}%
\end{equation}
where each $\mathcal{O}_{i}$ denotes the rings of integers of the field depicted above it, and each $k_{i}$ denotes the residue field of the ring depicted to its left.

The $K$-groups of the various rings attached to $F$ are related by the localization sequence
\begin{equation}\label{eqbdryhf}
\cdots \to K_i(k) \to K_i(\mathcal{O}) \to K_i(F) \overset{\partial }{\to} K_{i-1}(k) \to \cdots,
\end{equation}%
which can be used inductively for each step in the above downward ladder of residue fields because of an identification of $\Kb_{\mathcal{O},\mathfrak{m}}$ with the $K$-theory of the residue field $\kappa{(\mathfrak{m}})$.\footnote{As the rings are all regular, we can also work with the $K$-theory of coherent sheaves, where devissage applies. Hence, the $K$-theory of coherent sheaves with support in the maximal ideal is equivalent to the $K$-theory of the residue field. This yields the identification.}

The above localization sequence stems from the bi-cartesian square
\[
\xymatrix{
\Kb_{\mathcal{O},\mathfrak{m}} \ar[r] \ar[d] & \Kb_{\mathcal{O}} \ar[d]^{- \otimes_{\mathcal{O}}F} \\
\bullet \ar[r] & \Kb_{F}.
}
\]
In order to obtain this square, apply Example \ref{ex:localization_geometrically} to $\Spec \mathcal{O}$ and the closed subscheme cut out by the unique maximal ideal (the valuation ideal). The open complement is just $\Spec F$. In the present situation it makes no difference whether we use connective or non-connective $K$-theory.

\begin{rmk} We will soon generalize the above by instead using the square in Equation \eqref{f_sq1} below.
\end{rmk}

\begin{definition}\label{defi:higher_tame}
Let $F$ be an $n$-local field.
\begin{enumerate}
\item The \emph{higher tame symbol (in Quillen $K$-theory)} is defined to be the composition
$$\partial^{(1)} \circ \cdots \circ \partial^{(n)}\colon K_{n+1}(F) \to K_1(\kappa) = \kappa^{\times},$$
where  $\kappa$ is the last residue field and $\partial$ refers to the respective boundary maps coming (inductively for each residue field) from the localization sequence in \eqref{eqbdryhf}.
\item In many ways simpler, the \emph{higher tame symbol (in Milnor $K$-theory)} is defined to be the composition
$$\partial^{(1)} \circ \cdots \circ \partial^{(n)}\colon K_{n+1}^{M}(F) \to K_1(\kappa) = \kappa^{\times},$$
where $K^{M}$ refers to Milnor $K$-theory and $\partial$ is the boundary map in an entirely analogous localization sequence in motivic cohomology. See Remark \ref{rmk_RelationToMotivicCohomology} for the relation between Quillen and Milnor $K$-theory.
\end{enumerate}
\end{definition}

The higher tame symbol in Milnor $K$-theory is simpler because Milnor $K$-groups have a generator-relator presentation and the relevant boundary map can alternatively be defined by an explicit formula. See \cite[\S 2]{MR0260844}. In fact, historically this was known before the interpretation as motivic cohomology. Only the latter however shows how closely connected both viewpoints are.

We move on to the relative situation.

\begin{definition}If $X$ is additionally a scheme of finite type over $k$, then for every $k$-algebra $A$ we define
\[
A_{X,\xi} := F_{X,\xi}  \otimes_k  A,
\]
and for saturated flags, we note that the canonical isomorphism of Equation \eqref{llc5} can be promoted to an isomorphism of $k$-algebras, see \cite[Theorem 4.2]{MR3536437}.
\end{definition}

Following Morrow's \cite{morrow}, we give a self-contained construction of $F_{X,\xi}$, in a format which will be particularly useful for us later.

\begin{definition}\label{defi:equiheighted_localization}
An ideal $I \subset R$ of a Noetherian ring $R$ is called \emph{equiheighted} if all minimal prime ideals over $I$ have the same height in $R$. We define the localization of an $R$-module $M$ at $I$, to be
$$M_I = S^{-1}M,\text{ where } S=\{s \in R~|~s\text{ is a non-zero-divisor in }R/I\}.$$
\end{definition}

Geometrically, an equiheighted ideal defines a closed subspace of $\Spec R$, with all irreducible components having the same codimension in $\Spec R$. Although not completely obvious, the two operations introduced below preserve chains of equiheighted ideals \cite[Lemma 7.3]{morrow}.

\begin{definition}\label{defi:comp&loc}
Let $R$ be a Noetherian ring of Krull dimension $n$. For a chain of equiheighted ideals $\xi=(I_k \subset \dots \subset I_0)$, with $\height I_i = n-i$, we define the \emph{completion operation}
$$\Comp(R,\xi) = (\widehat{R}_{I_0},\xi).$$
We denote by
$$\Loc (R,\xi) = (R_{I_1},\xi'),$$
the \emph{localization operation}, where $\xi'$ is the restriction to $R_{I_1}$ of the shifted chain of ideals given by $I'_i = I_{i+1}.$
\end{definition}

\begin{example}
    If $R$ is a Noetherian domain of Krull dimension $1$, then for every prime ideal $\mathfrak{p}$, we can consider the chain $\xi=( 0 \subset \mathfrak{p})$. In this case, we have $(\Loc \circ \Comp) (R,\xi) = \mathrm{Frac}\;\widehat{R}_{\mathfrak{p}}$.
\end{example}

\begin{definition}\label{defi:HLF_complete_localize}
Let $R$ be an excellent reduced ring of Krull dimension $n$. For a chain of radical equiheighted ideas $(0 = I_n \subset I_{n-1} \subset \dots \subset I_0)$, with $\mathrm{ht} I_i = n-i$, we define
$$F_{\Spec R,\xi} \cong (\Loc\circ\Comp)^n(R,\xi).$$
\end{definition}
This definition is compatible with Definition \ref{def:ad1}.


%

\subsection{Boundary maps of a flag}

We begin by giving a precise definition of the completion of a scheme at a closed subscheme. Although this seems fairly straight-forward in the affine case, it is necessary to be finical in general.

\begin{definition}\label{defi:comp}
Let $X$ be a scheme and $\mathcal{I} \subset \Oo_X$ a sheaf of ideals for which the corresponding closed subscheme $Y\subset X$ is affine. We define the \emph{completion of $X$ at $Y$} to be the affine scheme $$\Cb_YX = \Spec \varprojlim_{n \in \mathbb{N}}\Gamma(X,\Oo_X/\mathcal{I}^n).$$
\end{definition}

A related construction is the formal neighbourhood $\widehat{X}_Y$. It is defined to be the direct limit in the sense of formal schemes, of the family of schemes $Y_n = \Spec_X \Oo_X/\mathcal{I}^n$. The completion of Definition \ref{defi:comp} on the other hand is equivalent to the direct limit of the $X$-schemes $Y_n$ in the category of affine schemes.

\begin{warning}
The definition above could lead to pathological situations if $Y$ was not assumed to be affine. For example, if $Y \subset \mathbb{P}^n$ is an embedded projective curve, the inverse limit $\varprojlim_{n \in \mathbb{N}} \Oo_X/\mathcal{I}^n$ in the category of $\Oo_X$-modules is not necessarily quasi-coherent.
\end{warning}

By virtue of Chevalley's theorem \cite{EGA} (or for a recent exposition, see Conrad \cite{MR2356346}), affineness of $Y$ only depends on the underlying reduced subscheme $Y^{\red} \subset X$ (i.e. Chevalley's theorem implies that a scheme is affine if and only if the associated reduced scheme is affine).

Given a variety with a flag of closed subschemes one can iteratively complete and localize at the flag. This is captured by the following algorithmic definition.

\begin{definition}\label{defi:(m)}
Let $X$ be a Noetherian $k$-scheme and $A$ a $k$-algebra. Given a flag of closed subschemes
$\xi \colon X = Z_n \supset Z_{n-1} \supset \cdots \supset Z_0 \supset Z_{-1} = \emptyset$,
with $Z_i$ of pure dimension $i$, we define a collection of schemes $X^{(i)}$ for $i = -1,\dots,n-1$ by running the following recursive algorithm:
\begin{itemize}
\item[(a)] $X^{(-1)} = X_A = X \times_k \Spec A$,
\item[(b)] $Z_j^{(k)} = X_j^{(k)} \times_{X} Z_j$,
\item[(c)] $X^{(k)} = \Cb_{Z_k^{(k-1)}} (X^{(k-1)}) \setminus Z_k^{(k-1)}$.
\end{itemize}
\end{definition}

We need to verify that this algorithm is well-defined, by checking that the affineness condition of Definition \ref{defi:comp} is satisfied whenever we perform step (c). This is the content of the following lemma.

\begin{lemma}
For $k \geq 0$ the schemes $Z_k^{(k-1)}$ are affine.
\end{lemma}

\begin{proof}
This statement is established by induction on $k$, where the base case $k=0$ is clear, since $X^{(0)} = \Cb_{Z_0^{-1}} X^{(-1)}$ is defined to be the completion of $X^{(-1)} = X \times_k \Spec A$ at the affine scheme $Z_0 \times_k A$. Here we used that $Z_0$ was assumed to be zero-dimensional, and therefore is automatically affine.

Let us assume that the assertion is known for all $k \geq 0$, which satisfy $k \leq m$. We will show that it also holds for $k = m+1$. By definition we have
$$Z_{m+1}^{(m)} = X^{(m)} \times_ X Z_{m+1}.$$
Since $Z_{m+1} \hookrightarrow X$ is a closed immersion (hence in particular affine), we see that the projection map $Z_{m+1}^{(m)} = X^{(m)} \times_ X Z_{m+1} \to X^{(m)}$ is also a closed immersion (and therefore affine). Moreover, the scheme $X^{(m)}$ is constructed as the completion at the scheme $Z_m^{(m-1)}$, which we know to be affine by the induction hypothesis. This shows that $X^{(m)}$ is affine, and therefore that the closed subscheme $Z_{m+1}^{(m)}$ is affine too. This concludes the proof.
\end{proof}

\begin{rmk}\label{rmk:appendix}
Pullback along the natural morphism $(\Cb_{Z_{m+1}^{(m)}}X^{(m)},Z_{m+ 1}^{(m)}) \to (X^{(m)},Z_{m+1}^{(m)})$ of pairs induces an equivalence of derived categories of perfect complexes with support condition. In particular we have an equivalence of $K$-theory $\Kb_{\Cb_{Z_{m+1}^{(m)}}X^{(m)},Z_{m+ 1}^{(m)}} \simeq \Kb_{X^{(m)},Z_{m+1}^{(m)}}$. For $A$ a Noetherian ring this is a direct consequence of Theorem 2.6.3 in \cite{MR1106918}. The proof of the general case is deferred to Proposition \ref{prop:efimov} in the appendix.
\end{rmk}

We recall the following result from Thomason--Trobaugh \cite[Porism 2.7.1]{MR1106918}.

\begin{lemma}\label{lemma:push_perfect}
Let $X$ be a scheme of finite type over $k$, with a subscheme $Z$ finite over $k$ (in particular $\dim Z = 0$). For every $k$-algebra $A$, we denote by
$$\pi\colon X_A = X \times_k \Spec A \to \Spec A$$
the canonical projection. If $\F \in \Perf_{Z_A}(X_A)$, then $\pi_*\F$ is a perfect complex of $A$-modules.
\end{lemma}

\begin{proof}
This is a special case of Porism 2.7.1 in \cite{MR1106918}. Up to change of notation, the latter considers a finitely presented map $h\colon X \to W$, a quasi-compact open subset $U \subset X$, which is the complement of a closed immersion $Z \to X$, such that $h|_Z$ is proper, and $h|_U$ is flat. Under these assumptions it is shown that the pushforward $h_*\F$ of a perfect complex $\F$ supported on $|Z|$, is perfect.

In order to apply this result, one observes that a morphism of finite type over a field is finitely presented. Moreover, being of finite presentation, flat, or proper, is a notion invariant under base change. Since every finite morphism is in particular proper, all the conditions of the porism cited above are met.
\end{proof}

We are now in a position to state the main result of this section. At first we need to introduce some notation. We denote by $\partial_m$ the morphism of $K$-theory spectra
$$\partial_m\colon \Omega^m\Kb_{X^{(m)},Z_{m+1}^{(m)}} \to \Omega^{m-1} \Kb_{X^{(m-1)},Z_{m}^{(m-1)}},$$ obtained as the boundary map of the bi-cartesian square
\begin{equation}\label{f_sq1}
\xymatrix{
\Kb_{X^{(m-1)},Z_{m}^{(m-1)}} \ar[r] \ar[d] & \Kb_{\Cb_{Z_m^{(m-1)}}X^{(m-1)},Z_{m+1}^{(m-1)}} \ar[d] \\
\bullet \ar[r] & \Kb_{X^{(m)},Z_{m+1}^{(m)}},
}
\end{equation}
where we have used the equivalence $\Kb_{\Cb_{Z_{m}^{(m-1)}}X^{(m-1)},Z_{m}^{(m-1)}} \simeq \Kb_{X^{(m-1)},Z_{m}^{(m-1)}}$ of Remark \ref{rmk:appendix}. The morphisms in the bi-cartesian square above are induced by the inclusion maps between the respective pairs of schemes.

\begin{definition}[Preliminary Contou-Carr\`ere symbol]\label{def:prelimCCsymbol}
Let $X$ be a Noetherian $k$-scheme, and $\xi$ a saturated flag of closed subschemes $Z_i$. For every $k$-algebra $A$, we have a projection $\pi\colon X_A \rightarrow \Spec A$. The pushforward $\pi_*$ sends $\Perf_{Z_A}(X_A)$ to $\Perf(A)$. Hence, we have a well-defined map
$$\pi_*\circ \partial_1 \circ \cdots \circ \partial_n\colon \Omega^n\Kb_{A_{X,\xi}} \to \Kb_A.$$
We call this the \emph{preliminary Contou-Carr\`ere symbol} $\sigma_{X,\xi}^A.$
\end{definition}


\section{Spectral Extensions and Higher commutators}\label{extensions}

In this section we introduce the notion of a central extension of a group by a spectrum. We then define a generalization of the commutator pairing for such \emph{spectral extensions} and relate it to Loday's Steinberg symbols in algebraic $K$-theory.

\subsection{Classical central extensions}

\subsubsection{Central Extensions}\label{subsub:central}

Let $A$ be an abelian group. A \emph{central extension of $G$ by $A$}, denoted $e$, is a short exact sequence
\begin{equation}\label{eq:central}
1 \to A \to^{\iota} E  \to^{p} G \to 1,
\end{equation}
such that $\iota(A) \subset Z(E)$, where $Z(E) \subset E$ denotes the centre of $E$.

\begin{definition}\label{defi:star}
We denote by $P_n(G)$ the set of $n$-tuples of pairwise commuting elements
 $$\{(g_1,\dots,g_n) \in G^n~|~g_ig_j = g_jg_i\text{ } \forall \text{ } 1\leq i,j\leq n\}.$$
\end{definition}

Given $(f,g) \in P_2(G)$, let $\widetilde{f}$, $\widetilde{g}$ be elements in $p^{-1}(f)$, respectively $p^{-1}(g)$. Since $p(\widetilde{f}\widetilde{g}\widetilde{f}^{-1}\widetilde{g}^{-1}) = 1$, we see that the commutator $[\widetilde{f},\widetilde{g}] = \widetilde{f}\widetilde{g}\widetilde{f}^{-1}\widetilde{g}^{-1}$ defines an element in $A$. Because $A$ is central in $E$, a simple computation shows that this element is independent of the choice of liftings.
\begin{definition}
Let $e$ be a central extension of $G$ by $A$. We denote by  $\star_{e}\colon P_2(G) \to A$ the function $(f,g) \mapsto \iota^{-1}[\widetilde{f},\widetilde{g}]$.
\end{definition}

Short calculations \cite[Exercise IV.3.8(a)]{BrownCohGrp} show that $\star$ is bi-multiplicative and anti-symmetric. For the convenience of the reader we include a proof.
\begin{lemma}\label{lemma:star_basic}
For $(g_1,g_2,h) \in P_3(G)$, we have the following relations:
\begin{itemize}
	\item[\emph{(a)}] $(g_1 \star_{e} h)\cdot (g_2 \star_{e} h) = g_1g_2 \star_{e} h$,
	\item[\emph{(b)}] $(g_1 \star_{e} g_2)^{-1} = g_2 \star_{e} g_1$.
\end{itemize}
\end{lemma}
\begin{proof}
The first identity can be established by the following computation:
$$(g_1 \star_{e} h)\cdot (g_2 \star_{e} h) = (\widetilde{g}_1\widetilde{h}\widetilde{g}_1^{-1}\widetilde{h}^{-1})\cdot{} (\widetilde{g}_2\widetilde{h}\widetilde{g}_2^{-1}\widetilde{h}^{-1}) = \widetilde{g}_1(\widetilde{g}_2\widetilde{h}\widetilde{g}_2^{-1}\widetilde{h}^{-1})\widetilde{h}\widetilde{g}_1^{-1}\widetilde{h}^{-1} = \widetilde{g}_1\widetilde{g}_2\widetilde{h}\widetilde{g}_2^{-1}\widetilde{g}_1^{-1}\widetilde{h}^{-1} = (g_1g_2) \star_e h,$$
where in the second equality sign we used that $(\widetilde{g}_2\widetilde{h}\widetilde{g}_2^{-1}\widetilde{h}^{-1})$ belongs to the centre of $G$.
The second identity follows from
$$({g_1}\star_e\widetilde{g}_2)^{-1}= (\widetilde{g}_1\widetilde{g}_2\widetilde{g}_1^{-1}\widetilde{g}_2^{-1})^{-1} = \widetilde{g}_2\widetilde{g}_1\widetilde{g}_2^{-1}\widetilde{g}_1^{-1} = g_2\star_e g_1.$$
This concludes the proof of the lemma.
\end{proof}

A central extension $e$ as in \eqref{eq:central} corresponds to a monoidal map from $G$ to the groupoid $BA$ (that is, the groupoid of $A$-torsors with the natural symmetric monoidal structure). To see this directly, one observes that every fibre $p^{-1}(g) \subset E$ has the structure of an $A$-torsor. Moreover, we have a natural isomorphism $p^{-1}(gh) \cong p^{-1}(g) \otimes_A p^{-1}(h)$ for every pair $(g,h) \in G^2$. Thus, \eqref{eq:central} gives rise to a map of monoidal groupoids
\begin{equation}\label{eq:GBA}
\phi\colon G \to BA,
\end{equation}
where $G$ is viewed as a discrete groupoid with monoidal structure given by the group operation, and $BA$ denotes the classifying groupoid of $A$-torsors. The following interpretation of the commutator pairing is well-known.

\begin{lemma}
For $(f,g) \in P_2(G)$ we have that $f \star g$ corresponds to the automorphism in $BA$ obtained from the following chain of morphisms
$$\phi(fg) \cong \phi(f)\phi(g) \cong \phi(g)\phi(f) \cong \phi(gf) = \phi(fg).$$
\end{lemma}

\begin{proof}
Choosing lifts $\widetilde{f}$ of $f$ and $\widetilde{g}$ of $g$, we can express the torsors $\phi(f)$ as $A\cdot \widetilde{f}$ and $\phi(g)$ as $A \cdot \widetilde{g}$. We can also write
$$\phi(fg) \cong A\cdot \widetilde{f}\widetilde{g} \cong \phi(f) \otimes_A \phi(g).$$
The symmetry constraint of $\otimes_A$ induces an isomorphism with $A \cdot \widetilde{g}\widetilde{f}$, which sends $\widetilde{f}\widetilde{g}$ to $[\widetilde{f},\widetilde{g}]\widetilde{g}\widetilde{f}$.

Using the identification $\phi(g) \phi(f) \cong \phi(gf)=\phi(fg)$ the element $\widetilde{g}\widetilde{f}$ is sent to $\widetilde{f}\widetilde{g}$. We conclude that the resulting automorphism of the torsor $\phi(fg) \cong A \cdot \widetilde{f}\widetilde{g}$ sends the element $\widetilde{f}\widetilde{g}$ to $[\widetilde{f},\widetilde{g}]\widetilde{f}\widetilde{g}$. Therefore, it corresponds to the commutator pairing $f \star g$.
\end{proof}

\subsubsection{Cohomological Reformulation}

The map \eqref{eq:GBA} is the looping of a map of pointed spaces
$$e\colon (BG,*) \to (B^2A,*).$$
Since the target is equivalent to an Eilenberg--Mac Lane space $B^2A \cong K(A,2)$ (as unpointed spaces), homotopy classes of (unpointed) maps $BG\to B^2A$ agree with $H^2(BG,A) = H^2_{\textit{grp}}(G,A)$. We denote the element in this cohomology group resulting from $e$ by $[e]$.

If $G$ is an abelian group, then the group homology $H_*(BG,\mathbb{Z}) = H_*^{\textit{grp}}(G,\mathbb{Z})$ carries a natural graded commutative ring structure. Topologically this follows from $BG$ inheriting a group structure from the commutative group $G$, endowing it with the structure of an $H$-group. Algebraically, this fact can be explained in terms of the \emph{shuffle product} on the normalized bar complex. In the remark below we recall its definition.
\begin{rmk}\label{rmk:shuffle}
Recall that the $\mathbb{Z}G$-module $B_k$  is defined to be the free module on symbols $(g_1|\dots|g_k)$, where the $g_i$ are pairwise distinct elements of the group $G$. Using that $G$ is abelian, we define
$$(g_1|\dots|g_k)\circ(g_{k+1}|\dots|g_{k+l}) = \sum_{\sigma}(-1)^{\sigma} (g_{\sigma^{-1}(1)}|\dots|g_{\sigma^{-1}(k+l)}),$$
where $\sigma$ runs over all permutations of $\{1,\dots,k+l\}$ satisfying $\sigma(1) \leq \dots \leq \sigma(k)$ and $\sigma(k+1) \leq \dots \leq \sigma(k+l)$ (so-called \emph{shuffles}). Extending $\mathbb{Z}G$-linearly, the shuffle product endows $\bigoplus_k B_k$ with the structure of a commutative dg-algebra.
\end{rmk}

This graded commutative ring structure brings us to the following definition.

\begin{definition}\label{defi:abstract_commutator}
Let $G$ be an arbitrary group. Given $(g_1,\dots,g_n) \in P_n(G)$, we denote by $\phi\colon\mathbb{Z}^n \to G$ the corresponding morphism of groups, sending the standard vector $e_i$ to $g_i$. Let $c$ denote $(e_1 \circ \dots \circ e_n)$ as in Remark \ref{rmk:shuffle}. We set $(g_1\circ \dots \circ g_n) := \phi_*(c).$
\end{definition}

The class in $H_2^{\textit{grp}}(G,\mathbb{Z})$ corresponding to the cycle $(f \circ g)$ should be understood as an \emph{abstract commutator}. A pair $(f,g)$ of commuting elements induces a map $\mathbb{T}^2 = B\mathbb{Z}^2 \to BG$. Topologically speaking, the cycle $(f\circ g)$ is obtained by pushforward of the fundamental class of the torus $B\mathbb{Z}^2$ to $BG$.

The following lemma is standard (e.g. it is an immediate consequence of \cite[Exercise IV.3.8.(b,c)]{BrownCohGrp} combined with \cite[Theorem V.6.4(iii)]{BrownCohGrp}).
\begin{lemma}\label{lemma:cap}
Let $\langle -,-\rangle$ denote the natural pairing between group cohomology and homology. Given a central extension $e$ of $G$ by $A$, corresponding to the class $[e] \in H^2_{\textit{grp}}(G,A)$, we have for all $(f,g) \in P_2(G)$ the identity
$$f \star_{e} g = \langle [e], (f \circ g) \rangle.$$
\end{lemma}

The following definition illustrates the flexibility of the cohomological viewpoint on commutators. We use the notation $K(A,k)$ to denote an unpointed space, which represents the (co-)functor $H^k(-,A)$ valued in abelian groups.

\begin{definition}\label{defi:highercommutators}
A \emph{higher central extension} of $G$ by $B^k A = K(A,k)$ is an element $[e]$ of $H^{k+2}_{\textit{grp}}(G,A)$. Given $(g_1,\dots,g_{k+2}) \in P_{k+2}(G)$ we define $$g_1 \star_e \dots \star_e g_{k+2} = \langle [e], (g_1 \circ \dots \circ g_{k+2}) \rangle.$$
\end{definition}

In the following subsection we will formally generalize this definition to include central extensions by arbitrary spectra, not just those of Eilenberg--Mac Lane type.

\subsection{Spectral Extensions}

\subsubsection{Stable $\infty$-categories and spectra}\label{subsub:spectra}

A fundamental example of a stable $\infty$-category (see \ref{app:stable}) is given by the stable $\infty$-category $\Sp$ of \emph{spectra}, which is defined to be the limit
$$\Sp:=\invlim[\Spaces_{\bullet} \xleftarrow{\Omega} \Spaces_{\bullet} \xleftarrow{\Omega} \cdots].$$
where $\Spaces_{*}$ denotes the category of pointed spaces, and $\Omega$ denotes the pointed loop space functor.

Every pointed space $\Xc=(X,x_0)$ gives rise to a spectrum, denoted by $\Sigma^{\infty}\Xc$. The infinite suspension functor $\Sigma^\infty$ has a right adjoint
\[
\Sigma^{\infty}\colon \Spaces_{\bullet} \to \Sp \text{,} \qquad \text{namely} \qquad \Omega^{\infty}\colon \Sp \to \Spaces_{\bullet}.
\]
The latter functor is equivalent to the projection to the first component
$$\invlim[\Spaces_{*} \xleftarrow{\Omega} \Spaces_{*} \xleftarrow{\Omega} \cdots] \to \Spaces_{*}.$$
There is an array of functors to the category of abelian groups
$(\pi_i)_{i \in \mathbb{Z}}\colon\Sp \to \mathrm{Ab},$
inducing a $t$-structure on $\Sp$ with heart $\Sp^{\heartsuit} = \{X \in \Sp|\pi_i(X) = 0 \text{ for } i \neq 0\}  \cong \mathrm{Ab}$.

The subcategory $\Sp_{[0,1]} = \{X \in \Sp|\pi_i(X) = 0 \text{ for } i \neq 0,1\}$ is equivalent to the $2$-category of \emph{Picard groupoids} (that is, group-like symmetric monoidal groupoids). More generally, the $\infty$-category of connective spectra $\Sp_{\geq} = \{X \in \Sp|\pi_i(X) = 0 \text{ for } i \leq -1\}$ is equivalent to the $\infty$-category of Segal's $\Gamma$-spaces (i.e. Picard $\infty$-groupoids, or equivalently, \emph{infinite loop spaces}).

The behaviour of the $\infty$-category of spectra with respect to this $t$-structure reveals a remarkable similarity with the derived category $D(\mathbb{Z})$ of abelian groups. This time we have \emph{homology groups}
$$H_i\colon D(\mathbb{Z}) \to \mathrm{Ab},$$
inducing a $t$-structure on $D(\mathbb{Z})$. Again, the heart $D(\mathbb{Z})^{\heartsuit}$ is equivalent to the category of abelian groups. Chain complexes in $D(\mathbb{Z})_{[0,1]}$, i.e. those concentrated in degree $0$ and $1$, are, according to a theorem of Deligne, equivalent to \emph{strictly commutative Picard groupoids}. The Dold-Kan correspondence asserts that objects in $D(\mathbb{Z})_{\geq 0}$ correspond to \emph{simplicial abelian groups}.

It seems therefore appropriate to think of \emph{spectra} as another generalization of abelian groups. The derived category of abelian groups serves a similar purpose, but working with spectra corresponds to only stipulating a \emph{weak commutativity law}, which allows spectra to capture phenomena which could not be seen in the strict framework of chain complexes of abelian groups.

\subsubsection{Generalized Group Cohomology}

For every spectrum $\mathbb{E}$, we have an associated \emph{generalized cohomology theory} denoted by
$$H^i(-,\mathbb{E})\colon \Spaces \to \mathrm{Ab}.$$
We define generalized group cohomology to be $H^i_{\textit{grp}}(G,\mathbb{E}) = H^i(BG,\mathbb{E})$.

\begin{definition}\label{defi:spectral_extension}
A \emph{spectral extension} of $G$ by $\mathbb{E}$ is a class $[e] \in H^2_{\textit{grp}}(G,\mathbb{E})$.
\end{definition}

Every abelian group $A$ can be viewed as a spectrum $HA$ by means of the Eilenberg--Mac Lane construction. Forgetting base points, we have an equivalence of unpointed spaces $\Omega^{\infty}\Sigma^k HA \simeq B^kA$.

A higher central extension of $G$ by $B^kA$ in the sense of Definition \ref{defi:highercommutators}, is given by an element of $H_{\textit{grp}}^{k+2}(G,A) = H_{\textit{grp}}^2(G,\Sigma^kHA)$. By the discussion above, we can therefore say that a higher central extension of $G$ by $B^kA$ is the same thing as a spectral extension of $G$ by the $k$-fold suspension spectrum $\Sigma^k HA$.

\begin{rmk}
    The definition of a spectral extension given in Definition \ref{defi:spectral_extension} introduces only the cocycle (up to equivalence) of what should be a \emph{central extension by a spectrum}. Without doubt it would be possible to give a definition along the lines of Paragraph \ref{subsub:central}. However, spelling out such a definition would certainly be more cumbersome than the shortcut used in Definition \ref{defi:spectral_extension}, which is exactly the viewpoint we need to study higher commutators in the next paragraph.
\end{rmk}

\subsection{The case of spectral extensions}

\subsubsection{Basic definitions}

Our definition of higher commutators for spectral extensions hinges on three $\infty$-categories whose objects belong to the canon of classical homotopy theory. These $\infty$-categories have already made an appearance earlier in the paper.

\begin{itemize}
\item[(a)] The $\infty$-category of unpointed spaces $\Spaces$, as defined in \cite[Definition 1.2.16.1]{Lurie:bh}. 
\item[(b)] The $\infty$-category of pointed spaces will be referred to as $\Spaces_{\bullet}$ (see \cite[Notation 1.4.2.5]{Lurie:ha}).
\item[(c)] The stable $\infty$-category of spectra $\Spectra$ (see \cite[Definition 1.4.3.1]{Lurie:ha}).
\end{itemize}

All three $\infty$-categories happen to be generated under small colimits by a single object.
Spaces $\Spaces$ are generated by the singleton $\{\bullet\}$ (see \cite[Theorem 5.1.5.6]{Lurie:bh} applied to $S$ being the $\infty$-category consisting of a single object and only the identity morphism), pointed spaces $\Spaces_{\bullet}$ by a pointed space with two elements $(S^0,x_0)$ (combine the aforementioned result, and \cite[Proposition 4.8.2.11]{Lurie:ha}), and $\Spectra$ is generated by the sphere spectrum $\Sb$ (see \cite[Corollary 1.4.4.6]{Lurie:ha}).
Furthermore, these $\infty$-categories do not just exist in isolation from each other, but are related by a chain of functors.

\begin{itemize}
\item[(d)] The functor $(-)_+\colon \Spaces \to \Spaces_{\bullet}$ is well-defined (up to a contractible space of choices) by the fact that it commutes with small colimits and sends the singleton space $\{\bullet\}$ to a pointed space with two elements $(S^0,x_0)$. Informally speaking it assigns to a space $X$ the pointed space $X_+$ obtained as the disjoint union $X \sqcup \{x_0\}$ with base point $x_0$.
\item[(e)] The infinite suspension functor $\Sigma^{\infty}\colon \Spaces_{\bullet} \to \Spectra$ is well-defined (up to a contractible space of choices) by stipulating that it commutes with small colimits, and sends $(S^0,x_0)$ to the sphere spectrum $\Sb$.
\end{itemize}

\begin{definition}
The composition of $(-)_+$ and $\Sigma^{\infty}$ will be denoted by $\Sigma^{\infty}_+$.
\end{definition}

The fact that $\Spaces$, $\Spaces_{\bullet}$ and $\Spectra$ are generated by one object is not only convenient for defining functors between them, but also implies directly that they are presentable $\infty$-categories (see \cite[Theorem 5.5.1.1(6)]{Lurie:bh}). In \cite[\S 4.8.2]{Lurie:ha}, a symmetric monoidal structure on $\Pr^L$ (the $\infty$-category of small presentable $\infty$-categories) is used to establish the existence of symmetric monoidal smash products on $\Spaces_{\bullet}$ and $\Spectra$, as well as compatibility between them. Just as one can talk of commutative algebra objects in a symmetric monoidal (1-)category, one can talk about analogous objects in a symmetric monoidal $\infty$-category.  These go by the name of $E_\infty$-rings or $E_\infty$-objects (see \cite{May} or \cite[\S 7]{Lurie:ha}). We encourage the reader to think of these as a higher homotopical analogue of commutative rings, or commutative DGAs.

\subsubsection{Short summary}

The functor $\Sigma^{\infty}_+\colon \Spaces \to \Spectra$ is symmetric monoidal with respect to the cartesian symmetric monoidal structure on unpointed spaces and the smash product of spectra $\otimes$. That is, for two unpointed spaces $X$, $Y$ we have $\Sigma^{\infty}_+(X \times Y) \simeq \Sigma^{\infty}_+ X \otimes \Sigma^{\infty}_+ Y$.

In particular this functor preserves $E_{\infty}$-objects. We conclude that $\Sigma^{\infty}_+B\mathbb{Z}$ is a commutative ring spectrum. This induces a graded commutative product structure on $\pi_*$ and allows one to define higher commutators. We will now describe all of this in more detail.

\subsubsection{Facts from modern homotopy theory}\label{ssub:notation_and_facts}

\begin{itemize}
\item[(f)] There exists a canonical symmetric monoidal structure on $\Spaces$, the \emph{cartesian symmetric monoidal structure} (see \cite[Sect. 2.4.3]{Lurie:ha}). We denote the corresponding symmetric monoidal $\infty$-category by $\Spaces_{\times}$. Up to a contractible space of choices it is well-defined by the fact that $\{\bullet\}$ is a unit, and the induced bi-functor $\times\colon \Spaces \times \Spaces \to \Spaces$ commutes in both variables with small colimits.
\item[(g)]  There exists a canonical symmetric monoidal structure on $\Spaces_{\bullet}$, we denote the symmetric monoidal category by $(\Spaces_{\bullet})_{\wedge}$. Up to a contractible space of choices it is well-defined by the property of having $(S^0,x_0)$ as a unit, and the induced bi-functor $\wedge\colon \Spaces_{\bullet} \times \Spaces_{\bullet} \to \Spaces_{\bullet}$ commuting with small colimits in both variables (see \cite[Remark 4.8.2.11]{Lurie:ha}).
\item[(h)] There exists a canonical symmetric monoidal structure on $\Spectra$, we denote the symmetric monoidal category by $(\Spectra)_{\otimes}$. Up to a contractible space of choices it is well-defined by the property of having the sphere spectrum $\Sb$ as a unit, and the induced bi-functor $\otimes\colon \Spectra \times \Spectra \to \Spectra$ commuting with small colimits in both variables (see \cite[Corollary 4.8.2.19]{Lurie:ha}).
\item[(i)] The functors $(-)_+$ and $\Sigma^{\infty}$ have a natural symmetric monoidal structure (well-defined up to a contractible space of choices). In particular we have equivalences (well-defined up to a contractible space)
\begin{equation}\label{eqn:iso1}X_+ \wedge Y_+ \simeq (X \times Y)_+
\end{equation}
 for $X,Y \in \Spaces$ and
\begin{equation}\label{eqn:iso2}
\Sigma^{\infty}\Xc \otimes \Sigma^{\infty}\mathcal{Y} \simeq \Sigma^{\infty}(\Xc \wedge \mathcal{Y})
\end{equation}
 for $\Xc, \mathcal{Y} \in \Spaces_{\bullet}$.
\end{itemize}

The facts (f)-(h) are well-known outside of the context of stable $\infty$-categories. After passing to homotopy categories, one recovers the classical concepts. In particular, the symmetric monoidal structure $\wedge$ on $\Spaces_{\bullet}$ may be thought of as the smash product of pointed spaces
$$(X,x_0) \wedge (Y,y_0) = \left((X \times Y)/(X \times \{y_0\} \cup \{x_0\} \times Y), (x_0,y_0)\right).$$
The advantage of the present approach is that it foregrounds the treatment of homotopy coherence, rather than having to build this after the fact for a particular space-level construction. We give a more detailed account of the proof of (i), since it is only implicit in \cite[\S 4.8.2]{Lurie:ha}.

\begin{proof}[Proof of (i)]
In \cite[Definition 4.8.2.8]{Lurie:ha} Lurie defines what it means for a small colimit preserving functor of presentable $\infty$-categories $\Spaces \to \C$ to realize $\C$ as an idempotent object. It is then shown (see \cite[Proposition 4.8.2.9]{Lurie:ha}) that there is an equivalence on the full subcategory of $\Fun(\Delta^1,\mathsf{Pr}^L)$ corresponding to such morphisms, and the over-category of presentable symmetric monoidal $\infty$-categories over $\Spaces_{\times}$.

In \cite[Proposition 4.8.2.11 \& 4.8.2.18]{Lurie:ha} it is shown that $\Spaces_{\bullet}$ and $\Spectra$ are naturally idempotent categories with respect to the canonical functors $\Spaces \to^{(-)_+} \Spaces_{\bullet}$ and $\Spaces \to^{\Sigma^{\infty}_+} \Spectra$. This is then used to deduce the existence of the symmetric monoidal structure mentioned in (g,h). The same observation implies that $(-)_+$ and $\Sigma^{\infty}$ are symmetric monoidal functors (see \cite[Proposition 4.8.2.7]{Lurie:ha}).
\end{proof}

Putting all of the facts recited above together, we obtain the following consequences.

\begin{corollary}\label{cor:Sigmaplus}
The composition of functors $\Sigma^{\infty}(-)_+ \circ \mathsf{forget}\colon \Spaces_{\bullet} \to \Spaces \to \Spectra$ is naturally equivalent to $\Sb \oplus \Sigma^{\infty}(-)$.
\end{corollary}

\begin{proof}
We begin by considering the category of pointed simplicial sets, which we will use as a model for the $\infty$-category of pointed spaces.
Let $\Xc= (X,x_0)$ be a pointed simplicial set. We use the notation $S^0$ to denote the pointed simplicial set corresponding to the $0$-sphere. We denote by $i\colon S^0 \to X_+$ the map of pointed simplicial sets which sends the base point of $S^0$ to the base point of $X_+$, and the unique non-base point of $S^0$ to $x_0 \in X$. Let $s\colon X_+ \to S^0$ be the unique left-inverse to this map in the category of pointed simplicial sets. We denote by $g\colon X_+ \to \Xc$ the unique map of pointed simplicial sets, such that $g|_X = \id_X$.
These maps belong to a natural commutative diagram of pointed simplicial sets
\[
\xymatrix{
 S^0 \ar[r]_i \ar[d] & X_+ \ar[d]^g \ar@/_0.7pc/@{-->}[l]_{s} \\
\bullet \ar[r] & \Xc.
}
\]
The dashed arrow refers to the well-defined retract in pointed simplicial sets. For every $X$ there is a unique retract, hence it is natural. Passing from model categories to $\infty$-categories (see \cite[A.2(2)]{Lurie:bh} we obtain a natural commutative diagram of functors taking values in $\Spaces_{\bullet}$.
\[
\xymatrix{
 S^0 \ar[r] \ar[d] & (-)_+\circ \mathsf{forget} \ar[d] \ar@/_0.7pc/@{-->}[l] \\
\bullet \ar[r] & \id.
}
\]
Furthermore we remark that this is a cofibre diagram in $\Spaces_{\bullet}$. Applying the functor $\Sigma^{\infty}$ (which has a right adjoint and hence preserves small colimits by \cite[Proposition 5.2.3.5]{Lurie:bh} ) we obtain a natural bi-cartesian diagram of $\Spectra$-valued functors $\Spaces_{\bullet} \to \Spectra$, with a canonical splitting
\[
\xymatrix{
\Sb \ar[r] \ar[d] & \Sigma^{\infty}(-)_+\circ \mathsf{forget} \ar[d] \ar@/_0.7pc/@{-->}[l] \\
0 \ar[r] & \Sigma^{\infty}.
}
\]
We conclude that there is a natural equivalence $\Sigma^{\infty}(-)_+\circ \mathsf{forget} \simeq \Sb \oplus \Sigma^{\infty}$ of $\Spectra$-valued functors.
\end{proof}

Specialising this to the pointed space $(S^1,1)$ we obtain the equivalence:

\begin{corollary}\label{cor:sbshift}
$\Sigma^{\infty}_+S^1 \simeq \Sb \oplus \Sigma\Sb$.
\end{corollary}

The following assertion lies at the heart of the definition of higher commutators for spectral extensions.

\begin{corollary}\label{cor:can_splitting}
The equivalence of Corollary \ref{cor:Sigmaplus} induces for a pointed space $\Xc = (X,x_0)$ a natural morphism and a natural left inverse thereof
$$\Sigma^{\infty}\Xc^{\wedge n}  \leftrightarrows \Sigma^{\infty}_+(X^n) .$$
\end{corollary}

\begin{proof}
We have recorded in (e) above that $\Sigma^{\infty}_+$ is symmetric monoidal. Hence for every positive integer $n$, and every unpointed space $X$ we get a contractible space of morphisms
$$\Sigma^{\infty}_+(X^n) \to (\Sigma^{\infty}_+ X)^{\otimes n}.$$
For $\Xc = (X,x_0)$ we can draw on the natural equivalence of functors $\Sigma^{\infty}(-)_+ \simeq \Sb \oplus \Sigma^{\infty}(-)$ for $\Xc$, which allows us to define a natural morphism with left inverse
$$(\Sigma_+^{\infty} X)^{\otimes n} \leftrightarrows  \Sigma^{\infty} \Xc^{\wedge n}.$$
Here we used repeatedly that $\otimes$ commutes with small colimits in its entries.
\end{proof}

We observe the following:

\begin{rmk}
The corollary above can be refined to produce a natural equivalence
$$\Sigma^{\infty}_+(X^n) \simeq \oplus_{i=0}^n(\Sigma^{\infty}\Xc^{\wedge i})^{\oplus{ {n}\choose{i}}}$$
for every pointed space $\Xc = (X,x_0)$.
\end{rmk}

The last conclusion we draw is again rather general.

\begin{corollary}\label{cor:ring_spectrum}\label{cor:monoidtorings}
The functor $\Sigma^{\infty}_+\colon \Spaces \to \Spectra$ sends $E_{\infty}$-objects in (unpointed) spaces to $E_{\infty}$-ring spectra.
\end{corollary}

\begin{proof}
It is a general statement that symmetric monoidal functors preserve $E_{\infty}$-objects (also called commutative algebra objects in \cite{Lurie:ha}). This follows directly from the definitions, we give the proof since we could not find a reference. In the notation of \cite[Chapter 2]{Lurie:ha}, a symmetric monoidal structure on an $\infty$-category is encoded by a functor $\C_{\otimes} \to N(\mathsf{Fin}_*)$ satisfying certain properties (see \cite[Definition 2.0.0.7]{Lurie:ha}). A symmetric monoidal functor is given by a commutative diagram (see \cite{Lurie:ha} for a precise account of further technical conditions required from the functor)
\[
\xymatrix{
\C_{\otimes} \ar[r] \ar[rd] & \D_{\otimes} \ar[d] \\
& N(\mathsf{Fin}_*).
}
\]
On the other hand, a commutative algebra object in $\C_{\otimes}$ is encoded by a section $\C_{\otimes} \leftrightarrows N(\mathsf{Fin}_*)$ which is a map of $\infty$-operads (see \cite[Definition 2.1.3.1]{Lurie:ha}). It is therefore clear that a symmetric monoidal functor sends such a section for $\C_{\otimes} \to N(\mathsf{Fin}_*)$ to one for $\D_{\otimes}$.
\end{proof}

\subsubsection{The definition of higher commutators}

Let $X$ be a groupoid, that is a category where all morphisms are invertible. The groupoid $X$ gives rise to an unpointed space, namely the geometric realization of its nerve $|NX|$. Henceforth this will be implicit, and we use this construction to realize the $2$-category of groupoids as a full subcategory of the $\infty$-category of unpointed spaces $\Spaces$.

Every group $G$ gives rise to a groupoid $BG$, by definition the category with a unique object $\{\ast\}$ and $\Aut_{BG}(\ast) = G$. Consistent with the paragraph above, we also denote the associated unpointed space by $BG$. However we remark that the functor
$B\colon \mathsf{Grps} \to \Spaces$
factors through the $\infty$-category of pointed spaces $\Spaces_\bullet$:
$$\mathcal{B}\colon G \mapsto (BG,\ast).$$

Let $\Eb$ be a spectrum, and $X$ a groupoid. A spectral extension of $X$ by $\Eb$ is defined to be a morphism $\Sigma^{\infty}_+X \to \Sigma^2 \Eb$. In particular, for $G$ a group, a spectral extension of $G$ by $\Eb$ is given by a morphism $e\colon \Sigma^{\infty}_+BG \to \Sigma^2\Eb$.

\begin{definition}
Let $x \in X$ be an object. We denote by $P^n_{X}$ the groupoid whose objects are tuples $(x;g_1,\dots,g_n)$ where $x \in X$ is an object, and $(g_1,\dots,g_n) \in (\Aut_{X}(x))^n$ is an $n$-tuple of pairwise commuting automorphisms. Morphisms $(x;g_1,\dots,g_n) \to (y;h_1,\dots,h_n)$ in $P^n_{X}$ are given by a morphism $f\colon x \to y$, such that $f \circ g_i = h_i \circ f$ for all $1 \leq i \leq n$. This defines an endofunctor $P_n$ of the $2$-category of groupoids $\mathsf{Grpd}$.
\end{definition}

Objects of $P^n_{X}$ are pairwise commuting $n$-tuples of automorphisms in $X$ at a fixed object $x$. It follows that there is a natural equivalence $P^n_{X} \simeq \Map((B\mathbb{Z}^n)_+,X_+)$.

\begin{lemma}
With respect to the embedding of groupoids in unpointed spaces, the functor $$P^n\colon \mathsf{Grpd} \to \mathsf{Grpd}$$ is naturally equivalent to $\Map_{\Spaces}(B\mathbb{Z}^n,X)$.
\end{lemma}
\begin{proof}
We obtain the object $x$ as the image of $\ast$ under the corresponding map of unpointed spaces $B\mathbb{Z}^n \to X$. Since $\Aut_{B\mathbb{Z}^n}(x_0) = \mathbb{Z}^n$ we have a canonical choice for an $n$-tuple of pairwise commuting automorphisms, given by the standard basis of $\mathbb{Z}^n$. We then transport this choice along the induced map of groups $\mathbb{Z}^n \to \Aut_{X}(x)$ to conclude the proof.
\end{proof}

Recall that $B\mathbb{Z}^n$ is a strict abelian group object in $\Spaces$. We obtain from Corollary \ref{cor:ring_spectrum} that $\Sigma_+^{\infty}(B\mathbb{Z}^n)$ is an $E_{\infty}$-ring spectrum. In particular we see that $\pi_*(\Sigma_+^{\infty}(B\mathbb{Z}^n))$ is a graded commutative algebra. We will write $(x|y)$ to denote the product of two elements, and more generally $(x_1|\cdots |x_n)$ for the product of $n$ elements.

\begin{definition}\label{defi:hcspectra}
Let $e\colon \Sigma^{\infty}_+X \to \Sigma^2\Eb$ be a spectral extension of a groupoid $X$ by a spectrum $\Eb$.
\begin{itemize}
\item[(a)] Corollary \ref{cor:Sigmaplus} applied to the pointed space $(S^1,1)=(B\mathbb{Z},\ast)$ yields a splitting $\Sigma^{\infty}_+B\mathbb{Z} \simeq \Sb \oplus \Sigma \Sb$. The map of spectra $\Sigma \Sb \to \Sigma^{\infty}_+B\mathbb{Z}$ will be denoted by $f$.
\item[(b)] For an integer $i$ satisfying $1 \leq i \leq n$ we let $\mathbb{Z} \to \mathbb{Z}^n$ be the map $\lambda \mapsto (0,\dots,0,\lambda,0,\dots, 0)$ given by the inclusion of the $i$-th component. We write $\phi_i\colon B\mathbb{Z} \to B\mathbb{Z}^n$ for the induced map of pointed spaces. The induced element $(\phi_i)_*(f) \in \pi_1(\Sigma^{\infty}_+B\mathbb{Z}^n)$ is denoted by $e_i$.
\item[(c)] Let $(B\mathbb{Z}^n) \to^g X$ be an $n$-tuple of pairwise commuting automorphisms. We denote the induced map of spectra $\Sigma^{\infty}_+(B\mathbb{Z}^n) \to^g \Sigma^{\infty}_+(X) \to^e \Sigma^2\Eb$ by $\phi(g,e)$. The higher commutator is defined to be $\phi(g,e)_*(e_1 | \cdots | e_n) \in \pi_{n-2}(\mathbb{E})$ and will be denoted by $g_1 \ast_e \cdots \ast_e g_n$.
\end{itemize}
\end{definition}

\subsubsection{Comparison}

In order to show that this definition is non-trivial we compare it to the construction of Definition \ref{defi:highercommutators}.

\begin{lemma}\label{lemma:aby}
Let $G$ be a group, and let $[e] \in H^{n+2}_{grp}(G,A)$ be a higher central extension corresponding to a map $e\colon BG \to B^{n+2}A$ and let $(g_0,\dots,g_{n+1})$ be an $(n+2)$-tuple of pairwise commuting elements. With respect to the natural isomorphism $\pi_{n}(K(A,n)) \simeq A$ we have that the spectral higher commutator $g_0 \ast_e \cdots \ast_e g_{n+1}$ agrees with the higher commutator $(g_0,\dots,g_{n+1})_e$ of Definition \ref{defi:highercommutators}.
\end{lemma}

\begin{proof}
Recall that for an unpointed space $X$ we have the Hurewicz morphism $h\colon \pi_*(\Sigma^{\infty}_+X) \to H_*(X,\mathbb{Z})$, obtained by applying the functor $\pi_*$ to the morphism of spectra
\begin{equation}\label{ringspecmor}
\Sigma^{\infty}_+X \to \Sigma^{\infty}_+X \otimes \mathbb{Z},
\end{equation}
where we use the notation $\mathbb{Z}$ to denote the Eilenberg-Mac Lane spectrum corresponding to the ring $\mathbb{Z}$.

If $X$ is an $E_{\infty}$-object in (unpointed) spaces (that is, a commutative monoid), we have already seen that $\pi_*(\Sigma^{\infty}_+X)$ and $H_*(X,\mathbb{Z})$ are endowed with a graded commutative product structure. The Hurewicz morphism respects this product, since the morphism \eqref{ringspecmor} is a morphism of $E_{\infty}$-ring spectra, induced by the morphism of $E_{\infty}$-ring spectra $\Sb \to \mathbb{Z}$.

We have a commutative diagram of abelian groups
\[
\xymatrix{
\pi_{n+2}(\Sigma^{\infty}_+ B\mathbb{Z}^n) \otimes H^{0}(B\mathbb{Z}^n,\Sigma^{n+2} A) \ar[r]^-{h \otimes \id} \ar[rd] & H_{n+2}(B\mathbb{Z}^n,\mathbb{Z}^n) \otimes H^{n+2}(B\mathbb{Z}^n,A) \ar[d] \\
& A,
}
\]
and furthermore we have $h(e_0|\cdots|e_{n+1}) = (e_0|\dots|e_{n+1})$ by the discussion above. By virtue of Definitions \ref{defi:hcspectra} and \ref{defi:highercommutators} we conclude that the new notions of higher commutators agree.
\end{proof}

\subsubsection{Computing higher commutators recursively}\label{sub:recursion}

For a group $G$, and $g\in G$, we write $C_G(g)$ to denote the centralizer. We now describe a version of the classical slash product to associate to a spectral extension $e\colon \Sigma^{\infty}_+X \to \Sigma^2\mathbb{E}$ and an element $g \in \Aut_{X}(x)$, a spectral extension $e\langle g \rangle \colon \Sigma^{\infty}_+BC_{\Aut_{X}(x)}(g) \to \Sigma^2(\Omega \mathbb{E})$.

\begin{definition}\label{defi:recursy}
Let $X$ be a groupoid, $x \in X$ an object, and $g \in \Aut_{X}(x)=G$ an automorphism. We denote by $e\colon \Sigma^{\infty}_+ X \to \Sigma^2\mathbb{E}$ a spectral extension of $X$ by $\mathbb{E}$.
\begin{itemize}
\item[(a)] We let $(BC_{\Aut_{X}(x)}(g)) \to X$ be the map of unpointed spaces induced by the inclusion $C_{\Aut_{X}(x)}(g) \subset \Aut_{X}(x)$.
\item[(b)] Let $B\mathbb{Z} \times BC_{\Aut_{X}(x)}(g) \to BG$ be the map of unpointed spaces induced by the map of  groups $\mathbb{Z} \times C_{\Aut_{X}(x)}(g) \to G$ sending $1\in\mathbb{Z}$ to $g$, and given by the inclusion of $C_{\Aut_X(x)}(g)$.
\item[(c)] We denote by $\Sigma \Sb \to \Sigma^{\infty}_+ B\mathbb{Z} \simeq \Sigma^{\infty}_+ S^1$ the map specified by Corollary \ref{cor:sbshift}.
\item[(d)] The map $e\langle g \rangle\colon \Sigma^{\infty}_+ BC_{\Aut_{X}(x)}(g) \to \Sigma^2(\Omega \mathbb{E})$ is defined to be the adjoint to the map
$$\Sigma \Sigma^{\infty}_+BC_{\Aut_{X}(x)}(g) \to \Sigma^2\mathbb{E}$$ defined by the composition
$$\Sigma \Sigma^{\infty}_+BC_{\Aut_{X}(x)}(g) \simeq \Sigma \Sb \otimes \Sigma^{\infty}_+BC_{\Aut_{X}(x)}(g) \to \Sigma^{\infty}_+(B\mathbb{Z} \times BC_{\Aut_{X}(x)}(g)) \to \Sigma^{\infty}_+ G \to \Sigma^2\mathbb{E},$$
where we have used that $\Sigma^{\infty}_+$ is symmetric monoidal as explained in (i) above.
\end{itemize}
\end{definition}

The following assertion follows right from the definitions.

\begin{lemma}\label{lemma:recursion}
Let $(g_1,\dots,g_n) \in P_n(G)$. Then we have
$(g_1\star \dots \star g_n)_e \cong (g_2 \star \dots \star g_n)_{e\langle g_1 \rangle}.$
\end{lemma}

\subsubsection{Comparison with Osipov--Zhu's definition for $n=3$}

Recall that a groupoid endowed with a symmetric monoidal structure is called a \emph{Picard groupoid}, if the monoidal structure is \emph{group-like}. That is, the induced monoid structure on the set of isomorphism classes is a group structure. We denote by $\mathbf{P}$ a Picard groupoid, and by $\mathbf{0} \in \mathbf{P}$ a unit. The group $\Aut_{\mathbf{P}}(0)$ is abelian (as a consequence of the Eckmann--Hilton trick), and will be denoted by $\mathbf{\Omega}\mathbf{P}$. Similar conventions will be applied to Picard $2$-groupoids, that is, group-like symmetric monoidal $(2,1)$-groupoids.

Given a spectral extension $e$ of $G$ by $\mathbb{E}$ and an element $g \in G$, we constructed (see Definition \ref{defi:recursy}) a spectral extension $e\langle g \rangle$ of the centralizer $C(g)$ by $\Omega \mathbb{E}$. This shifted spectral extension satisfies the identity (Lemma \ref{lemma:recursion})
$$(g_1\star \dots \star g_n)_e \cong (g_2 \star \dots \star g_n)_{e\langle g_1 \rangle}.$$
Readers of Osipov--Zhu's \cite{MR2833793} will recognize the similarity with their recursive definition of higher commutators. The authors of \emph{loc. cit.} associate to an extension $\psi$ of a group $G$ by a Picard groupoid $\mathbf{P}$, and an element $f \in G$ a (graded) central extension $\psi_f$ of $C(f)$ by $\Aut_{\mathbf{P}}(\mathbf{0})$. Eventually, the commutator $C_3(f,g,h)$ is defined to be $\mathrm{Comm}(\psi_f)(g,h)$ with respect to the latter central extension. Here $\mathrm{Comm}(\psi_f)$ denotes the commutator pairing of \cite[Lemma-Definition 2.5]{MR2833793}.

\begin{proposition}\label{prop:OZ_C3}
Let $\phi\colon G \to B\mathbf{P}$ be the monoidal map corresponding to a central extension of $G$ by $\mathbf{P}$. We denote by $e\colon \Sigma^{\infty}BG \to \Sigma\mathbb{B}\mathbf{P}$ the corresponding spectral extension of $G$ by $\Omega\mathbb{B}\mathbf{P}$, the spectrum associated to the Picard groupoid $\mathbf{P}$. Then,
$$C_3(f,g,h) = f\star g \star h.$$
\end{proposition}

\begin{proof}
At first we recall Osipov--Zhu's construction of the central extension of the centralizer $C(f)$ by $\mathbf{\Omega}^2\mathbf{P}$. In \cite[Lemma-Definition 2.13]{MR2833793}
they define a symmetric monoidal map
$$C(f) \to \mathbf{P},$$
which sends $g \in C(f)$ to the element of $\mathbf{P}\simeq Aut_{B\mathbf{P}}(\mathbf{0})$ given by
\begin{equation}\label{mappy}\phi(fg) \cong \phi(f)\phi(g) \cong \phi(g)\phi(f) \cong \phi(gf) \cong \phi(fg).\end{equation}
It is well-known that the $(3,1)$-category of group-like symmetric monoidal $(2,1)$-groupoids is equivalent to the full subcategory $\Spectra_{[0,2]}$ of the $\infty$-category of spectra $\Spectra$, consisting of spectra $\Eb$ with vanishing $\pi_i(\Eb)$ for $i \notin [0,2]$. This assertion can be deduced from a result of Boardman--Vogt and May (see \cite[Theorem 5.2.6.10]{Lurie:ha}).

This equivalence allows one to consider $\mathbf{P}$ as a spectrum, which we denote $\Omega\mathbb{B}\mathbf{P}$. Osipov--Zhu's map \eqref{mappy} is then an explicit description of the adjoint to
$$\Sigma \Sigma^{\infty}_+BC(f) \simeq \Sigma^{\infty}_+(B\mathbb{Z} \times BC(f)) \to \mathbb{B}\mathbf{P},$$
and hence is equivalent to the central extension $e\langle f \rangle$ defined in Definition \ref{defi:recursy}. We infer the following assertion:

\begin{claim}\label{firstclaim}
The symmetric monoidal map $\psi_f\colon C(f) \to \mathbf{P}$ defined in \cite[Lemma-Definition 2.13]{MR2833793} is homotopic to the map $e\langle f \rangle$ of Definition \ref{defi:recursy}, with respect to the natural embedding of Picard groupoids into the $\infty$-category of spectra.
\end{claim}

It remains to compare Osipov--Zhu's $\mathrm{Comm}(\psi)(g,h)$ of \cite[Lemma-Definition 2.5]{MR2833793} with $g \ast_{e\langle f \rangle} h$. This is the content of the next assertion.

\begin{claim}\label{secondclaim}
Let $H$ be a group, $\mathbf{P}$ a Picard groupoid, and $\psi\colon H\to \mathbf{P}$ a monoidal morphism. We denote by $\alpha$ the corresponding spectral extension of $H$ by $\Omega\mathbb{B}\mathbf{P}$. Then we have for $(g,h) \in P_2(H)$ the equality $\mathrm{Comm}(\psi)(g,h) = g \ast_{\alpha} h$ of elements of $\Aut_{\mathbf{P}}(\mathbf{0})$.
\end{claim}

For any $g\in C(f)$, we have that the map $(\psi_f)_g\colon C(g)\to \Aut_{\mathbf{P}}(\mathbf{0})$ of \cite[Lemma-Definition 2.5]{MR2833793} is homotopic to $\alpha\langle g\rangle$ (by an argument analogous to the one above, one category level down).

To deduce Claim \ref{secondclaim}, observe that it follows directly from the definition given in \cite[Lemma-Definition 2.5]{MR2833793} that $\mathrm{Comm}(\psi)(g,h) = \psi_g(h)$, and similarly, we know by virtue of Lemma \ref{lemma:recursion} that $\alpha\langle g \rangle (h) = g \ast_{\alpha} h$. We deduce $$\mathrm{Comm}(\psi)(g,h) = g \ast_{\alpha} h.$$ This concludes the proof of Claim \ref{secondclaim}. The proposition follows.
\end{proof}

\subsection{Spectral Extensions coming from the {K}-theory of rings}\label{sub:examples_commutators}
We begin with a quick review of the relevant facts about $K$-theory.  This will also serve to fix notation.  Experts should feel free to skip ahead.


\subsubsection{Steinberg Symbols}
In the following we denote by $R$ a ring, again assumed commutative and unital. Careful inspection of the definition of your choice of algebraic $K$-theory, reveals the existence of a canonical morphism
\begin{equation}\label{eqn:GL}
\Sigma^\infty_+ \coprod_{n \in \mathbb{N}_{\geq 1}} B\GL_n(R) \to \Kb_R.
\end{equation}
More generally, for a stable $\infty$-category $\C$, there is a canonical morphism
\begin{equation}\label{eqn:canonical_extension}
\Sigma_+^\infty \C^{\times} \to \Kb_{\C}.
\end{equation}
The morphism \eqref{eqn:GL} is a special case of this construction.\footnote{i.e. after factoring through the inclusion $P_f(R)^\times\to\Perf(R)^\times$.}
\begin{definition}\label{defi:canonical}
    The existence of the morphism \eqref{eqn:GL} can be restated as saying that the groupoid $\coprod_{n \in \mathbb{N}}B\GL_n(R)$ is canonically endowed with a central extension by $\Omega^2\Kb_R$. Similarly, \eqref{eqn:canonical_extension} amounts to the $\infty$-groupoid $\C^{\times}$ being endowed with a central extension by $\Omega^2\Kb_{\C}$. We will denote the extensions by $e_R$ and $e_{\C}$ respectively.
\end{definition}

The central extension of $\GL_n(R)$ by $\Omega^2\Kb_R$ has appeared in work of Safronov \cite{Safronov:2013uq}. The theory of higher commutators developed in this section enables us to generalize \emph{Steinberg symbols} to a general stable $\infty$-category.

\begin{definition}\label{defi:Steinberg_Waldhausen}
We denote by $(g_1,\dots,g_n) \in P_n(\C^\times,x)$ a map of unpointed spaces $\mathbb{T}^n \to \C^{\times}$ mapping the base point of $\mathbb{T}^n$ to $x$. The map
$$(g_1\star\cdots\star g_n)_{e_\C}\colon \Sigma^\infty\mathbb{S}^n_+ \to \Kb_{\C}$$
is referred to as the higher commutator with respect to the natural extension of $\C^{\grp}$ by $\Omega^2\Kb_{\C}$.
\end{definition}

The justification of the terminology \emph{Steinberg symbol} is provided by the next proposition, which compares the higher commutators of Definition \ref{defi:Steinberg_Waldhausen} with Loday's higher Steinberg symbols, for the category of finitely generated projective $R$-modules.

\begin{proposition}\label{prop:Steinberg}
Let $R$ be a commutative ring, and $r_1,\dots,r_n \in R^{\times}$ be an $n$-tuple of units in $R$. The higher commutator $(r_1 \star \dots \star r_n)_{e_R}$, computed with respect to the spectral extension $e_R$ of Definition \ref{defi:canonical}, agrees with Loday's higher Steinberg symbol $\{r_1,\dots,r_n\}$.
\end{proposition}

Before giving the proof, we recall Loday's definition from \cite{MR0447373}. In modern language, Loday's construction of the Steinberg symbols relies on the $E_\infty$-ring structure of $K_R$ (in which the product is induced by the tensor product $\otimes$ of $R$-modules). If $\alpha_1,\dots,\alpha_n$ is an $n$-tuple of paths in $K_R$ based at $\mathbf{0} \in K_R$, the multiplication $\otimes$ induces a map
$$\Sigma^\infty{(\mathbb{S}^{1})}^{\wedge n} \to\Kb_R^{\wedge n}\to \Kb_R,$$
which defines an element $(\alpha_1 | \dots | \alpha_n)$ of $\pi_n(\Kb_R) = K_n(R)$.
%

\begin{proof}[Proof of Proposition \ref{prop:Steinberg}]
Let $\C$ be an exact category with a bi-exact symmetric monoidal structure $\otimes$. This endows the maximal pointed groupoid $\C^{\times}$ with a symmetric monoidal structure $\otimes$. By definition, the canonical map $\Sigma_+^{\infty}\C^{\times} \to \Kb_{\C}$ is a map of $E_\infty$-ring spectra.

For $\C = P_f(R)$ the symmetric monoidal exact category of finitely generated projective $R$-modules, we have a symmetric monoidal morphism
$BR^{\times} \to (P_f(R))^{\times}.$
It is obtained by viewing $(BR^{\times},\otimes)$ as the symmetric monoidal category of free $R$-modules of rank $1$. Therefore we have a morphism $BR^{\times} \to P_f(R)^{\times}$ of $E_{\infty}$-objects in (unpointed) spaces. By virtue of Corollary \ref{cor:monoidtorings} we obtain a morphism of $E_{\infty}$-ring spectra
$$\Sigma_+^{\infty} BR^{\times} \to \Sigma_+^{\infty}P_f(R)^{\times}.$$

For $(r_1,\dots,r_n) \in R^{\times}$ the resulting map $B\mathbb{Z}^n \to BR^{\times}$ is symmetric monoidal, and therefore, another application of Corollary \ref{cor:monoidtorings} yields a morphism of $E_{\infty}$-ring spectra
$$\Sigma^{\infty}_+B\mathbb{Z}^n \to \Sigma^{\infty}_+BR^{\times}.$$
Composing the morphism of $E_{\infty}$-ring spectra defined above, we obtain
$$\psi\colon \Sigma_+^{\infty}B\mathbb{Z}^n \to \Sigma_+^{\infty}P_f(R)^{\times}\to \Kb_R.$$
Recall that we have a functor from the homotopy category of $E_{\infty}$-ring spectra to the category of graded commutative rings.
This implies the equality
$\psi_*(e_1 | \cdots | e_n) = (r_1|\dots|r_n),$
where we denote by $(e_1,\dots,e_n) \in \mathbb{Z}^n$ the standard basis of $\mathbb{Z}^n$.
By definition of higher commutators, the left hand side agrees with $r_1 \ast \cdots \ast r_n$. The right hand side on the other hand is given by Loday's higher Steinberg symbol $\{r_1,\dots,r_n\}$. This concludes the proof.
\end{proof}



\section{The CC symbol via Tate categories}\label{symbols}

Now we are almost ready give the full definition of our higher Contou-Carr\`ere symbol, pursuing the strategy which had called \textit{Idea 2} in the introduction.

\subsection{Lattices and Tate objects}

\subsubsection{Tate Objects in Exact Categories}\label{subsub:Tate_objects}

We recall the constructions of Ind, Pro, and Tate objects in exact categories, and refer the reader to \cite{MR3510209} for background. The ideas of these constructions go back to papers by Beilinson \cite{MR923134} and Kato \cite{MR1804933}, and have also been studied by Previdi in \cite{MR2872533}. We also refer the reader to Drinfeld's theory of Tate $R$-modules introduced in \cite{MR2181808}.

A \emph{filtered set} $I$ is a set $I$ together with a partial ordering $\leq$, such that for each pair $(i,j) \in I^2$ there exists a $k \in I$, satisfying $i \leq k$ and $j \leq k$. Every filtered set can be viewed as a category in a straightforward manner.

Let $\C$ be an exact category. An \emph{admissible} Ind-object in $\C$ indexed by $I$ is a functor $X\colon I \to \C$, such that the relation $i \leq j$ determines an admissible monomorphism with respect to the exact structure of $\C$. For example, we can take $I$ to be the set $\mathbb{N}$ with its natural ordering. An $\mathbb{N}$-indexed admissible Ind-object in $\C$ can then be pictured as a formal colimit of a diagram
\begin{equation}\label{eqn:ind-obj}
X_0 \hookrightarrow X_1 \hookrightarrow X_2 \hookrightarrow \cdots .
\end{equation}
Every admissible Ind-object gives rise to a left exact presheaf. To $X\colon I \to \C$ one associates the presheaf
$A \mapsto \colim_{i \in I} \Hom(A,X(i)).$
The resulting full subcategory of $\Lex(\C)$ of all objects of this shape is denoted by $\Ind^a(\C)$. In Theorem 3.7 of \cite{MR3510209} the authors showed that $\Ind^a(\C)$ is an extension closed subcategory of $\Lex(\C)$. This implies that it inherits a structure of an exact category.

Admissible Pro-objects in $\C$ are defined dually, i.e. by replacing the role of admissible monomorphisms by admissible epimorphisms. In short we have,
$\Pro^a(\C) = (\Ind^a(\C^{\op}))^{\op}$. An admissible Pro-object indexed by a filtered set $I$ is a functor $X\colon I^{\op} \to \C$, which sends $i \leq j$ to an admissible epimorphism in $\C$. For $I = \mathbb{N}$ we obtain the dual depiction of a Pro-object as a formal limit of a diagram
\begin{equation}\label{eqn:pro-obj}
X_0 \twoheadleftarrow X_1 \twoheadleftarrow X_2 \twoheadleftarrow \cdots .
\end{equation}

An \emph{elementary Tate object} is an admissible Ind-Pro-object, i.e. an object $V$ in $\Ind^a \Pro^a(\C)$, which can be (non-canonically) written as an extension
\begin{equation}\label{eqn:tate}
L \hookrightarrow V \twoheadrightarrow V/L,
\end{equation}
with $L \in \Pro^a(\C)$ and $V/L \in \Ind^a(\C)$. We refer to any such $L$ as a \emph{lattice} in $V$. The category of elementary Tate objects in $\C$ has a natural exact structure (Theorem 5.4 in \cite{MR3510209}), and will be denoted by $\elTate(\C)$.

\begin{proposition}[Kapranov]
If $k$ is a field, then $\elTate(\mathsf{Vect}_k^{fd})$, i.e. the exact category of elementary Tate objects of finite-dimensional $k$-vector spaces, is equivalent to the category of locally linearly compact topological $k$-vector spaces (as exact categories).
\end{proposition}
See \cite[\S 1.1.2]{KapranovSemiInfinite}.


The exact category $\Tate(\C)$ of \emph{Tate objects} in $\C$ is defined to be the idempotent completion of $\elTate(\C)$. If $R$ is a ring, and $\C = P_f(R)$, the exact category of finitely-generated projective $R$-modules, then $\Tate(P_f(R))$ contains Drinfeld's category of Tate $R$-modules as a full subcategory. See \cite[Thm. 5.26]{MR3510209}, where we show that for countable index sets $I$, the two categories are in fact equivalent. We emphasize that in \cite{MR2181808}, Drinfeld refers to what we call lattices as \emph{co-projective lattices}.

\begin{definition}\label{defi:grp}
For a category $\D$ (respectively $\infty$-category), we denote by $\D^{\grp}$ the maximal groupoid contained in $\D$ (respectively $\infty$-groupoid).
\end{definition}

The following result is \cite[Prop. 3.3]{Braunling:2014vn}.
\begin{proposition}\label{prop:gr}
For an idempotent complete exact category $\C$, we denote by $\Gr^{\leq}_{\bullet}(\C)$ the simplicial object in groupoids, which parametrizes chains $(V \supset L_n \supset \cdots \supset L_0)$, where $V$ is an elementary Tate object in $\C$, and each $L_i$ is a lattice in $V$. We have a forgetful morphism $\Gr^{\leq}_{\bullet}(\C) \to \elTate(\C)^{\grp}$, which induces an equivalence $|\Gr^{\leq}_{\bullet}(\C)| \to^{\cong} \elTate(\C)^{\grp}$.
\end{proposition}

\subsubsection{The Index Map}\label{index}

Let $\C$ be an exact category, following Waldhausen \cite{MR802796} we denote by $S_{n}(\C)$ the exact category, whose objects correspond to chains
$$X_1 \hookrightarrow \cdots \hookrightarrow X_n$$
of admissible monomorphisms (plus a fixed choice for all possible quotients among these objects). The $S_n(\C)$ fit together to give a simplicial object $S_{\bullet}(\C)$ in the $2$-category of exact categories: face maps are given by omitting an object/composing maps, and degeneracies by inserting the identity map.

Waldhausen's treatment of algebraic $K$-theory in \cite{MR802796} implies that, for an exact category $\C$, the classifying space $BK_{\C}$ is equivalent to the geometric realization of the simplicial object in groupoids $|S_{\bullet}\C^{\grp}|$.

Now let $\C$ be an idempotent complete exact category, and let $\Gr_\bullet^\le(\C)$ be as in Proposition \ref{prop:gr}.
\begin{definition}\label{defi:index}
    Let $\Index \colon \Gr^{\leq}_{\bullet}(\C) \to S_{\bullet}(\C)$ be the map sending $(V \supset L_n \supset \cdots \supset L_0)$ to $(L_1/L_0 \hookrightarrow \cdots \hookrightarrow L_n/L_0)$. Whenever convenient,
		\begin{itemize}
		\item the geometric realization $\elTate(\C)^{\grp} \to BK_{\C}$, as well as
		\item the induced map $K_{\elTate(\C)}\to BK_{\C}$ (see \cite[Cor. 3.5]{Braunling:2014vn})

		\end{itemize}
		will also be denoted $\Index$ and called the \emph{index map} as well.
\end{definition}
This is the map which we had alluded to in the introduction of the paper, see Equation \eqref{e_c6}.

For every elementary Tate object $V$, we obtain from $$B\Aut(V) \to^{\alpha } \elTate(\C)^{\grp} \to^{\Index} BK_{\C},$$ a monoidal map
$$\Aut(V) \to K_{\C},$$
by applying the loop space functor $\Omega$. Above, the map $\alpha$ is the one coming from the construction of Remark \ref{rmk:Ccc4}.

\subsection{The classical Contou-Carr\`ere Symbol}

\subsubsection{The Contou-Carr\`ere Symbol}

We had recalled the classical tame symbol in Equation \eqref{lsi1}. The Contou-Carr\`ere symbol arises as a ``deformation" of the tame symbol for the discrete valuation ring $R=k((t))$. For every (commutative) $k$-algebra $A$, we can consider the ring of formal Laurent series $A((t))$, which is almost never a discrete valuation ring. Nonetheless, there exists a natural pairing $A((t))^{\times} \times A((t))^{\times} \to A^{\times}$, which specializes to the tame symbol for the case $A = k$. For $A\neq k$, the explicit formula \eqref{lsi1} no longer holds. However the interpretation of the tame symbol as a graded commutator \cite{MR1013132} remains valid for Contou-Carr\`ere symbols by work of Anderson--Pablos Romo \cite{MR2036223} and Beilinson--Bloch--Esnault \cite{MR1988970}. We hence begin by summarizing the definition using graded commutators.

We denote by $\Vv_A$ the Tate object $A((t))$ in $\Tate(P_f(A))$, see \cite[Example 10]{MR3621099} for a precise definition. There is a natural map
$A((t))^{\times} \to \Aut(\Vv_A)$. Let $\mathbb{B}^{\mathbb{Z}}\G_m(A)$ denote the spectrum associated to the Picard groupoid of graded $A$-lines. For each $A$, the index map and determinant give rise to a spectral extension
\begin{equation}\label{spectralKM}
    \Sigma^{\infty}_+B(A((t))^{\times}) \to \Sigma_+^{\infty}B\Aut(\Vv_A) \to \Kk_{\elTate(P_f(A))} \to^{\Index} \Sigma\Kk_{P_f(A)} \xrightarrow{\det} \Sigma\mathbb{B}^{\mathbb{Z}}\G_m(A).
\end{equation}
Looping the adjoint of this map yields an $E_1$-map
\begin{equation}\label{KMext}
A((t))^{\times} \to B^{\mathbb{Z}}\G_m(A),
\end{equation}
classifying a graded central extension of $A((t))^{\times}$. The construction is natural in maps $A\to A'$, so it defines a central extension of group-valued sheaves. We record this observation in the following definition. Recall that the loop group $L\G_m$ is defined as the group-valued presheaf
$$L\G_m \colon (\Aff_k)^{\op} \to \Grp$$
sending $A$ to $A((t))^{\times}$.

\begin{definition}\label{defi:KacMoody}
The graded central extension \eqref{KMext} of $L\G_m$ will be denoted by $$\phi_{KM} \colon L\G_m \to B^{\mathbb{Z}}\G_m,$$ and referred to as the \emph{Kac--Moody extension} of the loop group. We denote the \emph{spectral Kac--Moody extension} \eqref{spectralKM} by
$$e_{sKM}\colon \Sigma^{\infty}_+BL\G_m \to \Sigma\Kk,$$
where $\Kk$ denotes the presheaf in connective spectra, sending a ring $A$ to $\Kk_A$.
\end{definition}
Note that the Kac-Moody extension is obtained from the spectral Kac-Moody extension by looping and applying the determinant.

We can now recall the following well-known result, which generalizes the main result of the paper \cite{MR2036223} to arbitrary $k$-algebras (without restricting to the artinian case).

\begin{proposition}\label{prop:CC_commutator}
The graded central extension $\phi_{KM}$ of Definition \ref{defi:KacMoody} relates to the Contou-Carr\`ere symbol by means of the relation
$$(-,-)^{-1} = - \star_{\phi_{KM}} -.$$
\end{proposition}

\begin{proof}
Proposition 3.3 of \cite{MR1988970} verifies that the classical notion of the Kac--Moody extension of loop groups has this property. In \cite[\S 5.3, and Prop. 5.3]{Braunling:2014vn} we compare the extension $\phi_{KM}$ with its classical definition in terms of determinant lines.
\end{proof}

\subsection{Higher Contou-Carr\`ere Symbols}\label{sub:higher}

We begin this Subsection with a definition, in order to avoid the cumbersome notation $A((t_1))\cdots((t_n))^{\times}$.

\begin{definition}\label{defi:k-fold_loop_group}
The $n$-fold loop group $L^nG$ of a group-valued presheaf $G$ is defined to be the group-valued presheaf which sends the affine scheme $\Spec A$ to $G(A((t_1))\cdots((t_n)))$.
\end{definition}

There is an analogue of the Kac--Moody extension for loop groups. Denoting by $\Vv_A^n$ the $n$-Tate object $A((t_1))\dots((t_n))$ in $\nTate(P_f(A))$, we have a natural map
$$L^n\G_m(A) \to \Aut(\Vv_A^n)$$
for every $k$-algebra $A$. The index map gives rise to a spectral extension
\begin{equation}\label{spectralcan}
\Sigma^{\infty}_+BL^n\G_m(A) \to \Sigma^{\infty}_+B\Aut(\Vv_A^n) \to \Kb_{\nTate(P_f(A))} \to^{\Index^n} \Sigma^{n}\Kb_{P_f(A)}
\end{equation}
of $L^n\G_m(A)$ by $\Sigma^{n-2}\Kb_{P_f(A)}$. As above, the construction is natural in maps $A\to A'$, so it defines a central extension of sheaves in groups.

\begin{definition}\label{defi:higher_Kac_Moody}
The spectral extension \eqref{spectralcan} of $L^n\G_m$ by $\Sigma^{n-2}\Kb$ will be referred to as the {\em canonical spectral extension} of the $n$-fold loop group $L^n\G_m$. We denote the corresponding map of spectra by $e_n$.
\end{definition}

As an application of this construction we give a definition of higher Contou-Carr\`ere symbols.

\begin{definition}\label{defi:higher_CC}
Let $f_0,\dots,f_n \in L^n\G_m(A) = A((t_1))\dots((t_n))^{\times}$. We denote by $\det$ the determinant map $K_1(A) \to A^{\grp}$. The Contou-Carr\`ere symbol $(f_0,\dots,f_n)$ is defined to be the higher commutator
$$\det((f_0\star \dots \star f_n)_{e_n}^{(-1)^n}).$$
\end{definition}

The study of the higher Contou-Carr\`ere symbol $(f_0,\dots,f_n)$ for an $(n+1)$-tuple in $A((t_1))\dots((t_n))$, with $A$ a $k$-algebra, has been pioneered by Osipov--Zhu in the case of $n = 2$ (see \cite{Osipov:2013fk}). They identified this symbol with a higher commutator in a central extension of the double loop group $L^2\G_m$ by $B^{\mathbb{Z}}\G_m$. Inspired by this observation and the one-dimensional case (Proposition \ref{prop:CC_commutator}), they define the two-dimensional Contou-Carr\`ere symbol for general $k$-algebras $A$ as a higher commutator $C_3(f,g,h)$.

\begin{proposition}\label{prop:OZ_CC}
    Definition \ref{defi:higher_CC} is compatible with the definition of Contou-Carr\`ere in dimension $1$, and Osipov--Zhu in dimension $2$.
\end{proposition}

The proof of the $1$-dimensional case was the content of Proposition \ref{prop:CC_commutator}. We now turn to verifying the assertion for $n = 2$.

\begin{proof}[Proof of the $2$-dimensional case:]
Osipov--Zhu construct a central extension of $L^2\G_m$ by the Picard groupoid $B^{\mathbb{Z}}\G_m$ (\cite[p. 28]{Osipov:2013fk}), and define $(f,g,h)$ for a triple in $A((t_1))((t_2))^{\times}$, as the higher commutator $C_3(f,g,h)$. We have seen in Proposition \ref{prop:OZ_C3} that $f \star g \star h = C_3(f,g,h)$. So to conclude the assertion, we need to verify that for $n = 2$ the spectral extension of $L^n\G_m$ constructed in Definition \ref{defi:higher_Kac_Moody} is related to the extension
$$\Sigma_+^{\infty}BL^2\G_m \to \Sigma^2\mathbb{B}^{\mathbb{Z}}\G_m.$$
constructed by Osipov--Zhu.

By Nisnevich descent, it suffices to consider rings $A$ with $K_{-1}(A) = 0$. We then have a commutative diagram
\[
\xymatrix{
\Sigma_+^{\infty}BL^2\G_m(A) \ar[r] \ar[rd]_{e_2} & \Sigma^2\Kk_{A} \ar[d] \ar[r]^-{\det^{\mathbb{Z}}} & \Sigma^2\mathbb{B}^{\mathbb{Z}}\G_m(A) \\
& \Sigma^2 \Kb_{A}. &
}
\]
Using the adjunction between $\Sigma_+^{\infty}$ and $\Omega^{\infty}$, we obtain a map
$$e_{2}\colon BL^2\G_m \to B^2B^{\mathbb{Z}}\G_m.$$
Picking a basepoint in $BL^2\G_m$ and looping once yields an $E_1$-map to the classifying space of the Picard groupoid of graded lines $B^{\mathbb{Z}}\G_m$
$$\phi\colon L^2\G_m \to BB^{\mathbb{Z}}\G_m.$$
We have to show that this morphism is $-1$ times of the one constructed by Osipov--Zhu. According to \cite[Prop. 3.28 \& Thm. 3.31]{BGWSegal}, $\phi$ sends $f \in L^2\G_m(A)$ to
$$\det{}^{\mathbb{Z}}(N/f L)  \otimes \det{}^{\mathbb{Z}}(N/L)^{\vee},$$
for $N$ a lattice containing both $f L$ and $L$, with the monoidal structure being defined in terms of common enveloping lattices. This is precisely the dual of the definition given by Osipov--Zhu \cite[p. 28]{Osipov:2013fk}.
\end{proof}

The comparison of the generalized Contou-Carr\`ere symbol with the classical cases in dimension $1$ and $2$ already shows that our definition produces a non-trivial map in these dimensions. We will explain why this is also the case in general.

\begin{rmk}
Let $k \subset k'$ be a field extension and $A \to^{\phi} k'$ a ring homomorphism. Since our constructions are functorial in the $k$-algebra $A$ we see that
$$\phi\left((f_0,\dots,f_n)\right) = \left(\phi(f_0),\dots,\phi(f_n)\right)$$
for $f_0,\dots,f_n \in L^n\G_m(A)$. If we choose $f_i = t_{i+1}$ for $i = 0,\dots,n-1$ and $f_n \in k^{\times}$ we obtain $\phi((f_0,\dots,f_n)) = \phi(f_n)$. This follows from Corollary \ref{cor:CCtame} below, which asserts that the higher Contou-Carr\`ere symbol for $A$ a field agrees with the tame symbol.
\end{rmk}

After these preparations, let us return to geometry. Let $X$ be a Noetherian $k$-scheme and $\xi$ a saturated flag of integral closed subschemes $Z_i$. Moreover, suppose we are given a $k$-algebra $A$. Equipped with this data, we defined objects $F_{X,\xi}$ and $A_{X,\xi}$ in \S \ref{subsub:reminder_adeles}.

By Theorem 7.10 of \cite{MR3510209}, the object $F_{X,\xi}$ carries a canonical structure of a higher Tate object. In particular, we see that  $F_{X,\xi}$ gives rise to an $n$-Tate object $\underline{F}_{X,\xi}$ in the abelian category $\Coh_{Z_0}(X)$ (coherent sheaves on $X$, set-theoretically supported at $Z_0$). If $X$ is defined over a field $k$, then, because $Z_0$ is $0$-dimensional, global sections give rise to an exact functor
$$\Gamma(X,-)\colon \nTate(\Coh_{Z_0} X) \to \nTate(k).$$
Thus, $F_{X,\xi}$ gives rise to an $n$-Tate object in the category of finite-dimensional vector spaces over $k$. If $A$ is an arbitrary $k$-algebra, the tensor product $- \otimes_k A\colon \Vect_f(k) \to P_f(A)$ determines an exact functor
$$- \widehat{\otimes}_k A\colon \nTate(k) \to \nTate(A).$$
\begin{definition}\label{defi:AXxi}
Let $X$, $\xi$, $k$, and $A$ be as described earlier. We define
$$\underline{A}_{X,\xi} = F_{X,\xi} \widehat{\otimes}_k A. $$
The $A$-module underlying $\underline{A}_{X,\xi}$ (via the forgetful functor $\nTate(A)\to\Mod(A)$) inherits a $k$-algebra structure from $F_{X,\xi}$; we denote this $k$-algebra by $A_{X,\xi}$. For a group scheme $G$ over $k$, we define the iterated loop group at $(X,\xi)$ to be the group-valued presheaf given by
$$L^n_{X,\xi}G(A) = G(A_{X,\xi}).$$
\end{definition}

By definition, we have $L^n_{X,\xi}\G_m = F_{X,\xi}^\times$. 
\begin{example}
If $X = \mathbb{A}^n_k= \Spec k[t_1,\dots,t_n]$, and $Z_k = \Spec k[t_1,\dots,t_k]$, then we have $A_{X,\xi} = A((t_1))\dots((t_n))$, and $L^n_{X,\xi}\G_m \cong L^n\G_m$.
\end{example}

Note that for any ring $R$, the exact category of finitely-generated projective modules $P_f(R)$ is the idempotent completion of the exact category of finitely-generated free $R$-modules. Therefore, any exact functor $\phi\colon P_f(R) \to \C$, into any idempotent complete exact category $\C$, is determined by $\phi(R)$ up to equivalence.

\begin{definition}[Spectral Contou-Carr\`ere symbol]\label{defi:spectral_CC}
Let $T\colon P_f(A_{X,\xi}) \to \nTate(A)$ be the unique functor sending $A_{X,\xi}$ to $\underline{A}_{X,\xi}$. The composition
$$\sigma_{X,\xi}^A\colon\Kb_{A_{X,\xi}} \xrightarrow{T} \Kb_{\nTate(A)} \xrightarrow{(-1)^n\Index^n} \Sigma^n\Kb_A$$
will be referred to as the \emph{spectral Contou-Carr\`ere symbol}.
\end{definition}

Replacing $K$-theory by $G$-theory (i.e. working with all coherent sheaves instead of only locally free ones), we obtain an analogous \emph{spectrification} of the tame symbol.

\begin{definition}\label{defi:spectral_tame}
Let $T\colon P_f(F_{X,\xi}) \to \nTate(\Coh_{Z_0}(X))$ be the unique functor sending $F_{X,\xi}$ to $\underline{F}_{X,\xi}$. The composition
$$\sigma_{X,\xi}\colon\Kb_{F_{X,\xi}} \xrightarrow{T} \Kb_{\nTate(\Coh_{Z_0}(X))} \xrightarrow{(-1)^n\Index^n} \Sigma^n\Gb_{X,Z_{0}} \xrightarrow{\cong} \Sigma^n \Kb_{Z_0} \to \Sigma^n\Kb_k$$
will be referred to as the \emph{spectral tame symbol}.
\end{definition}

Switching back to the Contou-Carr\`ere setup, we can use higher commutators to extract Contou-Carr\`ere symbols from the morphism of spectra in Definition \ref{defi:spectral_CC}.

\begin{definition}\label{defi:CC_flag}
Denote by $\det\colon K_1(A)\to A^\times$ the map induced by the determinant of matrices. For an $n+1$-tuple $f_0,\dots,f_n \in A_{X,\xi}^{\times}$ we define
$$(f_0,\dots,f_n) = \det((f_0\star \dots \star f_n)_{\sigma^A_{X,\xi}}),$$
and refer to this expression as the \emph{Contou-Carr\`ere symbol of $X$ at $\xi$}.
\end{definition}

\section{Comparison of both definitions}\label{symbolsviaboundarymap}

This section is devoted to linking higher Contou-Carr\`ere symbols to their classical counterparts.

%
%

\subsection{$K$-theory and Tate categories}\label{sub:delooping}

\subsubsection{For Exact Categories}

Schlichting developed a localization theorem for exact categories in \cite{MR2079996}, which states that, for every left (respectively right) s-filtering inclusion of exact categories
$\C \hookrightarrow \D,$
the quotient category $\D/\C$ carries a natural structure as an exact category. Moreover, attaching their associated stable $\infty$-categories to them as in Definition \ref{defi:perf},
\begin{equation}\label{eq14}
\xymatrix{
\Perf(\C) \ar@{^(->}[r]^-i & \Perf(\D) \ar@{->>}[r]^-q & \Perf(\D/\C)
}
\end{equation}
becomes an exact sequence of stable $\infty$-categories. Further, if $\C$ is idempotent complete, then by Propositon \ref{prop:char_K} (3), we obtain a bi-cartesian square of spectra
\[
\xymatrix{
\Kb_{\C} \ar[r]^i \ar[d] & \Kb_{\D} \ar[d]^q \\
0 \ar[r] & \Kb_{\D/\C},
}
\]
where the relation $\Kb_{\C} = \Kb_{\Perf(\C)}$ holds essentially by definition.
Schlichting observed that if one chooses $\D$ such that $\Kb_{\D} \cong 0$, then the boundary morphism of this square gives an equivalence
$\partial\colon \Kb_{\D/\C} \cong \Sigma \Kb_{\C}.$
Propositon \ref{prop:char_K} (1) guarantees that $\Kb_{\Inda(\C)} \cong 0$ for every exact category $\C$. Thus, we see that, for $\C$ idempotent complete, we have a canonical delooping
$\Kb_{\Inda(\C)/\C} \cong \Sigma \Kb_{\C}.$
Using similar techniques, Saito establishes an abstract equivalence $\Kb_{\Tate(\C)} \cong \Sigma \Kb_{\C}$ in \cite{Saito:2012fk}. In fact, this equivalence can be constructed as the composition
\begin{equation} \label{eq:ccx1}
\varphi : \Kb_{\Tate(\C)} \cong \Kb_{\elTate(\C)} \cong \Kb_{\elTate(\C)/\Proa(\C)} \cong \Kb_{\Inda(\C)/\C},
\end{equation}
followed by
$$\Kb_{\Inda(\C)/\C} \cong \Sigma\Kb_{\C}.$$
In the first row, the first equivalence follows from the cofinal invariance of non-connective $K$-theory (i.e. (2) of Proposition \ref{prop:char_K}). The second map is an equivalence as a corollary of the aforementioned localization theorem, and the third equivalence exists already on the level of exact categories (e.g. \cite[Prop. 5.32]{MR3510209}).

The index map of Definition \ref{defi:index} is an explicit description of these boundary maps. See \cite[Thm. 3.6]{Braunling:2014vn} for the proof.

\begin{theorem}\label{thm:index_boundary}
Let $\C$ be an idempotent complete exact category. The exact equivalence of exact categories $q:\elTate(\C)/\Pro(\C) \cong \Inda(\C)/\C$ (see \cite[Prop. 5.32]{MR3510209}) induces a commutative diagram
\[
\xymatrix{
\Sigma_+^{\infty}{\elTate(\C)^{\grp}}\ar[rr]^(.6){\Index} \ar[d]_{q} && \Sigma\Kb_{\C} \\
\Kb_{\Ind^a(\C)/\C}, \ar[urr]_{-\partial} &
}
\]
where both $\mathsf{Index}$ and $\partial$ arise as the boundary maps of the localization sequences discussed above.
\end{theorem}

This theorem motivates the following definition of the \emph{non-connective} index map.

\begin{definition}\label{defi:index_boundary}
For an idempotent complete exact category $\C$, we define the map $\Index\colon \Kb_{\Tate(\C)} \to \Sigma \Kb_{\C}$ as the composition so that the diagram
\[
\xymatrix{
\Kb_{\Tate(\C)}\ar[r]^{\Index} \ar[d]_{\varphi } & \Sigma\Kb_{\C} \\
\Kb_{\Ind^a(\C)/\C}. \ar[ur]_{-\partial} &
}
\]
commutes, where $\varphi$ is the map of Equation \eqref{eq:ccx1}.
\end{definition}

\subsubsection{Suspension and Calkin Objects for Stable $\infty$-Categories}

Let $\C$ be a stable $\infty$-category, and $\kappa$ an infinite cardinal. Recall Definition \ref{defi:suspension}, which defines the \emph{suspension} $\Ss_{\kappa}(\C)$ as the localization
$$\Ss_{\kappa}(\C) = \Ind_{\kappa}(\C)/\C,$$
and which defines the $\infty$-category of Calkin objects $\Calk(\C)$ as the idempotent completion of the suspension.

Since non-connective $K$-theory cannot distinguish between a category and its idempotent completion (see (2) of Proposition \ref{prop:char_K}):
$\Kb_{\Calk_{\kappa}(\C)} \cong \Kb_{\Ss_{\kappa}(\C)} \cong \Sigma\Kb_{\C},$
we will often omit the cardinal $\kappa$ from our notation. Following Schlichting \cite{MR2079996}, Blumberg--Gepner--Tabuada \cite{MR3070515} observed the following delooping property for $K$-theory introduced in \eqref{eqn:ncK}.

\begin{proposition}\label{prop:suspension}
The boundary map $\partial$ of the localization sequence of the exact sequence $$\C \hookrightarrow \Ind(\C) \twoheadrightarrow \Ss(\C)$$ of stable $\infty$-categories, induces an equivalence of non-connective $K$-theory spectra $\partial\colon\Kb_{\Calk(\C)} \cong \Sigma\Kb_{\C}$.
\end{proposition}

This result serves as a motivation to call $\Ss(\C)$ the suspension of $\C$. Recall that the suspension of a topological space $X$ is formed by embedding $X$ into the cone $CX$, which is \emph{contractible}. The resulting homotopy cofibre, obtained by taking the quotient space, is one possible incarnation of the suspension.  By analogy, Schlichting embeds a category $\C$ into an ambient $K$-contractible category $\Ind(\C)$, and takes the quotient to obtain the categorical suspension. A second possibility is to construct the suspension of $X$ by glueing a second copy $C^{-}X$ to the cone $CX$ along the common subspace $X$. Since $C^-X$ is contractible, this yields a homotopy equivalent space. Categorically this is analogous to pasting the $K$-contractible categories $\Inda(\C)$ and $\Proa(\C)$ along the common subcategory $\C$. This is the underlying idea of Saito's delooping statement. For later use, we record the following naturality property.

\begin{lemma}[Naturality]\label{lemma:S_naturality}
For every idempotent complete exact category $\C$, there exists a commutative diagram
\[
\xymatrix{
\Kb_{\Inda(\C)/\C} \ar[d]_{\partial } \ar[r]^-{\cong} & \Kb_{\Calk(\Perf(\C))} \ar[d]^{\partial } \\
\Sigma \Kb_{\C} \ar[r]^{\cong} & \Sigma\Kb_{\Perf(\C)}
}
\]
of spectra, where the horizontal maps $\Kb_{\C} \cong \Kb_{\Perf(\C)}$ are the equivalences stipulated by Lemma \ref{lemma:exact_comparison}, and the vertical maps are the equivalences coming from the boundary maps of the localization sequence discussed above.
\end{lemma}

\subsection{Comparison}

\begin{theorem}\label{thm:CCboundary}
Let $X$ be a Noetherian $k$-scheme, and $\xi$ a saturated flag of closed subschemes $Z_i$. For every $k$-algebra $A$, the spectral Contou-Carr\`ere symbol $\sigma_{X,\xi}^A$ of Definition \ref{defi:spectral_CC} agrees with the $n$-fold delooping of the preliminary Contou-Carr\`ere symbol of Definition \ref{def:prelimCCsymbol}.
\end{theorem}

The proof of this result will be given in the next paragraph. It uses the concept of realization functors which we will subsequently introduce. We conclude:

\begin{corollary}\label{cor:CCtame}
For $A$ a field, the Contou-Carr\`ere symbol agrees with the higher tame symbol for algebraic $K$-theory.
\end{corollary}

\subsection{Contou-Carr\`ere Symbols and Realization Functors}\label{sub:real}

\begin{definition}[Tate realization]\label{defi:tate_real}
Let $X$ be a Noetherian scheme, and $j\colon U \hookrightarrow X$ an open immersion, with complement denoted by $Z$. Let $W \supset Z$ be a closed subscheme of $X$, such that the open immersion $U \cap W \hookrightarrow W$ is affine.\footnote{In Lemma \ref{lemma:affineW} (a) we show that this is condition only depends on the underlying closed subspace $|W|$.} Then, we have exact functors $\ind$, $\pro$, and $\tate$, defined as follows.
\begin{itemize}
\item[(a)] The functor $\ind\colon \Coh_{|W| \cap U}(U) \to \Ind^a(\Coh_{|W|}(X))$ sends $\F \in \Coh_{|W|\cap U}(U)$ to $j_*\F$, viewed as an ascending union of coherent sheaves on $X$ with set-theoretic support in $|W|$.
\item[(b)] We denote by $i_n\colon Z^{(n)} \to X$ the inclusion of the $n$-th order infinitesimal neighbourhood of $Z$. We define $\pro\colon \Coh_W(X) \to \Pro^a(\Coh_{|Z|}(X))$ to be the functor sending $\F \in \Coh_{|W|}(X)$ to the Pro-system $\left((i_n)_*i_n^*\F\right)_{n \in \mathbb{N}}$.
\item[(c)] Combining (a) and (b) we obtain a functor $$\tate\colon \Coh_{|W|}(U) \to  \Ind^a\Pro^a(\Coh_{|Z|}(X)).$$
\end{itemize}
\end{definition}

\begin{rmk}\label{rmk:tate_real}
One can check that the functor $\tate$ of Definition \ref{defi:tate_real}(c) factors through $\elTate(\Coh_Z(X)) \subset \Ind^a\Pro^a(\Coh_Z(X))$. Indeed, for every $\F \in \Coh_{|W| \cap U}(U)$ we have a $4$-term exact sequence
$$\ker(\pro(\F) \to \tate(\F)) \hookrightarrow \pro(\F) \to \tate(\F) \twoheadrightarrow j_*j^*\F/\F.$$
The kernel on the left hand side is equivalent to $\ker(\F \to j_*j^*\F) \in \Coh_Z(X)$, hence the quotient of $\pro(\F)$ by this object lies again in $\Pro^a(\Coh_Z(X))$. The object on the right hand side $j_*j^*\F/\F$ lies in $\Ind^a(\Coh_Z(X))$. This allows us to represent $\tate(\F)$ as an extension of an admissible Ind-object by an admissible Pro object. Hence, $\tate(\F)$ has a lattice, i.e. is an elementary Tate object.
\end{rmk}

Definition \ref{defi:tate_real} contains the condition that the inclusion $U \cap W \hookrightarrow W$ is affine. It is important to note the following two observations.

\begin{lemma}\label{lemma:affineW} We have the following:
\begin{itemize}
\item[(a)] Let $W,W' \hookrightarrow X$ be closed subschemes of a separated Noetherian scheme $X$, satisfying $|W| = |W'|$. For an open subscheme $U \subset X$ we have that
$W \cap U \hookrightarrow W$ is affine if and only if $W' \cap U \hookrightarrow W'$ is affine.

\item[(b)] Let $W \hookrightarrow X$ be a closed immersion into a separated Noetherian scheme $X$, and $U \subset X$ an open subscheme with closed complement denoted by $|Z|$. If $|Z| \subset |W|$, and $\dim Z = \dim W - 1 = 0$, then the inclusion $W \cap U \hookrightarrow W$ is \emph{affine}.
\end{itemize}
\end{lemma}

\begin{proof}
Assertion (a) follows from the fact that a scheme is affine if and only if the underlying reduced scheme is affine (which is a consequence of Chevalley's theorem, see \cite{MR2356346}). To verify (b), we observe that $Z = \{z_1,\dots,z_k\}$ is a discrete subset consisting of closed points (since it is of dimension $0$), and we may therefore replace $W$ without loss of generality by $\Spec (\Oo_{W,z_1} \times \cdots \times \Oo_{W,z_k})$. Then, the complement $W \cap U = W \setminus Z$ agrees with the discrete subset $\{\eta_1,\dots,\eta_m\}$, where each $\eta_i$ is a generic point of the one-dimensional scheme $W$. Since each of the inclusions $\{\eta_i\} \hookrightarrow X$ is affine, the same is true for
$$W \cap U = \coprod_{i = 1}^m\{\eta_i\} \hookrightarrow X,$$
since a finite coproduct of affine schemes is affine.
\end{proof}

We call these functors \emph{realization functors}, since they associate to a coherent sheaf on $U = X \setminus Z$ a Tate object in $\Coh_Z(X)$. For our purposes it will be necessary to have similar functors for \emph{perfect complexes} on $U$ at our disposal. This is achieved by the following definition. We denote the derived $\infty$-category of pseudo-coherent complexes of sheaves (resp. complexes of quasi-coherent sheaves) on a scheme $X$ by $\mathsf{DCoh}(X)$ (resp. $\QC(X)$).

\begin{definition}[Calkin realization]\label{defi:calkin_real}
Let $X$, $U$, $Z$ be quasi-compact and quasi-separated schemes, with $j\colon Z \hookrightarrow X$ a closed embedding, and $U = X \setminus Z$.
\begin{itemize}
\item[(a)]We denote by
$$\ind\colon \Perf(X) \to \Ind(\Perf_Z(X)) \cong \QC_Z(X)$$
the functor given by $\F \mapsto \mathrm{fib}(\F \to j_*j^*\F)$.
\item[(b)] Let $\calk\colon \Perf(U) \to \Calk(\Perf_Z(X))$ be the functor induced by $\ind$: $$\Perf(U)\simeq \left(\Perf(X)/\Perf_Z(X)\right)^{\ic} \to \left(\Ind(\Perf_Z(X))/\Perf_Z(X)\right)^{\ic}.$$
\item[(c)] The functors $\ind$ and $\calk$ have a version for the stable $\infty$-categories of pseudo-coherent complexes of sheaves:
$$\ind\colon \mathsf{DCoh}(X) \to \Ind(\mathsf{DCoh}_Z(X)),$$
and
$$\calk\colon \mathsf{DCoh}(U) \to \Calk(\mathsf{DCoh}_Z(X)).$$
\end{itemize}
\end{definition}

\begin{lemma}\label{lemma:calk_basechange}
Let $X$, $U$, and $j\colon Z\hookrightarrow X$ be as in Definition \ref{defi:calkin_real}. Let $X \to W$ be a morphism of schemes. For an affine flat morphism $f\colon V \to W$ we denote the base changes $X \times_W V$, $U \times_W V$, and $Z \times_W V$ by $X_V$, $U_V$, and $Z_V$. We then have a commutative diagram
\[
\xymatrix@=0.2in{
\mathsf{DCoh}(U) \ar[r]^-{\calk} \ar[d]_{f^*} & \Calk(\mathsf{DCoh}_Z(X)) \ar[d]^{f^*} \\
\mathsf{DCoh}(U_V) \ar[r]^-{\calk} & \Calk(\mathsf{DCoh}_{Z_V}(X_V)).
}
\]
\end{lemma}
\begin{proof}
    The assumptions on $f$ imply that we have a commuting square
    \begin{equation*}
        \xymatrix@=0.2in{
            \Perf(X) \ar[r]^-{\ind} \ar[d]_{f^*} & \Ind(\Perf_Z(X)) \ar[d]^{f^*} \\
            \Perf(X_V) \ar[r]^-{\ind} & \Ind(\Perf_{Z_V}(X_V)).
        }
    \end{equation*}
    By quotienting and taking idempotent, we obtain the commuting square of the lemma.
\end{proof}
The Tate and Calkin realization for coherent sheaves (respectively pseudo-coherent complexes) are related by the composition of the natural exact functors
$$\Tate(\Coh_Z(X)) \to^q \Calk(\Coh_Z(X)) \to^{\Psi} \Calk(\mathsf{DCoh}_Z(X)).$$

\begin{lemma}\label{lemma:calk_n_tate}
Let $X$ be Noetherian, and $Z$, $U$, and $W$ satisfy the conditions of Definition \ref{defi:tate_real}. We have a commuting square
\[
\xymatrix@=0.2in{
\Coh_{|W|\cap U}(U) \ar[r]^-{\tate} \ar[d]_{\Phi} & \Tate(\Coh_{|Z|}(X)) \ar[d]^{\Psi q[-1]} \\
\mathsf{DCoh}_{|W|\cap U}(U) \ar[r]^-{\calk} & \Calk(\mathsf{DCoh}_{|Z|}(X))
}
\]
of $\infty$-categories, where $\Phi\colon \Coh_{|W|\cap U}(U) \to \mathsf{DCoh}_{|W|\cap U}(U)$ denotes the canonical functor.
\end{lemma}

\begin{proof}
According to Definition \ref{defi:calkin_real} we have that, for every pseudo-coherent complex $\F$ on $X$ with set-theoretic support in $|W|$,
$$\calk(j^*\F) =\mathrm{fib}(\F \to j_*j^*\F).$$
Since $j$ is proper and affine by assumption, we have that for $\F \in \Coh(X)$ the expression $j_*j^*\F$ has vanishing higher cohomology groups. In particular, $\calk(j^*\F)$ can be represented by the admissible Ind-object $j_*j^*\F/\F[-1]$. By Remark \ref{rmk:tate_real}, this admissible Ind-object represents the Calkin object corresponding to the Tate object $\tate(\F)$. The general case, i.e. of a coherent sheaf on $U$ which does not extend to $X$, follows by passing to idempotent completions.

The discussion above gives rise to the top square in the commutative diagram below, where $\Phi\colon \Coh_{|W|\cap U}(U) \to \mathsf{DCoh}_{|W|\cap U}(U) $ denotes the canonical functor to the derived category, and $q$ is the exact functor of Theorem \ref{thm:index_boundary}:
\[
\xymatrix@=0.2in{
\Coh_{|W|\cap U}(U) \ar[r]^-{\tate} \ar[d]_{\id} & \Tate(\Coh_{|Z|}(X)) \ar[d]^q\\
\Coh_{|W|\cap U}(U) \ar[r]^{\phi} \ar[d]_{\Phi} & \Calk(\Coh_{|Z|}(X)) \ar[d]^{\Psi}\\
\mathsf{DCoh}_{|W|\cap U}(U) \ar[r]^-{\calk[1]} & \Calk(\mathsf{DCoh}_{|Z|}(X)),
}
\]
where $\phi$ is the functor obtained by sending $\F \in \Coh_{|W| \cap U}(U)$ to $j_*\F/\widetilde{F} \in \Calk(\mathsf{DCoh}_{|Z|}(X))$, where $\widetilde{\F}$ is a pseudo-coherent subsheaf of $j_*\F$, such that $j^*\widetilde{F} = \F$. The outer square yields the required commutative diagram.
\end{proof}

For $X$ a Noetherian scheme, and $\xi$ a saturated flag of closed subschemes we denote by $X^{[m]}$ the scheme obtained by applying Definition \ref{defi:(m)} of $X^{(m)}$ for $A = k$.
We now construct a sequence of Tate realization functors
$$\tate\colon \Coh_{Z_j^{(j)}}(X^{[j]}) \to \Tate(\Coh_{Z_{j-1}^{[j-1]}}(X^{[j-1]})).$$
Lemma \ref{lemma:affineW}(b) implies that the crucial affineness condition of Definition \ref{defi:tate_real} is satisfied in this now case for dimension reasons. Composition of these exact functors yields a well-defined exact functor
$$\tate^n\colon \Coh(X^{[n]}) \to \nTate(\Coh_{Z_0}(X)).$$

The proposition below can also be obtained from \cite{MR3510209}.

\begin{proposition}\label{prop:tate_adelic}
The functor $\pi_*\tate^n$ agrees with the ($n$-Tate object valued) Beilinson-Parshin ad\`eles $\underline{F}_{X,\xi}$.
\end{proposition}

\begin{proof}
The functors $\ind$ and $\pro$ of Definition \ref{defi:tate_real} mirror localization and completion with respect to the scheme $X$. In particular, we see for $\F \in \Coh(W)$ that $\pi_*\tate^n(\F)= \underline{F}_{X,\xi}$.
\end{proof}

Composing $\nTate$ with pushforward $\pi\colon \Coh_{Z_0}(X) \to \Coh(\Spec k) = \Vect_f(k)$, we obtain an exact functor
$\Coh(X^{[n]}) \to \nTate(k).$

\begin{definition}\label{defi:tate_real_rel}
Let $X$ be a Noetherian $k$-scheme, and $\xi$ a saturated flag of closed subschemes. For every $k$-algebra $A$ we denote by $\Coh^{\flat}(X^{(n)})$ the full subcategory of $\Coh(X^{(n)})$, consisting of coherent sheaves which are pulled back along the canonical map $s\colon X^{(n)} \to X^{[n]}$. Denoting by
$$(-)_A\colon \nTate(k) \to \nTate(A)$$
the exact functor induced by $-\otimes_k A\colon \Vect_f(k) \to P_f(A),$ we have a unique $A$-linear functor
$$(\pi_*\circ \tate^n)_A\colon \Coh^{\flat}(X^{(n)}) \to \nTate(A),$$
such that the following diagram commutes:
\[
\xymatrix@C=0.5in{
\Coh(X^{[n]}) \ar[r]^-{\pi_*\tate^n} \ar[d]_{s^*} & \nTate(k) \ar[d]^{(-)_A} \\
\Coh^{\flat}(X^{(n)}) \ar[r]^-{(\pi_*\tate^n)_A} & \nTate(A)
}
\]
\end{definition}

\begin{proposition}\label{prop:CCboundary_step}
We denote by $\VB_f(W)$ the exact category of \emph{free vector bundles} on a scheme $W$. Let $X$ be a finite type, separated $k$-scheme of dimension $n$, and let $\xi$ be a saturated flag of closed subschemes. For every $k$-algebra $A$, the diagram
\[
\xymatrix{
\VB_f(X^{(n)}) \ar[r]^-{(\pi_* \tate^n)_A} \ar[d]_{s^*} & \nTate(A) \ar[d]^{q[-n]} \\
\Perf(X^{(n)}) \ar[r]^-{\pi_*\calk^n} \ar[r] & \Calk^n(\Perf(A))
}
\]
is commutative.
\end{proposition}

\begin{proof}
For $A = k$ this follows from applying the comparison of Lemma \ref{lemma:calk_n_tate} iteratively. The general case follows from the base change invariance of the Calkin realization (Lemma \ref{lemma:calk_basechange}), and Definition \ref{defi:tate_real_rel} of the functor $(\pi_* \circ \tate^n)_A$ by base change.
\end{proof}

We are now ready to prove that the spectral Contou-Carr\`ere symbol $\sigma_{X,\xi}^A$ can be represented as the composition $\pi_* \circ \partial_{\mathcal{Z}^{\text{loc}}_{X_A,\xi_A}}$.

\begin{proof}[Proof of Theorem \ref{thm:CCboundary}]
Proposition \ref{prop:CCboundary_step} established a compatibility between the Tate and Calkin realization:
$\pi_*\calk^n \simeq q(\pi_* \tate^n)_A [-n].$
Applying the non-connective algebraic $K$-theory functor $\Kb_-$ to this equivalence, we obtain two equivalent maps
\begin{equation}\label{eqn:calkvstate}
\Kb_{\pi_*\calk^n} \simeq \Kb_{q(\pi_* \tate^n)_A [-n]} \colon \Kb_{X^{(n)}} \to \Kb_{\Calk^n(\Perf(A))} \simeq \Sigma^n\Kb_{A}.
\end{equation}
Here, we made use of the fact that non-connective algebraic $K$-theory is cofinally invariant, i.e. cannot distinguish between an exact category and its idempotent completion. In particular,
$$\Kb_{\VB_f(X^{(n)})} \cong \Kb_{\VB(X^{(n)})} \cong \Kb_{X^{(n)}}.$$
By Proposition \ref{prop:tate_adelic}, and by the definition of $(\pi_*\tate^n)_A$ by base change (Definition \ref{defi:tate_real_rel} ), we see that the second map of \eqref{eqn:calkvstate} agrees with
$\Kb_{X^{(n)}} \to \Kb_{\nTate(A)} \to^{\Index^n} \Sigma^n\Kb_{A}.$
Hence, this map agrees with the spectral Contou-Carr\`ere symbol, by Definition \ref{defi:spectral_CC}. To conclude the proof we have to compare the first map with the $n$-fold composition of the boundary map
$\Kb_{X^{(n)}} \to^{\partial^n} \Sigma^n \Kb_{X_A,(Z_0)_A} \to^{\pi_*} \Sigma^n\Kb_A.$ Definition \ref{defi:calkin_real} implies that for every triple $(X,U,Z)$ we have a commutative cube of stable $\infty$-categories
\[
        \xymatrix@=0.1in{
          && 0 \ar[rrr] \ar'[d][dd]
              & & & \Perf(U) \ar[dd]^{\calk}        \\
         \Perf_Z(X) \ar[urr]\ar[rrr]\ar@{=}[dd]
              & & & \Perf(X) \ar[urr]\ar[dd] \\
          && 0 \ar'[r][rrr]
              & & & \Calk(\Perf_Z(X))                \\
          \Perf_Z(X) \ar[rrr]\ar[urr]
              & & & \Ind(\Perf_Z(X)), \ar[urr]        }
\]
where the top square comes from the localization sequence of the closed embedding, and the lower square corresponds to the short exact sequence of stable $\infty$-categories $$\Perf_Z(X) \to \Ind(\Perf_Z(X)) \to \Calk(\Perf_Z(X)).$$
Since the top and bottom face are localization sequences, applying $\Kb_-$ yields a commutative cube with top and bottom face being bi-cartesian. In particular, we obtain a commutative triangle relating the boundary maps of the bottom and top face:
\[
\xymatrix{
\Kb_{\Perf(U)} \ar[r]^{\partial} \ar[d]_{\calk} & \Sigma\Kb_{X,Z} \\
\Kb_{\Calk(\Perf_Z(X))}. \ar[ru]_{\partial} &
}
\]
Applying this comparison $n$ times, we see that $\Kb_{\pi_*\calk^n}$ is equivalent to $\pi_*\partial_n \circ \cdots \circ \partial_1$.
\end{proof}

\section{Reciprocity}\label{reciprocity}

Let $X$ be a proper integral curve over a field $k$. We write $X_0$ for its set of closed points. For every commutative $k$-algebra $A$, $x\in X_0$, and a pair of units in the ring of \emph{$A$-valued rational functions}
$$f,g \in A(X)^{\times} = (k(X) \otimes_k A)^{\times}$$
the Contou-Carr\`ere symbol gives an element $(f,g)_x$ of $A^{\times}$.

\begin{theorem}[Weil, Anderson--Pablos Romo, Beilinson--Bloch--Esnault]
The product below is well-defined and satisfies
$$\prod_{x \in X_0} (f,g)_x = 1.$$
\end{theorem}

This reciprocity law has been proven by Weil for $A = k$, it was generalized to the case of artinian rings $A$ by Anderson--Pablos Romo \cite{MR2036223}, and to general $A$ by Beilinson--Bloch--Esnault \cite[\S 3.4]{MR1988970}. Recently, P\'al has shown in \cite{MR2722779} that, for artinian rings, the relative version of Weil reciprocity follows from the absolute case $(A = k)$ after a change of fields.

This section is concerned with an extension of this result to varieties of arbitrary dimension (and arbitrary rings $A$). The absolute case ($A = k$) is due to Kato \cite{MR862638} (however, the case of surfaces was pioneered by Parshin). Recent work of Osipov--Zhu \cite{Osipov:2013fk} established a Contou-Carr\`ere reciprocity law for surfaces and artinian rings.

\medskip
Fix an integer $0 \leq i \leq n$. As before, $A$ denotes a $k$-algebra over a field $k$. The main player is an $n$-dimensional, integral, separated $k$-scheme of finite type $X$, \emph{together} with an \emph{almost saturated flag}
\begin{equation}\label{eqn:almost_saturated}
\zeta=(X = Z_n \supset \cdots Z_{i+1} \supset Z_{i-1} \supset \cdots Z_{0}),
\end{equation}
of closed integral subvarieties, satisfying $\dim Z_j = j$. If $i = 0$, we assume that $Z_1$ is proper over $k$.
\medskip

For every closed equiheighted $i$-dimensional subset $Z$, satisfying $Z_{i+1} \supset Z \supset Z_{i-1}$ we obtain a \emph{saturated flag} $\xi_Z$.
Note that we denote saturated flags by the letter $\xi$ for the sake of visual distinctness.

In order to formulate the reciprocity law, we need to construct an analogue of the ring of $A$-valued rational functions $A(X)$ on a curve $X$. This ring $A_{\zeta}(X)$ should be naturally associated to the data $(X,\zeta)$ and the $k$-algebra $A$. Further, for each $Z$ as above, we require a specialization homomorphism
$$A_{\zeta}(X) \to^{i_{\zeta}} A_{X,\xi_Z}.$$
The latter is required to make sense of the factors of the product
$$\prod_{Z_{i+1}\supset Z \supset Z_{i-1}}(f_0,\dots,f_n)_{\xi_Z}.$$


\begin{definition}\label{defi:F_xi_A} We define the following:
\begin{enumerate}
\item[(a)]Let $X$ be a separated $n$-dimensional $k$-scheme of finite type, $A$ a $k$-algebra, and $\zeta$ an almost complete flag in $X$.
For each $Z_{i+1}\supset Z \supset Z_{i-1}$ with $Z$ of pure dimension $i$ and not necessarily irreducible, we denote the ring of regular functions on the scheme $(\Comp \circ \Loc)^{n-i-1} \circ \Loc \circ (\Comp \circ \Loc)^i(X_A,(\xi_Z)_A)$ by $A^{}_{\zeta,Z}(X)$. We define the ring
$A_{\zeta}^{}(X)$ to be the direct limit $$A^{}_{\zeta}(X) = \varinjlim_{Z_{i-1} \subset Z \subset Z_{i+1}}A^{}_{\zeta,Z}(X) ,$$
where $Z$ is a closed subset of pure dimension $i$ (not necessarily irreducible).
\item[(b)] For every $Z$ as in (a), we denote the natural ring homomorphism $A_{\zeta}^{}(X) \to A_{X,\xi_Z}$ by $i_{\zeta}$.
\end{enumerate}
\end{definition}

In the definition above we can apply the operations $\Loc$ and $\Comp$ because we may replace the scheme by a suitable affine open neighbourhood.

After having introduced this colimit, we observe that the algebraic $K$-theory is manageable for formal reasons. This will be used in the proof of our main result.

\begin{rmk}\label{rmk:colimit}
Since non-connective $K$-theory of rings commutes with filtered colimits (Theorem 7.2 of \cite{MR1106918}) one has
$$\Kb_{A^{}_{X,\zeta}}  = \varinjlim_{Z_{i-1} \subset Z \subset Z_{i+1}}\Kb_{(\Comp \circ \Loc)^{n-i-1} \circ \Loc \circ (\Comp \circ \Loc)^i(X_A,(\xi_Z)_A)}.$$
\end{rmk}

We are now ready to state the main result of this section, in a classical formulation:

\begin{theorem}[Reciprocity for Contou-Carr\`ere symbols]\label{thm:classicalCC}
Let $X$ be an integral separated $n$-dimensional $k$-scheme of finite type, and let $A$ be a commutative $k$-algebra. Let $\zeta$ be an almost saturated flag as in \eqref{eqn:almost_saturated}. For every  $(n+1)$-tuple $f_0,\dots,f_n \in A_{\zeta}^{}(X)^\times$ we have that the product below is well-defined and satisfies the identity
$$\prod_{Z_{i+1} \supset Z \supset Z_{i-1}}(f_0,\dots,f_n)_{\xi_Z} = 1,$$
where $Z$ is integral and of dimension $i$.
\end{theorem}

We will deduce this result in Subsection \ref{sub:theproof} from an \emph{abstract reciprocity law} for compositions of boundary maps (see Corollary \ref{cor:abstract_Kato}). The reciprocity relation will be generalized to the existence of a null-homotopy for a certain map of spectra. We refer to such a construction as \emph{spectrification} (following Beilinson).

\subsection{Abstract Reciprocity Laws}

\subsubsection{Notation}

In Appendix \ref{sub:derivedcompletion} we explain a mild generalization of a construction due to Efimov, which allows one to complete a stable $\infty$-category $\C$ at a full subcategory $\Sf$. The resulting category is denoted by $\C_{\widehat{\Sf}}$. This is a categorical analogue of completion in commutative algebra. We refer the reader to the appendix of Efimov's \cite{Efimov:2010fk} for more details.

\begin{definition}\label{defi:iteration}
Let $\C$ be a stable $\infty$-category as in Paragraph \ref{subsub:completion}.
\begin{itemize}
\item[(a)] A chain of localizing subcategories $\mathbf{S}_0 \subset \mathbf{S}_1 \subset \dots$, will be referred to as a \emph{flag} in $\C$.
\item[(b)] We denote $\Sf_i$ by $\C_{[i]}$.
\item[(c)] We denote $\C_{\widehat{\Sf_i}}$ by $\C_{\widehat{[i]}}$.
\item[(d)] We write $\C_{(i)} = (\C/\Sf_i)^{\ic}$.
\item[(e)] We write $\C_{\widehat{(i)}} = (\C_{\widehat{[\Sf_i]}}/\Sf_i)^{\ic}$.
\item[(f)] Given a flag on $\C$ as above, we define the \emph{iterated removal-completion operation} by
$$\C_{\widehat{(0,n)}} = ((\C_{\widehat{(0,n-1)}})_{\widehat{\Sf_n}} /(\Sf_n)_{\widehat{(0,n-1)}})^{\ic},$$ with $\C_{\emptyset} = \C$.
\end{itemize}
\end{definition}

Let $X$ be a scheme. Given a flag of closed subschemes in $X$,
$$\xi = (X = Z_n \supset Z_{n-1} \supset \dots \supset Z_0),$$
we obtain a flag of localizing subcategories $\Sf_0,\dots,\Sf_{n-1}$ of $\Perf(X)$ by defining $\Sf_i$ to be $\Perf_{Z_i}(X)$. The following example is a special case of Proposition \ref{prop:efimov} in the appendix.

\begin{example}\label{ex:affinespacestd}
Let $X$ be affine $n$-space $\mathbb{A}^n = \Spec k[t_1,\dots, t_n]$, and $\xi$ the flag given by $Z_i= \Spec k[t_1,\dots, t_i]$. We then have
$\C_{\widehat{(0,n)}} \cong \Perf(\Spec k((t_1))\dots((t_n))).$
\end{example}

\subsubsection{Reciprocity Laws}

In the following we denote by $\C$ a stable $\infty$-category, and consider a chain $\Sf_0,\Sf_1,\dots,\Sf_n$ of localizing subcategories (as considered above). We will be concerned with the composition of boundary maps, connecting the non-connective $K$-theory spectra of various stable $\infty$-categories constructed from $\C$ with the help of the localizing subcategories.

Let $\C$ be a stable $\infty$-category together with a localizing subcategory $\Sf$. With respect to the terminology introduced in Definition \ref{defi:pre-iteration} we have short exact sequences of stable $\infty$-categories
$$\Sf \hookrightarrow \C \twoheadrightarrow \C_{(\Sf)}\text{ and } \Sf \hookrightarrow \C_{\widehat{\Sf}} \to \C_{\widehat{(\Sf)}}$$
and canonical functors
\begin{equation}\label{FG}\C \to^F \C_{\widehat{\Sf}} \text{ and }\C_{(\Sf)} \to^G \C_{\widehat{(\Sf)}}.\end{equation}
Furthermore, these short exact sequences and functors belong to a commutative diagram:
\begin{equation}\label{completediagram}
\xymatrix{
\Sf \ar@{^(->}[r] \ar[d]_{\id_{\Sf}} & \C \ar[d]^F \ar@{->>}[r] & \C_{(\Sf)} \ar[d]^G \\
\Sf \ar@{^(->}[r] & \C_{\widehat{\Sf}} \ar@{->>}[r] & \C_{\widehat{(\Sf)}}
}
\end{equation}
On the level of algebraic $K$-theory, the localization sequences yields a boundary map
$$ \Omega\Kb_{\widehat{(\Sf)}} \to^{\widehat{\partial}} \Kb_{\Sf}.$$

\begin{theorem}[Abstract Weil reciprocity]\label{thm:abstract_Weil}
Let $\C$ be a stable $\infty$-category together with a localizing subcategory $\Sf$. We assume the existence of an exact functor $\C \to^c \D$, where $\D$ denotes as well a stable $\infty$-category. We denote the inclusion $\Sf \hookrightarrow \C$ by $a$, and the restriction $c|_{\Sf}$ by $b$:
\[
\xymatrix{
\Sf \ar[r]^a \ar[rd]_b & \C \ar[d]^c \\
& \D
}
\]
Under these assumptions the map $\Omega\Kb_{\C_{(\Sf)}} \to \Kb_{\D}$ defined as the composition $b \circ \partial \circ G$
\[
\xymatrix{
\Omega\Kb_{\C_{(\Sf)}} \ar[r]^-{b\partial G} \ar[d]_G & \Kb_{\D} \\
\Omega\Kb_{\C_{\widehat{(\Sf)}}} \ar[r]^-{\partial} & \Kb_{\Sf} \ar[u]_b
}
\]
is homotopic to the zero map.
\end{theorem}

\begin{proof}
We have a commutative diagram:
\[
\xymatrix{
\Omega\Kb_{\C_{(\Sf)}} \ar[r]^{\partial} \ar[d]_G & \Kb_{\Sf} \\
\Omega\Kb_{\C_{\widehat{(\Sf)}}}  \ar[ru]_{\partial} &
}
\]
Commutativity of the diagram above follows from the naturality of boundary maps in algebraic $K$-theory (applied to \eqref{completediagram}). This implies
$$b \circ \partial \circ G \simeq c\circ a \circ \partial\colon \Omega \Kb_{(\Sf)} \to \Kb_{\D}.$$

We may therefore focus on establishing the null-homotopy of the map $a \circ \partial$. We have a commuting diagram of spectra, with the square being bi-cartesian:
\[
\xymatrix@=0.2in{
& \Omega\Kb_{\C_{(\Sf)}} \ar[ldd]_-{b\partial} \ar[r] \ar[d]^{\partial} & 0 \ar[d]  \\
&  \Kb_{\Sf} \ar[r]^a \ar[ld]_-{\;\; b} & \Kb_{\C} \ar[lld]_c  \\
\Kb_{\D} & &
}
\]
Commutativity of the square implies $a \circ \partial \simeq 0$.
\end{proof}

\begin{example}[Weil reciprocity]
Let $X$ be a proper, integral curve over a field $k$, we set $\C = \Perf(X)$, and for every reduced $0$-dimensional closed subscheme $Z$ (not assumed to be irreducible) we let $\Sf$ be $\Perf_{|Z|}(X)$. We then have $\C_{(\Sf)} \cong \Perf(X \setminus Z)$. Using the (derived) pushforward functor to the base field
$\pi_*\colon \Perf(X) \to \Perf(k)$, Theorem \ref{thm:abstract_Weil} implies now that the canonical map
$$\pi_*\partial\colon\Omega\Kb_{X \setminus Z} \to \Kb_{k}$$
is homotopic to zero. The field of rational functions $k(X)$ arises as the direct limit
$k(X) \cong \colim_Z \Oo_{X \setminus Z},$
in particular we have $\Kb_{k(X)} \cong \colim_Z \Kb_{X \setminus Z}$, by virtue of Theorem 7.2 in \cite{MR1106918}. Since we have a functor from the direct limit of $\infty$-categories $\Perf(X)_{\widehat{(\Sf)}}$ to $\Perf(\Ab_X)$, by virtue of Theorem \ref{thm:general_efimov}(a), we obtain the commutative diagram in the stable $\infty$-category of spectra on the left:
\[
\xymatrix{
\Omega\Kb_{k(X)} \ar[r]^-{\pi_*\partial \simeq 0} \ar[d]_{- \otimes_F \Ab_X} & \Kb_{k} \\
\Omega\Kb_{\Ab_X} \ar[ur] &
}
\qquad \qquad \qquad
\xymatrix{
K_2(k(X)) \ar[r]^{\pi_*\partial = 1} \ar[d]_{- \otimes_F \Ab_X} & K_1(k) \\
K_2(\Ab_X). \ar[ru]
}
\]
Passing to homotopy groups, we obtain the commutative diagram of abelian groups on the right. Thus, we see that $\prod_{x \in X_0} \pi_*\partial\{f,g\} =\prod_{x \in X_0} N_{\kappa(x)/k}(f,g)_x = 1$, for all pairs of invertible rational functions on $X$.
\end{example}

Similarly one could use this result to prove reciprocity for Contou-Carr\`ere symbols, relative to any $k$-algebra $A$. We will give more details at the end of this section, when discussing the proof of reciprocity for higher-dimensional varieties.

\begin{theorem}[Abstract Parshin reciprocity]\label{thm:abstract_Parshin}
We denote by $\C$ a stable $\infty$-category, and by $\Sf_0 \subset \Sf_1 \subset \C$ a length $2$ chain of localizing subcategories. The construction of \eqref{FG} applied to $(\Sf_1)_{\widehat{(0)}} \subset \C_{\widehat{(0)}}$ yields a functor $G\colon \C_{\widehat{(0)}(1)} \to \C_{\widehat{(0,1)}}$ which belongs to a commutative diagram
\[
\xymatrix@=0.3in{
\Omega^2\Kb_{\C_{\widehat{(0)}(1)}} \ar[r]^{\partial} \ar[d]_G & \Omega\Kb_{\C_{\widehat{(0)}[1]}} \ar[r]^{\partial} \ar[d]^{\id} & \Kb_{\C_{[0]}} \\
\Omega^2\Kb_{\C_{\widehat{(0,1)}}} \ar[r]^{\partial} & \Omega\Kb_{\C_{\widehat{(0)}[1]}}, \ar[ru]_{\partial} &
}
\]
such that the composition of the top row is equivalent to the zero map.
\end{theorem}

\begin{proof}
As in the proof of Abstract Weil reciprocity, the existence of the commutative diagram follows directly from the naturality of boundary maps. We therefore turn to proving the existence of a null-homotopy for the composition of the top row. Similar to the proof of Theorem \ref{thm:abstract_Weil} we show that this composition factors through the juxtaposition of two subsequent maps in an exact sequence of spectra (thus is homotopic to $0$). This is achieved by the commuting diagram below on the left:
\[
\xymatrix@=0.24in{
& \Omega^2\Kb_{\C_{\widehat{(0)}(1)}} \ar[ldd]_{\partial\partial} \ar[r] \ar[d]^{\partial} & 0 \ar[d]  \\
& \Omega \Kb_{\C_{\widehat{(0)}[1]}} \ar[r]^a \ar[ld]_{\partial} &  \Omega \Kb_{\C_{\widehat{(0)}}}, \ar@{-->}[lld]^{\partial}  \\
\Kb_{\C_{[0]}} & &
}
\qquad
\xymatrix@=0.1in{
          && 0 \ar[rrr] \ar'[d][dd]
              & & & \Kb_{\C_{\widehat{[0]}[1]}} \ar[dd]_{\widehat{a}}       \\
         \Omega\Kb_{\C_{\widehat{(0)}[1]}} \ar[urr]\ar[rrr]\ar[dd]_a
              & & & \Kb_{\C_{[0]}} \ar[urr]\ar[dd]^(.3){\id} \\
          && 0 \ar'[r][rrr]
              & & &\Kb_{\C_{\widehat{[0]}}}               \\
          \Omega\Kb_{\C_{\widehat{(0)}}} \ar@{-->}[rrr]\ar[urr]
              & & & \Kb_{\C_{[0]}}. \ar[urr]
}
\]
provided we can establish the existence of the dashed map. To explain the diagram, note that almost all of the maps appearing in the commutative diagram above are boundary sequences for localization sequences in algebraic $K$-theory, the exception being $a\colon \Omega \Kb_{\widehat{(0)}[1]} \to \Omega \Kb_{\C_{\widehat{(0)}}}$ which is induced by the inclusion of the localizing subcategory
$$a \colon \C_{\widehat{(0)}[1]} \hookrightarrow \C_{[0]}.$$

A suitable candidate for the dashed map is given by the $K$-theory boundary morphism of the exact sequence of stable $\infty$-categories
$\C_{[0]} \hookrightarrow \C_{\widehat{[0]}} \twoheadrightarrow \C_{\widehat{(0)}}.$
Naturality of boundary maps implies the existence of a commutative diagram with exact rows as depicted above on the right. Most of the arrows in the cubical diagram are not labelled. The respective maps are well-defined by the fact that the rows are localization sequences in $K$-theory. The morphism $\widehat{a}$ is induced by the inclusion of the localizing subcategory $\C_{\widehat{[0]}[1]} = (\Sf_1)_{\widehat{\Sf}_0} \hookrightarrow \C_{\widehat{\Sf}_0} = \C_{\widehat{[0]}}$ (see also Definition \ref{defi:pre-iteration}).

The front square of the commuting cube amounts to the existence of the commuting triangle containing the dashed map above. This concludes the proof.
\end{proof}

Let us explain how this result implies Parshin's reciprocity statement for surfaces.

\begin{example}[Parshin reciprocity]
Let $Y$ be an integral separated excellent surface. We denote by $\C$ the stable $\infty$-category $\Perf(Y)$. For a fixed closed point $x \in Y$, we obtain a localizing subcategory $\Sf_0 = \Perf_{\{x\}}(Y)$. Moreover, for every integral curve $C$, with $x \in C$, we have a localizing subcategory $\Sf_1 = \Perf_{|C|}(Y)$. Theorem \ref{thm:general_efimov}(a) implies that $\C_{\widehat{(0,1)}} \cong \Perf(\Ab_{Y,C,x})$, and a direct limit of the $\infty$-categories $\C_{\widehat{(0)}(1)}$ yields $\Perf(\Frac(\widehat{\Oo_{Y,x}}))$. Hence, by Theorem \ref{thm:abstract_Parshin}, we have a commutative diagram of $K$-theory groups
\[
\xymatrix{
K_3(\Frac(\widehat{\Oo_{Y,x}})) \ar[r] \ar[d] & K_1(\Oo_{Y,x}/\mathfrak{m}_{Y,x}) \\
K_3(\Ab_{Y,?,x}), \ar[ru]
}
\]
in which the top map is trivial. Here $\Ab_{Y,?,x}$ denotes the ring of ad\`eles for chains, $Y \supset C \supset \{x\}$, where $C$ can be an arbitrary irreducible curve containing $x$. In particular, we see that for every triple $f_0,f_1,f_2 \in \Frac(\widehat{\Oo_{Y,x}})^{\times}$ we have the identity
$$\prod_{C \ni x}(f_0,f_1,f_2)_{x\in C} = 1,$$
where the product is indexed by integral closed curves containing $x$.
\end{example}

Combining Theorems \ref{thm:abstract_Weil} and \ref{thm:abstract_Parshin}, we obtain an abstract analogue of Kato reciprocity. In the next subsection we will use this result to deduce a reciprocity law for Contou-Carr\`ere symbols.

\begin{corollary}[Abstract Kato reciprocity]\label{cor:abstract_Kato}
Let $\C$ be a stable $\infty$-category. We fix positive integers $i$ and $n$, and assume that we have a chain of localizing subcategories $\Sf_j \subset \C$, indexed by $0 \leq j \leq n-1$.
\begin{itemize}
\item[(a)] If $i = 0$, suppose that we have a commutative diagram
\[
\xymatrix{
\Sf_0 \ar[r]^a \ar[rd]_b & \Sf_1 \ar[d]^c \\
& \D,
}
\]
where $\D$ denotes a stable $\infty$-category. Then, the morphism $\Omega^n\Kb_{\C_{(0)\widehat{(1,n-1)}}} \to \Kb_{\D}$ defined as the following composition
$$\Omega^n\Kb_{\C_{(0)\widehat{(1,n-1)}}} \to^{G_{\widehat{(1,n-1)}}} \Omega^n\Kb_{\C_{\widehat{(0,n-1)}}} \to^{\partial^n} \Kb_{\C_{[0]}} \to^b \Kb_{\D}$$
is null-homotopic. Here, $G\colon \C_{(0)} \to \C_{\widehat{(0)}}$ denotes the functor of \eqref{FG} applied to $\Sf_0 \subset \C$, and
$$G_{\widehat{(1,n-1)}}\colon \C_{(0)\widehat{(1,n-1)}} \to \C_{\widehat{(0,n-1)}}$$
denotes the induced functor, obtained by applying the functorial construction $(-)_{\widehat{(1,n-1)}}$ to $G$.

\item[(b)] If $i \neq 0$, then the following composition
$$\Omega^n\Kb_{\C_{\widehat{(0,i-1)}(i)\widehat{(i+1,n-1)}}} \to^{G_{\widehat{(i+1,n-1)}}} \Omega^n\Kb_{\C_{\widehat{(0,n-1)}}} \to^{\partial^n} \Kb_{\C_{[0]}}$$
is null-homotopic. Here, $G\colon \C_{\widehat{(0,i-1)}(i)} \to \C_{\widehat{(0,i)}}$ denotes the functor of \eqref{FG} applied to the localizing subcategory $(\Sf_i)_{\widehat{(0,i-1)}} \subset \C_{\widehat{(0,i-1)}}$, and
$$G_{\widehat{(i+1,n-1)}}\colon \C_{\widehat{(0,i-1)}(i)\widehat{(i+1,n-1)}} \to \C_{\widehat{(0,n-1)}}$$
denotes the induced functor, obtained by applying the functorial construction $(-)_{\widehat{(i+1,n-1)}}$ to $G$.
\end{itemize}
\end{corollary}

\begin{proof}
The first assertion follows directly from Theorem \ref{thm:abstract_Weil}, when setting $\C = \Sf_1$, and $\Sf = \Sf_0$.

We will now turn to the proof of the second assertion.
For $j \leq i-1$ we denote by
\begin{equation}\label{eqn:local_boundary2}
\partial_j\colon \Omega^{j+1}\Kb_{\C_{\widehat{(0,j)}[j+1]}} \to \Omega^j\Kb_{\C_{\widehat{(0,j-1)}[j]}}
\end{equation}
the boundary morphism in $K$-theory, associated to the short exact sequence
\begin{equation}\label{eqn:basic_local_zigzag}
\C_{\widehat{(0,j-1)}[j]} \hookrightarrow \C_{\widehat{(0,j-1)}\widehat{[j]}[j+1]} \twoheadrightarrow \C_{\widehat{(0,j)}[j+1]}
\end{equation}
of stable $\infty$-categories. Analogously, we have the boundary maps
\begin{equation}\label{eqn:local_boundary3}
\partial_i\colon \Omega^{i+1}\Kb_{\C_{\widehat{(0,i-1)}(i)[i+1]}} \to \Omega^i\Kb_{\C_{\widehat{(0,i-1)}[i]}},
\end{equation}
and for $j \geq i+1$
\begin{equation}\label{eqn:local_boundary4}
\partial_j\colon \Omega^{j+1}\Kb_{\C_{\widehat{(0,i-1)}(i)\widehat{(i+1,j)}[j+1]}} \to \Omega^j\Kb_{\C_{\widehat{(0,j-1)}(i)\widehat{(i+1,j-1)}[j]}}.
\end{equation}
We want to show that the composition of these boundary maps satisfies
$\partial_0 \circ \cdots \circ \partial_{n-1} \simeq 0.$
In fact, Theorem \ref{thm:abstract_Parshin} implies that $\partial_{i-1} \circ \partial_i \simeq 0$. To see this one chooses the $\C$ in \emph{loc. cit.} to be the stable $\infty$-category $\C_{\widehat{(0,i-2)}[i+1]}$, $\Sf_0 = \C_{\widehat{(0,i-2)}[i-1]}$, and $\Sf_1 = \C_{\widehat{(0,i-2)}[i]}$.
\end{proof}

\begin{example}[Kato reciprocity]\label{ex:Kato}
Let $X$ be an integral separated excellent scheme of pure dimension $n$. Let $\zeta$ denote an almost saturated flag of closed integral subschemes
$$\zeta=(X \supset Z_{n-1} \supset \cdots\supset Z_{i+1}\supset Z_{i-1}\supset \cdots \supset Z_0),$$
indexed by $j \neq i$, with $\dim Z_j = j$. If $i = 0$, we assume that $Z_1$ is proper over a field $k$. For every (not necessarily irreducible) reduced closed subscheme $Z_i$ of pure dimension $i$, and $Z_{i+1} \supset Z \supset Z_{i-1}$  we obtain a natural chain of localizing subcategories $\Sf_j:=\Perf_{|Z_j|}$ on $\C = \Perf(X)$.  Abstract Kato reciprocity (Corollary \ref{cor:abstract_Kato}) now implies the existence of a commutative diagram
\[
\xymatrix{
\Omega^n\Kb_{F_{X,\zeta}} \ar[r]^-{0} \ar[d]_{-\otimes_{F_{X,\zeta}}\Ab_{X,\zeta}}  & \Kb_{\D}, \\
\Omega^n\Kb_{\Ab_{X,\zeta}}, \ar[ru]_{\partial^n} &
}
\]
where we let $\D = \Perf(k)$ for $i = 0$, and $\Perf_{Z_0}(X)$ otherwise. As before, this implies that for an $(n+1)$-tuple of invertible elements of $F_{X,\zeta}$, we have
$$\prod_{Z_{i+1}\supset Z \supset Z_{i-1}}(f_0,\dots,f_n)_{\xi_Z} = 1,$$
where $Z$ is integral and of dimension $i$.
\end{example}

\subsection{Reciprocity for Contou-Carr\`ere Symbols}\label{sub:theproof}

In the following we fix a separated, reduced $k$-scheme $X$ of finite type and dimension $n$, a $k$-algebra $A$, and an integer $i$. As in Example \ref{ex:Kato}, $\zeta$ denotes an almost saturated flag of closed integral subschemes
$$\zeta=(X \supset Z_{n-1} \supset \cdots\supset Z_{i+1}\supset Z_{i-1}\supset \cdots \supset Z_0),$$
indexed by $0 \leq j \leq n$ with $j \neq i$, and satisfying $\dim Z_j = j$. The condition of being \emph{almost saturated} stipulates that up to the choice of $Z_{i+1} \supset Z_i\supset Z_{i-1}$, the flag cannot be further extended. If $i = 0$, we assume that $Z_1$ is proper over a field $k$.

Alluding to the notation of abstract Kato reciprocity (Corollary \ref{cor:abstract_Kato}), we define
$$\C = \Perf(X_A)\text{, } \qquad \qquad \Sf_j = \Perf_{|(Z_j)_A|}(X_A),$$
where $Z_i = Z$ is a not necessarily irreducible closed subset of pure dimension $i$, satisfying $Z_{i+1} \supset Z \supset Z_{i-1}$.

\begin{lemma}\label{lemma:CCefimov}
Using the notation introduced earlier, we have the following equivalences.
\begin{itemize}
\item[(a)] $\C_{\widehat{(0,i-1)}(i)\widehat{(i+1,n-1)}} \cong \Perf((\Comp \circ \Loc)^{n-i-1} \circ \Loc \circ (\Comp \circ \Loc)^i(X_A,(\xi_Z)_A))$ (see Definition \ref{defi:F_xi_A}). In particular, taking the colimit of the diagram of these stable $\infty$-categories indexed by all possible $Z$, we obtain $\Perf(A_{\zeta}^{}(X))$.

\item[(b)] For each $Z_{i+1}\supset Z\supset Z_{i-1}$ we denote by $\xi_Z$ the corresponding complete flag. Then we have $\C_{\widehat{(0,n-1)}} \cong \Perf(A_{X,\xi_Z})$.
\end{itemize}
\end{lemma}

\begin{proof}
The second assertion is a direct consequence of Theorem \ref{thm:general_efimov}(c). The first assertion is proven by similar means as the results in Subsection \ref{subsub:HLFcatcomp}: as in \emph{loc. cit.} one proceeds by induction, where the $i$-th step (due to the absence of completion) has to be treated separately (using Lemma \ref{lemma:pre_ind_step} instead of Corollary \ref{cor:ind_step}).
\end{proof}

Using the equivalences of stable $\infty$-categories, provided by Lemma \ref{lemma:CCefimov}, abstract Kato reciprocity implies the following corollary.

\begin{corollary}[Spectral Contou-Carr\`ere reciprocity]\label{cor:spectralCCrec}
The following composition
$$\Omega^n\Kb_{A^{}_{\zeta,Z}(X)} \to^{i_{\zeta}} \Omega^n\Kb_{A_{X,\xi_Z}^{}} \to^{\sigma_{X,\xi_Z}^A} \Kb_{A}$$
is null-homotopic (see Definition \ref{defi:F_xi_A}).
Taking the filtered colimit over $Z$, we obtain the composition
$$\colim_Z\Omega^n\Kb_{A_{\zeta,Z}^{}(X)} \to^{i_{\zeta}} \colim_Z\Omega^n\Kb_{A_{X,\xi_Z}} \to^{\colim\sigma_{X,\xi_Z}^A} \Kb_A$$
which is also null-homotopic.
\end{corollary}

\begin{proof}
Using Remark \ref{rmk:colimit} one obtains the second commuting triangle from the first (including the null-homotopy), by taking a colimit ranging over the collection of all possible $Z_{i+1} \supset Z \supset Z_{i-1}$. At the beginning of this subsection, we have already defined a chain of localizing subcategories $\Sf_j$ on $\C = \Perf(X_A)$, which allows us to evoke abstract Kato reciprocity (Corollary \ref{cor:abstract_Kato}). We only need to verify that one of the conditions (a) or (b) holds, in order to apply this result. If $i = 0$, then $Z_1$ is proper over $k$ by assumption. By virtue of Lemma \ref{lemma:push_perfect} we obtain a pushforward functor
$$\pi_*\colon \Sf_1 \cong \Perf_{(Z_1)_A}(X_A) \to \Perf(A),$$
which yields the required commutative diagram
\[
\xymatrix{
\Perf_{(Z_0)_A}(X_A) \ar@{^(->}[r] \ar[rd]_{\pi_*} & \Perf_{(Z_1)_A}(X_A) \ar[d]^{\pi_*} \\
& \Perf(A).
}
\]
If $i \geq 1$, there is nothing to check. This concludes the proof of the first assertion.

The second assertion also follows by applying abstract Kato reciprocity (Corollary \ref{cor:abstract_Kato}). For $j \neq i$ one defines $\Sf_j$ as before, and in degree $i$ one sets $\Sf_i =\colim_{Z}\Perf_{Z_A}(X_A)$, where $Z$ ranges over all closed subsets which are of pure dimension $i$ and satisfy $Z_{i-1} \subset Z \subset Z_{i+1}$.
\end{proof}

\begin{proof}[Proof of Theorem \ref{thm:classicalCC}]
Let $f_0,\dots,f_n$ be a commuting $(n+1)$-tuple of units in the ring $A_{\zeta}^{}(X)$. This corresponds to a map $\Sigma_+^{\infty}\mathbb{T}^{n+1} \to \Sigma^\infty_+(BA_{\zeta}(X)^{\grp})^{}$. The right hand side can be expressed as a colimit by definition of the ring $A_{\zeta}^{}(X)$ (see Definition \ref{defi:F_xi_A}). Because the torus is compact, the map factors through a map $\Sigma_+^{\infty}\mathbb{T}^{n+1} \to \Sigma_+^\infty(BA_{\zeta,Z}(X)^{\grp})^{}$ for some $Z$.

The ring $A_{\zeta,Z}(X)$ splits into a product over the irreducible components of $Z = \bigcup_{k=1}^m W_k$. Therefore, spectral Contou-Carr\`ere reciprocity \ref{cor:spectralCCrec} yields a commutative diagram
\[
\xymatrix{
\bigoplus_{k=1}^m\Omega^n\Kb_{A^{}_{\zeta,W_k}(X)} \ar[r]^{\;\;\;\;\;\;\;\cong} & \Omega^n\Kb_{A^{}_{\zeta,Z}(X)} \ar[r]^0 \ar[d]  & \Kb_{A} \\
& \Omega^n\Kb_{\Ab_{X,\xi_Z,A}^{}}. \ar[ru]_{\sigma_{X,\xi_Z}^A} &
}
\]
Passing to homotopy groups, and applying the resulting maps to the object represented by the Steinberg symbol $\{f_0,\dots,f_n\}$ (i.e. a higher commutator by Proposition \ref{prop:Steinberg}), we obtain the identity
$$\prod_{i=1}^m\pi_*\partial_{W_i}^n\{f_0,\dots,f_n\} = \prod_{i=1}^m(f_0,\dots,f_n)_{\xi_{W_i}} = 1.$$
This concludes the proof.
\end{proof}

\appendix

\section{Categorical and homotopical framework}\label{inftycats}

\subsection{\texorpdfstring{$\infty$-}{Infinity }Categories}

We briefly review the main ideas from the theory of $\infty$-categories that are repeatedly used in our work. For a more detailed overview, we refer the reader to Groth's survey \cite{Groth:2010fk}.

\subsubsection{Spaces are $\infty$-Groupoids}\label{subsub:spaces}

The only topological spaces that play a role for us are those which are homotopy equivalent to a CW-complex. The term \emph{space} (regardless of pointed or unpointed) will always refer to topological spaces of this type. Since every space $X$ is weakly equivalent to the geometric realization of the simplicial set of singular simplices $S_{\bullet}(X)$, we could equivalently work with simplicial sets.

We now remind the reader of a hierarchy on the homotopy category of (unpointed) spaces.

\begin{itemize}
\item[-] A \emph{homotopy $0$-type} is an unpointed space homotopy equivalent to a discrete topological space,
\item[-] a \emph{homotopy $1$-type} is an unpointed space with vanishing higher homotopy groups,
\item[-] a \emph{homotopy $n$-type} is an unpointed space $X$ with $\pi_k(X) = 0$ for $k \geq n+1$.
\end{itemize}

The category of homotopy $0$-types is equivalent to the category of \emph{sets}. The category of homotopy $1$-types is closely related to the category of (small) groupoids $\Gg$. To a groupoid $\Gg$, one simply assigns the geometric realization of its nerve $|N\Gg|$. Vice versa, given an unpointed topological space $X$, we have the Poincar\'e groupoid $\pi_{\leq 1}(X)$. Its set of objects is the set of points in $X$. A morphism from $x \in X$ to $y \in X$ is a homotopy class of paths connecting $x$ and $y$.

The natural map of groupoids $\Gg \to \pi_{\leq 1}(|N\Gg|)$ is not a \emph{strict isomorphism}. However, it is an \emph{equivalence} of groupoids. Using this fact, one can show that the above functors
{\emph{induce an equivalence between the $2$-category of \emph{groupoids} and the $2$-category of \emph{homotopy $1$-types}.}
This motivates the following slogan of modern homotopy category:
\begin{quote}
{\emph{The collection of homotopy $n$-types forms the $(n+1)$-category of $n$-groupoids. Unpointed spaces correspond to $\infty$-groupoids.}}
\end{quote}

\subsubsection{Simplicial Sets and $\infty$-Categories}\label{sect:reviewnerveabit}

Intuitively speaking, an $\infty$-category $\C$ is a category enriched in $\infty$-groupoids (i.e. unpointed spaces). Hence, for every pair of objects $X,Y \in \C$ we have a space of morphisms $\Hom_{\C}(X,Y)$. Since this space will only matter up to homotopy, composition should not be expected to be defined strictly, but only up to a homotopy, which itself is well-defined up to higher homotopies of all orders. It is difficult to extract a meaningful definition from this heuristic description, but its value should not be underestimated. To a large extent it is possible to work with $\infty$-categories as a blackbox, as long as one accepts that there is a well-behaved calculus of homotopy coherent commutative diagrams.

In the rigorous setting of quasi-categories (see e.g. Lurie's \cite{Lurie:bh}), one defines $\infty$-categories as simplicial sets satisfying a mild technical condition. This definition is motivated by the classical construction of nerves of categories. Recall that for a classical category $\C$ we define its nerve $N\C$ to be the simplicial set with objects as $0$-simplices, morphisms as $1$-simplices, composable pairs of morphisms as $2$-simplices, etc. Grothendieck observed that one can reconstruct a category from its nerve (even up to isomorphism of categories, see e.g. \cite{Lurie:bh}). A simplicial set is the nerve of a category, if and only if it satisfies a collection of strict horn-filling conditions, the most important one of which is explained below.

The set of $2$-simplices of $N\C$ can be understood as the set of commuting triangles as depicted below on the left:
\[
\xymatrix{
& Y \ar[rd]^{g} & \\
X \ar[ur]^{f} \ar[rr]^{gf} & & Z
}
\qquad
\xymatrix{
& Y \ar[rd]^{g} & \\
X \ar[ur]^{f} & & Z
}
\qquad
\xymatrix{
& Y \ar[rd]^{g} & \\
X \ar[ur]^{f} \ar[rr]^{h} & & Z
}
\]
The horn-filling condition in this particular case amounts to stating that every diagram as depicted above in the middle can be completed to a commuting triangle as above. For a classical category this can always be achieved in precisely one way.

Even if one does not know the definition of an $\infty$-category, one could try to guess what the nerve of an $\infty$-category should be. Accepting the above slogan that, whatever $\infty$-categories are, we want to have a good calculus of commutative diagrams, we arrive as a definition for the set of $2$-simplices in the nerve at the set of commuting triangles as depicted above on the right.
There are two interesting new features. First of all we cannot say that $h$ is the composition of $f$ and $g$. Rather, $h$ is one of possibly many compositions of $f$ and $g$. The invisible $2$-cell of the triangle above should be thought of as a homotopy connecting both sides. It turns out that if we no longer require horns to be filled uniquely, this is sufficient to characterize nerves of $\infty$-categories. This is precisely how quasicategories are defined by Joyal and in \cite{Lurie:bh}.

What separates the subcategory of classical categories from its complement in quasicategories is the existence of a \emph{strict} composition operation for morphisms. In $\infty$-categories, composition is only well-defined up to a contractible space of choices. It is this little bit of extra homotopical glue, which makes the theory of $\infty$-categories so flexible.

As a natural consequence of this liberality, the only possible notion of \emph{commutative} diagrams is automatically \emph{homotopy coherent} in a strong sense.

If $I_{\bullet}$ is a simplicial set, then an $I_{\bullet}$-indexed commutative diagram in an $\infty$-category $\C$ is a map of simplicial sets $I_{\bullet} \to \C$. A commutative square
\[
\xymatrix{
X \ar[r] \ar[d] & Y \ar[d] \\
Z \ar[r] & W
}
\]
for example is a map of simplicial sets $(\Delta^1)^2 \to \C,$ sending the $0$-simplices of the square $(\Delta^1)^2$ to the objects $X,Y,Z,W$.

\subsection{Stable $\infty$-Categories}\label{app:stable}

We refer the reader to \cite[Ch. 1]{Lurie:ha} for a more detailed account. Every $\infty$-category $\C$ has an associated homotopy category $\Ho(\C)$, where the set of morphisms is defined to be the set of connected components
$\Hom_{\Ho(\C)}(X,Y) = \pi_0\Hom_{\C}(X,Y).$ A \emph{stable $\infty$-category} has a natural triangulated structure on its homotopy category. Examples include the stable $\infty$-category of \emph{spectra}, and other enhancements of triangulated categories (for example pre-triangulated dg-categories).

By definition, a stable $\infty$-category $\C$ is pointed, i.e. there exists an initial and final object $\bullet$. Moreover, we assume the existence of finite limits and colimits, as well as that a commutative diagram
\begin{equation} \label{aaa_e}
\xymatrix{
X \ar[r] \ar[d] & Y \ar[d] \\
Z \ar[r] & W
}
\end{equation}
is a pullback if and only if it is a pushout. The endofunctors $\Sigma\colon \C \to \C$, and $\Omega\colon \C \to \C$,
\[
\xymatrix{
X \ar[r] \ar[d] & \bullet \ar[d] \\
\bullet \ar[r] & \Sigma X,
}
\qquad \qquad \qquad
\xymatrix{
\Omega X \ar[r] \ar[d] & \bullet \ar[d] \\
\bullet \ar[r] & X,
}
\]
are defined by virtue of the cocartesian, respectively cartesian squares above. As a consequence of the definition of a stable $\infty$-category, $\Sigma$ and $\Omega$ are inverse equivalences. The induced functors on the homotopy category $\Ho(\C)$ give rise to the translation functors of the triangulated structure of $\Ho(\C)$. The distinguished triangles are the images of bi-cartesian squares of the form of Diagram \ref{aaa_e} with $W=\bullet$. We denote the $\infty$-category of stable $\infty$-categories by $\st$.

\section{Derived Completion of Schemes and Categories}\label{boundary}
\begin{verse}
\emph{\hfill ``You complete me.'' - J. Maguire}
\end{verse}
The study of \emph{derived completion} goes back to work of Greenlees--May \cite{Greenlees-May}, Dwyer--Greenlees \cite{Dwyer-Greenlees}, and was embedded into the realm of derived algebraic geometry by Lurie \cite{DAGXII} and Gaitsgory--Rozenblyum \cite{Gaitsgory:2011lh}. We will mostly follow Lurie \cite[Ch. 4 \& 5]{DAGXII}.

For every ring $R$, and an ideal $I$, we recall (see Subsection \ref{sub:derivedcompletion}) Lurie's definition of the derived completion $\widehat{R}_I^{\der}$. This is a connective $E_{\infty}$-ring spectrum (\cite[\S 4.2]{DAGXII}). If $R$ is Noetherian, the derived completion is canonically equivalent to its classical counterpart \cite[Prop. 4.3.6]{DAGXII}. However, for a non-Noetherian ring $R$, the derived completion is genuinely different, which affects the stable $\infty$-category of perfect complexes.

In Subsection \ref{sub:modification} we rephrase and generalize constructions of Efimov \cite{Efimov:2010fk}; we show how perfect complexes on the derived completions can be understood by an abstract construction on the level of stable $\infty$-categories.

We then use a calculation of Porta--Shaul--Yekutieli \cite{MR3181733} (see also \cite{Porta2015}) to conclude that $\widehat{R}_I^{\der}$ is in fact a classical ring, if $I$ is \emph{weakly proregular} in $R$ (see Definition \ref{defi:proregular}). This will allow us to remove derived rings from our work in retrospect.

\subsection{Derived Completion}\label{sub:derivedcompletion}

We fix a ring $R$ and a finitely generated ideal $I$. We briefly review the notion of \emph{derived complete} complexes of $R$-modules, as studied in \cite[\S 4.2]{DAGXII}. A review of this material in the language of triangulated categories is given in \cite[Tag 091N]{stacks-project}. We say that a complex of $R$-modules is $I$-\emph{complete}, if for every $x \in I$ the homotopy limit of the inverse system
$$ \lim [\cdots \rightarrow M \xrightarrow{x} M \xrightarrow{x} M], $$
i.e. the fibre of
$$\prod_{n \geq 0} M \xrightarrow{x} \prod_{n \geq 1} M,$$
vanishes in the stable $\infty$-category $\dMod(R)$. This is precisely the homotopical analogue of the condition that $x$ acts \emph{topologically nilpotently} on $M$, i.e.
$\bigcap_{n \in \Nb} I^n M = 0.$
The resulting full subcategory of $I$-complete objects in $\dMod(R)$ will be denoted by $\dMod(R)^{I\text{-comp}}.$ Note that in \cite{DAGXII} this subcategory is characterized differently (cf. \cite[Cor. 4.2.8 \& 4.2.12]{DAGXII}). For abstract reasons, the inclusion
$\dMod(R)^{I\text{-comp}} \subset \dMod(R)$
possesses a left adjoint (see \cite[Lemma 4.2.2]{DAGXII})
$$\widehat{(-)}^{\der}\colon\dMod(R) \to \dMod(R)^{I\text{-comp}},$$
which will be referred to as \emph{derived completion}. By Remark 4.2.6 in \emph{loc. cit.} this is moreover a symmetric monoidal functor, hence we obtain an $E_{\infty}$-ring spectrum $\widehat{R}_I^{\der}$; the \emph{derived completion of $R$ at $I$}.

\subsection{Modification of Stable $\infty$-Categories}\label{sub:modification}

Let $X$ be a scheme, $Z$ a closed subscheme, which is defined by a locally finitely-generated sheaf of ideals. The aforementioned derived completion operation allows one to define the derived formal scheme $\widehat{X}^{\der}_Z$ (see \cite[Def. 5.1.1]{DAGXII}). If $X$ is Noetherian, it is canonically equivalent to the formal completion $\widehat{X}_Z$. We denote by  $U$ the open complement $X \setminus Z$. Recall that $\QC(X)$ denotes the stable $\infty$-category of complexes of quasi-coherent sheaves on $X$. Pullback along the open immersion $j\colon U \to X$ induces a localization
$$j^*\colon\QC(X) \to \QC(U).$$
The kernel, i.e. the full subcategory of complexes $\F$ satisfying $j^*\F \simeq 0$, will be denoted by $\QC_Z(X)$. Since $j^*\F \simeq 0$ amounts to $\F|_U \simeq 0$, it is sensible to refer to such a complex of sheaves $\F$ as having \emph{set-theoretic support contained in $|Z|$}.

The $\infty$-category of compact objects in $\QC_Z(X)$ is given by $\Perf_Z(X)$, i.e. perfect complexes on $X$ with set-theoretic support contained in $|Z|$. Moreover, $\QC_Z(X)$ is \emph{compactly generated}, amounting to the relation
$\QC_Z(X) \cong \Ind \Perf_Z(X).$

Besides passing to open subschemes (\emph{localization} in terms of stable $\infty$-categories), and restricting set-theoretic support (\emph{localizing subcategories}), a third geometrically relevant operation is given by considering complexes of sheaves on the formal completion $\widehat{X}^{\der}_{Z}$.

Quasi-coherent sheaves on the formal completion $\widehat{X}_Z$ are closely related to the $\infty$-category $\QC_Z(X)$. In fact, we have an agreement of the full subcategories of almost connective complexes (\cite[Thm. 5.1.9]{DAGXII})
\begin{equation}\label{eqn:algebraizable_perfect}\QC(\widehat{X}_Z^{\der})^{\mathrm{acn}} \cong \QC_Z(X)^{\mathrm{acn}} \cong (\Ind \Perf_Z(X))^{\mathrm{acn}}.\end{equation}

Our main interest lies in the category of perfect complexes $\Perf(\widehat{X}^{\der}_Z)$ on $\widehat{X}^{\der}_Z$. Unlike the case of a scheme, it is not sufficient to consider the full subcategory of compact objects in $\QC(\widehat{X}^{\der}_Z)$ (denoted by upper script ``c''). As we have seen earlier, $\QC(\widehat{X}^{\der}_Z)^c \cong \QC_Z(X)^c \cong \Perf_Z(X)$ only yields perfect complexes with set-theoretic support contained in $|Z|$. In fact it is not very difficult to verify that structure sheaf $\Oo$ on the formal scheme $\mathrm{Spf}\;k[[t]]$ is not compact.

In the remainder of this subsection we will use the observations described here to develop categorical analogues of the geometric operations given by the removal of closed subschemes and completion.

\subsubsection{Completion}\label{subsub:completion}

Let $\C$ be an idempotent complete stable $\infty$-category, with a full stable subcategory $\Sf$, which is idempotent complete. We refer to such an $\Sf$ simply as \emph{localizing} subcategory of $\C$. Inspired by \eqref{eqn:algebraizable_perfect} we make the following definition for the \emph{completion of $\C$ at $\Sf$}. Proposition \ref{prop:affine_completion} below compares this definition with the derived completion of rings.

\begin{definition}\label{defi:completion}
The completion $\C_{\widehat{\Sf}}$ is defined to be the idempotent closure of the \emph{essential image}
$$\image[\C \to \Ind \Sf]$$
of the functor sending $\Gg \in \C$ to the presheaf\footnote{Recall that $\Ind(\C)$ can be realized as the $\infty$-category of limit-preserving functors $\C^{\op} \to \Sp$.} $\F \mapsto \Hom(\F,\Gg)$.
\end{definition}

Note that, because the inclusion $\Sf \subset \C$ preserves finite colimits by assumption, the presheaf associated to $\F \in \C$ preserves colimits as well, and thus yields a well-defined functor $\C \to \Ind \Sf$.

Just like in Efimov's \cite[p. 8]{Efimov:2010fk}, we think of $\C_{\widehat{\Sf}}$ as a completion on the level of \emph{Hom-spaces}, not altering the class of objects. The result below can be also found in \cite[Remark 5.3]{Efimov:2010fk} for Noetherian rings.

\begin{proposition}\label{prop:affine_completion}
If $\C \cong \Perf(R)$, where $R$ is a ring, and $\Sf = \Perf_{V(I)}(X)$ for some ideal $I \subset R$, then $\Perf(R)_{\widehat{\Sf}} \cong \Perf(\widehat{R}^{\der}_I)$.\end{proposition}
\begin{proof}
In the following we denote by $V(I) \subset \Spec R$ the closed subset corresponding to the ideal $I$. We begin the proof by connecting the derived formal completion $\widehat{R}_I^{\der}$ of Subsection \ref{sub:derivedcompletion} with $\Perf_{V(I)}(R)$. Theorem 5.1.9 and Proposition 5.1.17 in \cite{DAGXII} imply the existence of a commutative diagram
\[
\xymatrix@=0.1in{
\Perf(\widehat{R}_I^{\der}) \ar@{^{(}->}[rr] & & \QC_{V(I)}(\Spec R) \\
& \Perf(R) \ar[lu] \ar[ru] &
}
\]
of $\infty$-categories. Using that $\QC_{V(I)}(\Spec R)$ is compactly generated by $\Perf_{V(I)}(R)$, and the definition of $\Perf(R)_{\widehat{\Sf}}$ as the idempotent completion of the  essential image of the functor
$$\Perf(R) \to \Ind \Perf_{V(I)}(R) \cong \QC_{V(I)}(\Spec R),$$ we obtain a commutative diagram
\[
\xymatrix@=0.14in{
\Perf(\widehat{R}_I^{\der}) \ar@{^{(}->}[rr] & & \Ind \Perf_{V(I)}(R) \\
\Perf(R) \ar@{->>}[rr] \ar[u] & & \Perf(R)_{\widehat{\Sf}}, \ar@{^{(}->}[u] \ar@{_{(}-->}[ull]
}
\]
where we use the universal property of idempotent completion to produce the dashed arrow, together with the essential surjectivity of the lower horizontal functor up to idempotent completion. In order to conclude the proof, it suffices to show that we have an inclusion
$\Perf(\widehat{R}_I^{\der}) \subset \Perf(R)_{\widehat{\Sf}}$
of full subcategories of $\Ind \Perf_{V(I)}(R)$. This follows from the fact that $\Perf(\widehat{R}_I^{\der})$ is compactly generated by the structure sheaf (or free module) $\Oo$, which is contained in $\Perf(R)_{\widehat{\Sf}}$ by the commuting diagram above.
\end{proof}

In the result below, we use the notion of \emph{weak proregularity}, which was introduced by Alonso--Jeremias--Lipman \cite{MR1422312} and Schenzel \cite{MR1973941}.

\begin{definition}\label{defi:proregular}
Let $R$ be a ring, and $f \in R$ an element. We denote by $K(R,f)$ the Koszul complex $[R \to^{\cdot{} f} R]$, concentrated in degrees $-1$ and $0$. For a tuple $\overline{f} = (f_0,\dots, f_n)$ we define the Koszul complex as $K(R,\overline{f}) = \bigotimes_{i=0}^nK(R,f_i)$. An ideal $I \subset R$ is said to be \emph{weakly proregular}, if there exist generators $(f_0,\dots,f_n)$, such that for all integers $k$, the inverse system of cohomology groups $\left(H^k(K(R,(f_0^i,\dots,f_n^i)))\right)_i$ is \emph{pro-zero}, i.e. equivalent to the zero object in the category of pro-abelian groups.
\end{definition}

Every ideal $I$ in a Noetherian ring is weakly proregular. Moreover, the notion of weak proregularity is evidently invariant under flat base change. Hence, if $R$ is a Noetherian $k$-algebra, and $A$ is an arbitrary $k$-algebra, then the ideal $I_A = I \otimes_k A \subset R_A = R \otimes_k A$ is weakly proregular.

\begin{proposition}\label{prop:proregular}
If $I$ is \emph{weakly proregular} in $R$ (see Definition \ref{defi:proregular}), then $\widehat{R}^{\der}_I \cong \widehat{R}_I$. In particular, we see that, for a Noetherian $k$-algebra $R$, an ideal $I$, and an arbitrary $k$-algebra $A$, we have $\widehat{(R_A)}^{\der}_{I_A} \cong \widehat{(R_A)}_{I_A}$.
\end{proposition}

\begin{proof}
To prove this assertion we cite the main result of Porta--Shaul--Yekutieli \cite[Thm. 4.2]{MR3181733}. They prove that for every perfect generator $M \in \Perf_{V(I)}(A)$, the so-called \emph{double centralizer} is equivalent to the classical formal completion $\widehat{A}$. The double centralizer of $M$ is defined as follows. First one introduces the $E_1$-algebra $B = \End_R(M)$. The $E_1$-algebra $\End_B(M)$ is by definition the double centralizer of $M$.

We relate $\widehat{R}_I^{\der}$ to the double centralizer by observing that by definition its underlying $E_1$-ring agrees with the endomorphism algebra of the image of $R$ in $\Perf(R)_{\widehat{V(I)}}$:
$$\widehat{R}_I^{\der} \cong \End_{\Perf(R)_{\widehat{V(I)}}}(R).$$
The map $\Perf(R) \to \Ind \Perf_{V(I)}(R)$ is given by sending a module $N$ to the presheaf $\Hom_R(-,N)$ on $\Perf_{V(I)}(R)$. Since $M$ is a generator, we have $\Ind \Perf_{V(I)}(R) \cong \dMod(\End_R(M)^{\op}) \cong \dMod(B^{\op})$. In particular, we see that the $R$-module $R$ is sent to
$\Hom_R(M,R) = M^{\vee} \in \dMod(B)^{\op}.$
Thus, we have
$$\End_{B^{\op}}(M^{\vee}) \cong \End_{B}(M).$$
The right hand side is by definition the double centralizer of $M$, and therefore, by \emph{loc. cit.} agrees with the classical completion $\widehat{R}$. In particular, since this is a discrete $E_1$-ring, this argument specifies the $E_{\infty}$-structure as well.
\end{proof}

Since the Yoneda embedding of $\Sf$ is fully faithful, one obtains that $\Sf$ embeds fully faithfully into the formal completion $\C_{\widehat{\Sf}}$.

\begin{definition}\label{defi:pre-iteration}
Let $\Sf \subset \Tf \subset \C$ be a chain of localizing subcategories of $\C$. Then, we denote by
\begin{enumerate}

\item[(a)] $\Tf_{\widehat{\Sf}} \subset \C_{\widehat{\Sf}}$ the localizing subcategory given by the idempotent closure of the essential image
$\image[\Tf \to \C_{\widehat{\Sf}}],$ and by

\item[(b)] $\Tf_{(\Sf)}$ the idempotent completion of the essential image
$\image[\Tf \to \C/\Sf].$

\end{enumerate}
\end{definition}

As dictated by geometric intuition, completion of $\C$ at $\Sf$, followed by completion at $\Tf$, yields an $\infty$-category equivalent to $\C_{\widehat{\Sf}}$. Similarly, the completion of $X$ at $Z$ should be canonically equivalent to the completion of $U$ at $Z$, if $U$ is any open subscheme containing $Z$. This is the content of the next lemma, see also \cite[Thm. 4.1(iii)]{Efimov:2010fk}:

\begin{lemma}\label{lemma:pre-iteration}
\begin{itemize}
\item[(a)] Using the notation of Definition \ref{defi:pre-iteration}, the natural map
$\C_{\widehat{\Sf}} \xrightarrow{\cong} (\C_{\widehat{\Sf}})_{\widehat{\Tf_{\widehat{\Sf}}}}$
is an equivalence.
\item[(b)] Let $\Sf, \Tf$ be localizing subcategories of $\C$ such that for $X \in \Sf$ and $Y \in \Tf$ we have $\Hom_{\C}(X,Y)\cong 0$. We denote by $\D$ the idempotent completion of $\C/\Sf$. Then we have $\D_{\widehat{\Tf}} \cong \C_{\widehat{\Tf}}$.
\end{itemize}
\end{lemma}

\begin{proof}
(a) By definition, the right hand side agrees with the the essential image (up to idempotent completion)
$$\image[\C_{\widehat{\Sf}} \to \Ind \Tf_{\widehat{\Sf}}]^{\ic} \cong \image[\image[\C \to \Ind \Sf]^{\ic} \to \Ind(\image[\Tf \to \Ind \Sf]^{\ic})]^{\ic}.$$
The latter is equivalent to the essential image (up to idempotent completion)
$\image[\C \to \Ind \Sf]^{\ic},$
which agrees with $\C_{\widehat{\Sf}}$ by definition.
(b) At first we want to show that for $X \in \Sf$ and an arbitrary object $Y \in \C$ we have that the natural morphism of spaces of morphisms
$\Hom_{\C}(X,Y) \to \Hom_{\C/\Tf}(X,Y)$
is an equivalence. It suffices to show this for
$\Hom_{\Ho(\C)}(X,Y) \to \Hom_{\Ho(\C/\Tf)}(X,Y)$
by virtue of Whitehead's Lemma. This is a map of abelian groups, and hence we need to verify surjectivity and injectivity. A morphism $X \to^{\bar{f}} Y$ in $\C/\Tf$ can be represented by a zigzag
$X \to Y' \leftarrow Y$,
with the right hand arrow having fibre $F$ in $\Tf$. Since we have a distinguished triangle
$$\Hom_{\C}(X,F) \to \Hom_{\C}(X,Y') \to \Hom_{\C}(X,Y) \to \Sigma\Hom_{\C}(X,F)$$
and $\Hom_{\C}(X,F) \cong 0$, since $F \in \Tf$, we see that $\Hom_{\C}(X,Y) \cong \Hom_{\C}(X,Y')$.
A similar  argument can be used to show injectivity. This shows that we have a commutative diagram of stable $\infty$-categories
\[
\xymatrix{
\C \ar[d] \ar[rd] & \\
\C/\Tf \ar[r] & \Ind \Sf.
}
\]
This implies that the essential images of the right-pointing functors agree, and therefore shows $\C_{\widehat{\Sf}} \cong (\C/\Tf)_{\widehat{\Sf}}$.
\end{proof}

\begin{definition}\label{defi:complete-remove}
For a localizing subcategory $\Sf \subset \C$ we denote by $\C_{\widehat{(\Sf)}}$ the idempotent completion of the localization $\C_{\widehat{\Sf}}/\Sf$ of $\C_{\widehat{\Sf}}$ at $\Sf$.
\end{definition}

This localization should be imagined as the $\infty$-category  of perfect complexes on a punctured formal neighbourhood. For $X = \Spec R$ an affine scheme, and $Z = V(I)$ a closed subset, let $\Sf = \Perf_Z(X) \subset \Perf(X)$. We have $\Perf(X)_{\widehat{(\Sf)}} \cong \Perf(\Spec \widehat{R}^{\der}_I \setminus V(I))$.

See Efimov's \cite[Thm. 6.1]{Efimov:2010fk} for a global analogue of the following statement.
\begin{proposition}\label{prop:efimov}
Let $X$ be an excellent reduced scheme. Then the flag of localizing subcategories $\Sf_0,\dots, \Sf_{n-1}$ induced by
$$\xi\colon X = Z_n \supset Z_{n-1} \supset \dots \supset Z_0,$$
where we define $\Sf_i = \Perf_{Z_i}(X)$,
satisfies
$$\Perf(X)_{\widehat{(0,n-1)}} \cong \Perf(F_{X,\xi}),$$
where $F_{X,\xi}$ was defined in Definition \ref{defi:HLF_complete_localize}, and we use the notation of Definition \ref{defi:iteration}. Assume moreover that $X$ is an excellent, reduced $k$-scheme, where $k$ is a field. For every commutative $k$-algebra $A$ we have a natural equivalence
$$\Perf(X_A)_{\widehat{(0,n-1)}} \to \Perf(A_{X,\xi}),$$
where $A_{X,\xi} \cong \Ab_{X_A}(\xi,\Oo_{X_A})$.
\end{proposition}

The proof will be given in the next paragraph. Reasoning inductively, we will break the lemma down into several steps of independent interest.

\subsubsection{Higher Local Fields via Categorical Completion}\label{subsub:HLFcatcomp}
Recall Proposition \ref{prop:affine_completion}:
for $R$ a ring, and an ideal $I \subset R$, the functor $\Perf(R) \to \Perf_I(\widehat{R}_I^{\der})$ induces an equivalence
\begin{equation}\label{eqn:nat_functor}
\Perf(\widehat{R}_Z^{\der}) \to \Perf(R)_{\widehat{\Perf_{V(I)}(R)}}.
\end{equation}
The following Lemma uses the notion of \emph{equiheighted} ideals, and localization at equiheighted ideals, which were discussed in Definition \ref{defi:equiheighted_localization}.

\begin{lemma}\label{lemma:pre_ind_step}
Let $R$ be an excellent reduced $k$-algebra, and $I \subset R$ a radical equiheighted ideal in $R$, of height $1$. Moreover we assume that $R$ is semi-local, i.e. that the set of maximal ideals $\Max(R)$ is finite (therefore defining a closed subset of $\Spec R$). Then, for every $k$-algebra $A$, we have a canonical equivalence of stable $\infty$-categories
$$\Perf_{V(I_A)}(\Spec ((R_A)_{I})) \cong \Perf_{V(I_A)}(\Spec R_A \setminus \Max(R)),$$
where $(R_A)_{I})$ denotes the ring obtained by localizing $R_A$ at the equiheighted ideal $I$ (see Definition \ref{defi:equiheighted_localization}).
In particular, for the flag $S_0 = \Perf_{\Max(R)_A}(R_A) \subset \Perf_{V(I_A)}(R_A)$ the equivalence
$$\Perf((R_A)_I)_{[1]} \cong \Perf(R_A)_{(0)[1]},$$
using the notation of Definition \ref{defi:iteration}.
\end{lemma}

\begin{proof}
Let $\Uu^{\text{aff}}$ be the set of affine open subsets $\Spec R_U$ of $\Spec R$, containing all minimal prime ideals above $I$ (i.e. containing the generic points of $V(I) \subset \Spec R$). Inclusion of subsets induces a partial ordering on $\Uu^{\text{aff}}$. By definition, the localization $(R_A)_I$ can be expressed as the direct limit of rings
$(R_A)_I \cong \colim_{U \in \Uu^{\text{aff}}} (R_U)_A.$
In particular, we obtain
$$\Perf((R_A)_I) \cong \colim_{U \in \Uu^{\text{aff}}} \Perf((R_U)_A).$$
The same statements are true with support condition, reading as
$$\Perf_{V(I_A)}((R_A)_I) \cong \colim_{U \in \Uu^{\text{aff}}} \Perf_{V(I_A)}((R_U)_A).$$
Let $\Uu$ be the set of all open subsets $U \subset \Spec R$, containing all minimal prime ideals above $I$. Since every open subset is a union of affine open subsets, $\Uu^{\text{aff}} \subset \Uu$ is a final directed subset. Hence we have
$$\colim_{U \in \Uu^{\text{aff}}} \Perf_{V(I_A)}((R_U)_A) \cong \colim_{U \in \Uu} \Perf_{V(I_A)}(U_A).$$
The following two obervations conclude the proof:
\begin{itemize}
\item[(i)] We have $\Spec R \setminus \Max(R) \in \Uu$.

\item[(ii)] All the transition maps in the inverse system computing $\colim_{U \in \Uu} \Perf_{V(I_A)}(U_A)$ are equivalences. In particular, we have
$\Perf_{V(I_A)}(U_A) \cong \Perf_{V(I_A)}(\Spec (R_A)_{I})$
for each $U \in \Uu$.
\end{itemize}
Assertion (i) follows right from the definition of $\Uu$: since the minimal prime ideals above $I$ are of height $1$, they cannot contain any maximal ideals. Assertion (ii) fails to hold if one does not impose the support condition. The latter ensures that, for $U\subset V \in \Uu$ with $U \cap V(I) = V\cap V(I)$, we have that restriction induces an equivalence
$\Perf_{V(I_A)}(V_A) \to \Perf_{V(I_A)}(U_A).$
Since $I$ has height $1$ in $R$, the open set $V(I) \setminus \Max(R)$ consists precisely of the generic points of $V(I)$. Therefore, every $U \in \Uu$ intersects $V(I)$ in the same open subset $V(I) \setminus \Max(R)$. As we have just seen this implies that all transition maps
$\Perf_{V(I_A)}(V_A) \to \Perf_{V(I_A)}(U_A)$
for $U,V \in \Uu$ are equivalences. The two assertions (i) and (ii) imply now that $$\Perf_{V(I_A)}(\Spec (R_A)_{I}) \cong \Perf(\Spec R_A \setminus \Max(R)).$$
The second assertion of the Lemma is merely a reformulation, using the notation introduced in Definition \ref{defi:iteration}.
\end{proof}

\begin{corollary}\label{cor:ind_step}
Let $R$ and $A$ be a $k$-algebras, where $R$ is assumed to be Noetherian. We denote by $R_A$ the tensor product $R \otimes_k A$. Let $I_1 \subset I_0 \subset R$ be a chain of equiheighted ideals, such that $I_0$ induces an ideal of height $1$ in $R/I_1$ (i.e., relative codimension is $1$). Using the notation of Definition \ref{defi:iteration}, we have a natural equivalence
$$\Perf_{V(\widehat{I}_1)}(\Spec (\widehat{R_A}_{(I_0)})_{\widehat{I}_1}) \to \Perf(\Spec R_A)_{\widehat{(I_0)}[I_1]}.$$
\end{corollary}

\begin{proof}
Using Lemma \ref{lemma:pre_ind_step} we obtain the vertical equivalence in the commutative diagram of stable $\infty$-categories below
\[
\xymatrix{
\Perf(\Spec ((\widehat{R_A})_{I_0})_{\widehat{I}_1})_{[I_1]} \ar[r]^{\cong} \ar[d]_{\cong} & \Perf(\Spec R_A)_{\widehat{(I_0)}[I_1]} \\
\Perf(\Spec (\widehat{R_A})_{I_0})_{(I_0)[I_1]}.\ar[ru] &
}
\]
According to Definition \ref{defi:iteration}, the $\infty$-category in the bottom left corner agrees with the localization
$$(\Perf_{V(I_1)}(\Spec \widehat{R_A}^{\der}_{I_0})/\Perf_{V(I_0)}(\Spec \widehat{R_A}^{\der}_{I_0}))^{\ic}.$$
Hence, Proposition \ref{prop:affine_completion} yields the diagonal functor
$$\Perf(\Spec \widehat{R_A}_{I_0})_{(I_0)[I_1]} \to  \Perf(\Spec R_A)_{\widehat{(I_0)}[I_1]}. $$
Choosing an inverse for the vertical functor (well-defined up to a contractible space of choices), we obtain the required functor $\Perf_{V(\widehat{I}_1)}(\Spec (\widehat{R_A}_{I_0})_{\widehat{I}_1}) \to \Perf(\Spec R_A)_{\widehat{(I_0)}[I_1]}$.
\end{proof}

\begin{proof}[Proof of Proposition \ref{prop:efimov}]
We only give the proof of the second assertion, i.e. for $X$ a scheme over $k$. The first assertion is proven analogously. We may assume without loss of generality that $X$ is affine, since $Z_0$ is a finite union of closed points. Thus, let $R$ be a $k$-algebra, such that $X \cong \Spec R$.

Recall from Definition \ref{defi:HLF_complete_localize} that $A_{X,\xi}$ can be obtained by iteratively completing and localizing $R$ at a chain of equiheighted ideals $I_0 \supset I_1 \supset \dots \supset I_{n-1}$, corresponding to the closed subschemes $Z_0 \subset \cdots \subset Z_{n-1}$. We will use analogous notation for the ring
$$A_{X,\xi}=(\Loc \circ \Comp)^n R_A.$$
The asserted equivalence is a special case of the more general statement
\begin{equation}\label{eqn:adele_categories}
\Perf_{(Z_{k})_A}((\Loc \circ \Comp)^k R_A) \to \Perf(X_A)_{\widehat{(0,k-1)}[k]},
\end{equation}
which will be proven inductively.
Equation \eqref{eqn:adele_categories} for $k = 0$ amounts to the definition of $\Sf_0$:
$$\Perf_{(Z_0)_A}(R_A) \cong \Perf(X_A)_{[0]} = \Sf_0.$$
This will be the anchor point of our induction. We will prove that equation \eqref{eqn:adele_categories} holds for $k = m+1$ if it holds for $k = m$. Taking Ind-objects of both sides, and considering the (idempotent completion of the) essential image of $\Perf_{(Z_{m+1})_A}(R_A) \cong \Perf_{(Z_A)_{m+1}}(X_A)$, we see that (up to idempotent completion)
$$\Perf((\Loc \circ \Comp)^{m} R_A)_{\widehat{[m]}[m+1]} \cong \image[\Perf_{(Z_{m+1})_A}(R_A) \to \Ind \Perf_{(Z_{m})_A}((\Loc \circ \Comp)^m R_A)],$$
here we use that for a stable $\infty$-category $\C$ endowed with a flag of localizing subcategories, the functor $\C_{[i]} \to \C_{\widehat{(0,i)}[i+1]}$ is essentially surjective up to idempotent completion.
We have
$$\image[\Perf(X_A)_{[m+1]} \to \Ind \Perf(X_A)_{\widehat{(0,m-1)[m]}})] \cong \Perf(X_A)_{\widehat{(0,m-1)}\widehat{[m]}[m+1]},$$
by virtue of Definition \ref{defi:completion}. Since we are completing perfect complexes on the affine scheme $\Spec (\Loc \circ \Comp)^m R_A)$, Proposition \ref{prop:affine_completion} gives rise to a canonical functor
$$\Perf(\Comp(\Loc \circ \Comp)^{m} R_A)_{[m+1]} \to \Perf((\Loc \circ \Comp)^{m} R_A)_{\widehat{[m]}[m+1]} \cong \Perf(X_A)_{\widehat{(0,m)}[m+1]},$$
which is an equivalence. Corollary \ref{cor:ind_step} yields an equivalence
$$\Perf(\Loc \circ \Comp(\Loc \circ \Comp)^{m} R_A)_{[m+1]} \cong \Perf(\Comp(\Loc \circ \Comp)^{m} R_A)_{(m)[m+1]}.$$
Pairing this with the functoriality of localizing at the $m$-th localizing subcategory, we therefore obtain an equivalence
$$\Perf((\Loc \circ \Comp)^{m+1} R_A)_{[m+1]} \to \Perf(R_A)_{\widehat{(0,m)}[m+1]},$$
of stable $\infty$-categories.
\end{proof}

By similar techniques one proves the following:

\begin{theorem}\label{thm:general_efimov}
Let $X$ be an excellent, reduced $k$-scheme of pure dimension $n$, where $k$ is a field, and $A$ a $k$-algebra.
\begin{itemize}
\item[(a)] We denote by $\Sf_j \subset \Perf(X_A)$ the localizing subcategory given by the union of the subcategories $\Perf_{Z_A}(X_A)$ with $\dim Z \leq j$. Then we have $$\Perf(\Ab_X(|X|_n^{\red},\Oo_X)\widehat{\otimes}_k A) \cong \Perf(X_A)_{\widehat{(0,n)}}.$$

\item[(b)] Let $\xi\colon X = Z_n \supset  \cdots \supset Z_{i+1} \supset Z_{i-1} \supset \cdots \supset Z_0$ be an \emph{almost saturated flag} of equiheighted closed subschemes, satisfying $\dim Z_j = j$. Let $T_{\xi} \subset |X|_n^{\red}$ be the subset of reduced chains $\eta_0 \prec  \eta_1 \prec \cdots \prec \eta_n$, such that for $j \neq i$ we have that $\eta_j$ is a generic point of $Z_j$. Then we have the equivalence $\Perf(\Ab_X(T_{\xi},\Oo_X)\widehat{\otimes}_k A) \cong \Perf(X_A)_{\widehat{(0,n)}}$.
\end{itemize}
\end{theorem}

\begin{afterword}
Since the appearance of the preprint version of this paper on the arXiv in 2014, things have not been at a standstill. Gorchinskiy and Osipov have developed an alternative approach to a higher Contou-Carr\`ere symbol in their series of articles \cite{MR3438594}, \cite{MR3353121}. Their methods are entirely different from ours. Moreover, Musicantov and Yom Din have independently derived a similar reciprocity law \cite{MuY:2014}.
\end{afterword}

\bibliographystyle{amsalpha}
\bibliography{master,ollinewbib,refs}    

\end{document}